\documentclass[a4paper,11pt]{amsart}

\usepackage{amssymb,amsmath,amsthm,mathrsfs,enumerate,graphicx, color}
\usepackage[pdfpagelabels,colorlinks,linkcolor=blue,citecolor=black,urlcolor=blue]{hyperref}
\usepackage{esint}
\newtheorem{thm}{Theorem}[section]

\newtheorem{cor}[thm]{Corollary}
\newtheorem{lem}[thm]{Lemma}
\newtheorem{prop}[thm]{Proposition}
\newtheorem{defn}[thm]{Definition}
\newtheorem{rem}[thm]{Remark}

\newcommand{\AbbsA}[1]{\lvert #1 \rvert}
\newcommand{\AbbsB}[1]{\bigl\lvert #1 \bigr\rvert}

\newcommand{\BracketB}[1]{\bigl\lbrace #1 \bigr\rbrace}
\newcommand{\BracketC}[1]{\Bigl\lbrace #1 \Bigr\rbrace}

\newcommand{\ContainB}[1]{\bigl( #1 \bigr)}
\newcommand{\ContainC}[1]{\Bigl( #1 \Bigr)}

\newcommand{\GroupingB}[1]{\bigl[ #1 \bigr]}

\newcommand{\GroupingD}[1]{\biggl[ #1 \biggr]}

\newcommand{\RR}[1]{\mathbb{R}^{#1}} 
\newcommand{\Tempered}{\mathscr{S}'(\RR{n})} 

\newcommand{\Aclass}{\mathcal{A}} 

\newcommand{\lesi}{\lesssim}

\newcommand{\dx}{d\mu(x)}
\newcommand{\dy}{d\mu(y)}
\newcommand{\rx}{\rho(x)}
\newcommand{\rxx}{\rho(x_0)}
\newcommand{\ry}{\rho(y)}
\newcommand{\supp}{\operatorname{supp}}
\newcommand{\LL}{\mathfrak{L}}
\newcommand{\f}{\frac}

\newcommand{\Om}{\Omega}

\newcommand{\vc}{\infty}

\newcommand{\rad}{\rm rad}

\title[Maximal function characterizations for new local Hardy type spaces]{Maximal function characterizations for new local Hardy type spaces on spaces of homogeneous type}         
\author{The Anh Bui}
\address{Department of Mathematics, Macquarie University, NSW 2109,
Australia}
\email{the.bui@mq.edu.au, bt\_anh80@yahoo.com}
 
\author{Xuan Thinh Duong}
\address{Department of Mathematics, Macquarie University, NSW 2109,
Australia}

\email{xuan.duong@mq.edu.au}

\author{Fu Ken Ly}

\address{Mathematics Learning Centre, University of Sydney, NSW 2006, Australia}

\email{ken.ly@sydney.edu.au}


\begin{document}

\begin{abstract}
	Let $X$ be a space of homogeneous type and let $\mathfrak{L}$ be a nonnegative self-adjoint operator on $L^2(X)$ enjoying Gaussian estimates. The main aim of this paper is twofold. Firstly, we prove (local) nontangential and radial maximal function characterizations for the local Hardy spaces associated to $\mathfrak{L}$. This gives the maximal function characterization for local Hardy spaces in the sense of Coifman and Weiss provided that $\mathfrak{L}$ satisfies certain extra conditions. Secondly we introduce  local Hardy spaces associated with a critical function $\rho$ which are motivated by the theory of Hardy spaces related to Schr\"odinger operators and of which include the local Hardy spaces of Coifman and Weiss as a special case. We then prove that these local Hardy spaces can be characterized by (local) nontangential and radial maximal functions related to $\mathfrak{L}$ and $\rho$, and by global maximal functions associated to `perturbations' of $\mathfrak{L}$. We apply our theory to obtain a number of new results on maximal characterizations for the local Hardy type spaces in various settings ranging from Schr\"odinger operators on manifolds to Schr\"odinger operators on connected and simply connected nilpotent Lie groups.
\end{abstract}
\date{}

\maketitle

\tableofcontents

\section{Introduction}\label{sec: intro}

The main aim of this article is to obtain maximal function characterizations of various local Hardy-type spaces beyond the classical local Hardy spaces on a space of homogeneous type. 

Hardy spaces, which originated in the study of boundary values of holomorphic functions, have since  proven to be highly useful in many problems in analysis and partial differential equations. See for example \cite{BGS, FS, SW} and the references therein. Part of their usefulness arises from their many characterizations. We shall highlight the ones most pertinent to our article, which are maximal  and atomic characterizations. For $0<p\le 1$, a distribution $f\in \Tempered$ belongs to the Hardy space $H^p(\RR{n})$ if any of the following occurs:
\begin{enumerate}[(i)]
	\item $\sup\limits_{0<t<\infty} |e^{-t^2\Delta}f(x)|\in L^p(\RR{n})$
	\item $\sup\limits_{0<t<\infty}\sup\limits_{|x-y|<t} |e^{-t^2\Delta}f(y)|\in L^p(\RR{n})$
	\item $f$ has a decomposition $f=\sum_{j=0}^\infty \lambda_j a_j$, with $\sum_j |\lambda_j|^p<\infty$, and each $a_j$ is an `atom' in the following sense: $a_j$ is supported in some ball $B$, $|a_j|\le |B|^{-1/p}$, and the cancellation $\int x^\beta a_j(x)\,dx=0$ holds whenever $\beta$ is a multi-index of order $|\beta|\le \lfloor n(\f{1}{p}-1)\rfloor$. When $\f{n}{n+1}<p\le 1$ then one can use atoms with $\int a_j(x)\,dx=0$. 
\end{enumerate}
The objects in (i) and (ii) are typically referred to as the radial (or vertical) and the non-tangential maximal functions respectively. If we denote the spaces arising from (i), (ii) and (iii) by $H^p_{\Delta, \rad}(\RR{n})$, $H^p_{\Delta, \max}(\RR{n})$ and $H^p_{at}(\RR{n})$ then we can describe the above characterization more succinctly as
\begin{align}\label{intro1}
	H^p_{\Delta, \rad}(\RR{n})\equiv H^p_{\Delta, \max}(\RR{n}) \equiv H^p_{at}(\RR{n})
\end{align}
for all $0<p\le 1$. 

We are interested in generalizations of \eqref{intro1} to metric spaces other than $\RR{n}$ and to operators other than the Laplacian $-\Delta$. In the first direction Coifman and Weiss \cite{CW} introduced $H^p_{at}(X)$  on a space $X$ of homogeneous type (see \eqref{doublingcondition} below) and gave versions of \eqref{intro1} under further geometric conditions on $X$. Whether something like \eqref{intro1} holds without any extra condition on $X$ is still open, but the case when $X$ has  `reverse doubling' has been solved in \cite{U, YZ2} for $p\in (p_0,1]$ with certain $p_0\in (0,1)$.

For the second direction (in generalizing the Laplacian to some other operator $\LL$) we cite the body of work in \cite{DHMMY,DL,DY,DZ1,DZ2,DZ,HLMMY,HM,JY}. The starting point here is to replace the semigroup $e^{-t^2\Delta}$ in (i) and (ii) by some other semigroup $e^{-t^2\LL}$, but one can define an adaptation of (iii) by encoding the cancellation of atoms using $\LL$ in a certain way (see \cite{HLMMY} and also Definition \ref{def: L-atom} below). One may ask to what extent \eqref{intro1} can hold in these settings? That is, when do we have
\begin{align}\label{intro2}
	H^p_{\LL, \rad}(X)\equiv H^p_{\LL, \max}(X) \equiv H^p_{\LL,at}(X) 
\end{align}
for $0<p\le 1$? It turns out this can be achieved if $\LL$ is a non-negative and self adjoint on $L^2(X)$ with  Gaussian upper bounds on the kernel of $e^{-t^2\LL}$ (see assumptions (A1) and (A2) in Section \ref{sec: local hardy operator}). This was proved in full only recently in \cite{SY} (see also \cite{SY2}). Prior to \cite{SY} the direction $H^p_{\LL, \rad}(X)\supseteq H^p_{\LL, \max}(X) \supseteq H^p_{\LL,at}(X)$ can be found in \cite{DL,JY,HLMMY}, but the reverse direction was only known for special cases of $\LL$ \cite{DZ,HLMMY,HM}. 

We would like to point out in passing that one can add the atomic space of Coifman and Weiss $H^p_{at}(X)$, to the picture in \eqref{intro2} if the semigroup $e^{-t^2\LL}$ has H\"older regularity and is conservative (see assumptions (A3) and (A4) in Section \ref{sec: local hardy} below) for $\f{n}{n+\delta}<p\le 1$ where $\delta$ is the H\"older regularity exponent. We refer the reader to Lemma 9.1 in \cite{HLMMY} and the proof of Theorem \ref{mainthm2} in the present article. This yields one answer to the question of Coifman and Weiss when $X$ may not have reverse doubling but admits the existence of an operator $\LL$ with the appropriate properties. 

\bigskip

Our paper is concerned with local versions of the above theory. Local Hardy spaces $h^p(\RR{n})$ were introduced by Goldberg \cite{Go} to address certain shortcomings of their global counterparts (a good account of this is in \cite{Go}) and have proven to be more useful for certain problems in partial differential equations. They can be defined by restricting $t$ to less than 1 in the maximal functions of (i) and (ii) above, or by restricting the cancellation requirement in (iii) to only balls whose radii are less than 1. Then the following local version of \eqref{intro1} 
\begin{align}\label{intro3}
	h^p_{\Delta, \rad}(\RR{n})\equiv h^p_{\Delta, \max}(\RR{n}) \equiv h^p_{at}(\RR{n})
\end{align}
holds for $0<p\le 1$ (see \cite{Go}).

In the first part of our article we consider an operator $\LL$ satisfying (A1) and (A2) and by an appropriate modfication of (i)-(iii) we can define the local Hardy spaces $h^p_{\LL, \rad}(X)$, $h^p_{\LL, \max}(X)$ and $h^p_{\LL,at}(X)$  (see section \ref{sect-mainresult}). We then prove a generalization of \eqref{intro2} and \eqref{intro3} to 
\begin{align}\label{intro4}
	h^p_{\LL, \rad}(X)\equiv h^p_{\LL, \max}(X) \equiv h^p_{\LL,at}(X) 
\end{align}
for $0<p\le 1$, which is the content of Theorem \ref{mainthm1}. This can be viewed as a local version of those in \cite{SY}. If one further assumes (A3) and (A4) then one can add $h^p_{at}(X)$ to picture for $\f{n}{n+\delta}<p\le 1$, which is the content of Theorem \ref{mainthm2}. We remark that the ideas in the proof of Theorem \ref{mainthm1} rely on the innovations in \cite{SY}, although some significant modifications are needed, not least of which the development of an inhomogeneous Calder\'on reproducing formula (Proposition \ref{Calderon prop}).

In the second part of our article we consider local Hardy-type spaces where the notion of `localness' may vary spatially. More precisely we replace the role of 1 in the definitions of the spaces in \eqref{intro3} and \eqref{intro4} by a positive function $\rho(x)$ (which we call a `critical radius function') that does not fluctuate too quickly in a certain sense (see \eqref{criticalfunction}). Spaces induced by such a function $\rho$ arise as spaces related to lower order perturbations of $\LL$. A model case is the Schr\"odinger operator $-\Delta+V$ where one has
\begin{align}\label{intro5}
	H^p_{-\Delta+V, \,\rad}(X)=h^p_{at,\rho}(X)
\end{align}
for certain potentials $V$ and with $\rho$ related to $V$. We wish to point out that the atomic space in \eqref{intro5} is a modification of the atomic spaces of Coifman and Weiss -- see Definition \ref{def: rho atoms}. 
The spaces in \eqref{intro5} and their identification  were originally studied   in \cite{DZ1,DZ2,DZ} for $X=\RR{n}$, while variations have since  been considered in say \cite{D2,D3,LL,YZ}. 

With these examples in mind, we are interested in developing a general framework for \eqref{intro5} on a space $X$ of homogeneous type. This was done in \cite{YZ} for $p=1$ assuming that $X$  has reverse doubling (there the term `admissible function' is used for $\rho$); however we found we could not extend their approach to $p$ below 1. Thus a key motivation for our work is to find a way to address the scale $p<1$. We emphasize that we do not assume the reverse doubling condition in the theory.

We firstly obtain a generalization of \eqref{intro3} and \eqref{intro4} in Theorem \ref{mainthm2s}:
\begin{align}\label{intro6}
	h^p_{\LL, \rad,\rho}(X)\equiv h^p_{\LL, \max,\rho}(X) \equiv h^p_{at,\rho}(X) 
\end{align}
for $\f{n}{n+\delta}<p\le 1$. Next we extend \eqref{intro5} to an operator $L$ that can be considered a perturbation of $\LL$ in a sense (encapsulated in assumptions (B1)-(B3) in Section \ref{sec: local hardy rho}) and obtain
\begin{align}\label{intro7}
H^p_{L,\rad}(X)=H^p_{L,\max}(X)=h^p_{at,\rho}(X)
\end{align}
for a suitable range of $p$. This is contained in Theorem \ref{mainthm3}. It is worth noting that the proof \eqref{intro7} relies on the theory of local Hardy spaces that we develop for \eqref{intro4}. 

We conclude this introduction with some comments on our results. Firstly we give a list of examples of our setting in Section \ref{sec: applications}. Although the list is not exhaustive, this is intended to show the variety of possible
applications and the generality of our assumptions. Secondly we remark that our setting provides a unifying way to studying the maximal function chracterization for  local Hardy type spaces related to Schr\"odinger-type operators with non-negative potentials satisfying a reverse H\"older inequality. Note that these conditions are technical conditions which exclude potentials with small negative parts. We believe that our approach is flexible enough to give maximal function chracterizations for local  Hardy type spaces with weights or local Musielak-Orlicz Hardy type spaces. We shall leave these for a future project. Thirdly our approach can be adapted to settings with reverse doubling to give maximal function characterizations in terms of certain `approximations of the identity', extending the results in \cite{YZ} for $p=1$ to $0<p\leq 1$. See Remark \ref{rem-YZ}.

The rest of the article is organized in the following manner. Section \ref{sect-mainresult} gives the statement of our main results. In Section \ref{sec: prelim} we give some preliminary material including a covering lemma,  an inhomogeneous Calder\'on reproducing formula and some estimates for critical functions and functional calculus kernels. We prove \eqref{intro3} and \eqref{intro4} in Section \ref{sec: proof local hardy}, and \eqref{intro6} and \eqref{intro7} in Section \ref{sec: proof local hardy rho}. Section \ref{sec: applications} contains examples of situations for which our setting applies, and a few of the more technical proofs are relegated to the appendix in Section \ref{sec: appendix}.

\bigskip

Throughout the paper, we always use $C$ and $c$ to denote positive constants that are independent of the main parameters involved but whose values may differ from line to line. We will write $A\lesi B$ if there is a universal constant $C$ so that $A\leq CB$ and $A\sim B$ if $A\lesi B$ and $B\lesi A$. We denote $
a\wedge b =\min\{a,b\}, a\vee b =\max\{a,b\}$.  We will repeatedly apply the inequality $e^{-x}x^\alpha\leq C(\alpha) e^{-x/2}$ for $x\geq 0$ and $ \alpha>0$ without mention. We write $B(x,r)$ to denote the ball centred at $x$ with radius $r$. By a `ball $B$' we mean the ball $B(x_B, r_B)$ with some fixed centre $x_B$ and radius $r_B$. 
 
\section{Statement of main results}\label{sect-mainresult}
Throughout the rest of this article $X$ will be a space of homogeneous type. That is, $(X,d, \mu)$ is a metric space endowed with a nonnegative Borel measure $\mu$ with the following `doubling' condition: there exists a constant $C_1>0$ such that
\begin{equation}\label{doublingcondition}
\mu(B(x,2r))\leq C_1\mu(B(x,r))
\end{equation}
for all $x\in X$ and $r>0$ and all balls $B(x,r):=\{y\in X: d(x,y)<r\}$. In this paper, we assume that $\mu(X)=\vc$.

It is not difficult to see that the condition \eqref{doublingcondition} implies that there exists a constant $n\geq 0$ so that
\begin{equation}\label{doub2}
\mu(B(x,\lambda r))\leq C_2\lambda^n \mu(B(x,r))
\end{equation}
for all $x\in X, r>0$ and $\lambda\geq 1$, and
\begin{equation}\label{doub2s}
\mu(B(x, r))\leq C_3\mu(B(y,r))\Big(1+\f{d(x,y)}{r}\Big)^n
\end{equation}
for all $x,y\in X, r>0$.

Note that the doubling condition \eqref{doub2} implies that 
\[
\f{1}{\mu(B(x,\sqrt{t}))}\exp\Big(-\f{d(x,y)^2}{ct}\Big)\lesi \f{1}{\mu(B(y,\sqrt{t}))}\exp\Big(-\f{d(x,y)^2}{c't}\Big)
\]
and
\[
\f{1}{\mu(B(x,\sqrt{t}))}\exp\Big(-\f{d(x,y)^2}{ct}\Big)\lesi \f{1}{\mu(B(x,d(x,y)))}\exp\Big(-\f{d(x,y)^2}{c't}\Big)
\]
for any $c'>c$. These two inequalities will be used frequently without mentioning.

In this paper, unless otherwise specified, for a ball $B$ we shall mean $B=B(x_B,r_B)$.

\subsection{Local Hardy spaces associated to operators}\label{sec: local hardy operator}
Let $\mathfrak{L}$ be a nonnegative self-adjoint operator on $L^2(X)$ which generates semigroups $\{e^{-t\mathfrak{L}}\}_{t>0}$. Denote by $\widetilde{p}_t(x,y)$ and $\tilde{q}_t(x,y)$ the  kernels associated with $e^{-t\mathfrak{L}}$ and $t\mathfrak{L} e^{-t\mathfrak{L}}$, respectively. 

We assume that $\LL$ satisfies the following conditions.

\begin{enumerate}
	\item[(A1)] $\LL$ is  a nonnegative self-adjoint operator on $L^2(X)$;
	\item[(A2)] The kernel $\widetilde{p}_t(x,y)$ of $e^{-t\LL}$ admits a Gaussian upper bound. That is, there exist two positive constants $C$ and  $c$ so that for all $x,y\in X$ and $t>0$,
	\begin{equation}
	\tag{GE}\label{GE}
	\displaystyle |\widetilde{p}_t(x,y)|\leq \f{C}{\mu(B(x,\sqrt{t}))}\exp\Big(-\f{d(x,y)^2}{ct}\Big).
	\end{equation}
\end{enumerate}

We now give a definition of the (local) atomic Hardy spaces associated to operators for $0<p\leq 1$. Note that the particular case $p=1$ was investigated in \cite{GLY}. 
\begin{defn}\label{def: L-atom}
	Let $p\in (0,1]$, $q\in [1,\vc]\cap (p,\vc]$ and $M\in \mathbb{N}$. A function $a$ supported in a ball $B$ is called a  (local) $(p,q,M)_\LL$-atom if $\|a\|_{L^q(X)}\leq \mu(B)^{1/q-1/p}$ and either
	\begin{enumerate}[{\rm (a)}]
		\item $r_B\geq 1$;  or
		\item $r_B<1$ and  if there exists a
		function $b\in {\mathcal D}(\LL^M)$ such that
		\begin{enumerate}[{\rm (i)}]
			\item  $a=\LL^M b$;
		    \item $\supp\LL^{k}b\subset B, \ k=0, 1, \dots, M$;
		    \item $\|(r_B^2\LL)^{k}b\|_{L^q(X)}\leq
		r_B^{2M}\mu(B)^{\f{1}{q}-\f{1}{p}},\ k=0,1,\dots,M$.
			\end{enumerate}
	\end{enumerate}
\end{defn}
It is obvious that the atoms in (a) does not depend on $\LL$ and $M$ but for the sake of convenience we shall abuse notation and use $(p,q,M)_\LL$ to refer to atoms in both (a) and (b) of Definition \ref{def: L-atom}.\\

Next we define the atomic Hardy space $h^{p,q}_{\LL,at,M}(X)$.
\begin{defn}
	Given $p\in (\f{n}{n+1},1]$, $q\in [1,\vc]\cap (p,\vc]$ and $M\in \mathbb{N}$, we  say that $f=\sum
	\lambda_ja_j$ is an (local) atomic $(p,q,M)_\LL$-representation if
	$\{\lambda_j\}_{j=0}^\infty\in l^p$, each $a_j$ is a (local) $(p,q,M)_\LL$-atom,
	and the sum converges in $L^2(X)$. The space $h^{p,q}_{\LL,at,M}(X)$ is then defined as the completion of
	\[
	\left\{f\in L^2(X):f \ \text{has an atomic
		$(p,q,M)_\LL$-representation}\right\},
	\]
	with the norm given by
	$$
	\|f\|_{h^{p,q}_{\LL,at,M}(X)}=\inf\left\{\left(\sum|\lambda_j|^p\right)^{1/p}:
	f=\sum \lambda_ja_j \ \text{is an atomic $(p,q,M)_\LL$-representation}\right\}.
	$$

\end{defn}
For $f\in L^2(X)$, we define the localized nontangential maximal function as
$$
f^*_{\LL}(x) = \sup_{0<t<1}\sup_{d(x,y)<t}|e^{-t^2\LL}f(y)|
$$
and the localized radial maximal function as
$$
f^+_{\LL}(x) = \sup_{0<t<1}|e^{-t^2\LL}f(x)|.
$$

The maximal Hardy spaces associated to $\mathfrak{L}$ is defined as follows.
\begin{defn}
	Given $p\in (0,1]$,  the Hardy space $h^{p}_{\LL, {\rm max}}(X)$ is defined as the completion of
	$$
	\left\{f
	\in L^2(X): f^*_{\LL} \in L^p(X)\right\},
	$$
	with the norm given by
	$$
	\|f\|_{h^{p}_{\LL, {\rm max}}(X)}=\|f^*_{\LL}\|_{L^p(X)}.
	$$
	
	Similarly, the Hardy space $h^{p}_{\LL, {\rm rad}}(X)$ is defined as the completion of
	$$
	\left\{f
	\in L^2(X): f^+_{\LL} \in L^p(X)\right\},
	$$
	with the norm given by
	$$
	\|f\|_{h^{p}_{\LL,{\rm rad}}(X)}=\|f^+_{\LL}\|_{L^p(X)}.
	$$
\end{defn}
It is obvious that $h^{p}_{\LL, {\rm max}}(X)\subset h^{p}_{\LL, {\rm rad}}(X)$ for $0<p\leq 1$. Moreover, by a similar argument to Step I in the proof of Theorem 3.5 in \cite{DHMMY}, we obtain that $h^{p,q}_{\LL,at,M}(X)\subset h^{p}_{\LL, {\rm max}}(X)$ provided $p\in (0,1]$, $q\in [1,\vc]\cap (p,\vc]$ and $M>\f{n}{2}(\f{1}{p}-1)$. Hence, the following conclusion holds true
\begin{equation}\label{inclusion of HL}
h^{p,q}_{\LL,at,M}(X)\subset h^{p}_{\LL, {\rm max}}(X)\subset h^{p}_{\LL, {\rm rad}}(X).
\end{equation}

So it is both natural and interesting to raise the question of whether the reverse inclusion of \eqref{inclusion of HL} still holds true. Our first main result is to give an affirmative answer to this question.

\begin{thm}\label{mainthm1}
	Let $\LL$ satisfy (A1) and (A2). Let $p\in (0,1]$, $q\in [1,\vc]\cap (p,\vc]$ and $M>\f{n}{2}(\f{1}{p}-1)$. Then the Hardy spaces $h^{p,q}_{\LL,at,M}(X)$, $h^{p}_{\LL, {\rm max}}(X)$ and $h^{p}_{\LL, {\rm rad}}(X)$ coincide with equivalent norms.
\end{thm}
Due to this coincidence, we shall write $h^p_{\LL}(X)$ for any $h^{p,q}_{\LL,at,M}(X)$, $h^{p}_{\LL, {\rm max}}(X)$ and $h^{p}_{\LL, {\rm rad}}(X)$ with $p\in (0,1]$, $q\in [1,\vc]\cap (p,\vc]$ and $M>\f{n}{2}(\f{1}{p}-1)$.\\


\subsection{Local Hardy spaces}\label{sec: local hardy}

The second main result is to give a maximal function characterization for the local Hardy spaces on a space of homogeneous type. Note that this was proved by Goldberg \cite{Go} in the Euclidean setting, however in spaces of homogeneous type this problem is much more difficult. This was solved by Uchiyama \cite{U} for the Hardy spaces $H^p$, but the range of $p$ seems not to be optimal. The complete solution can be found in \cite{YZ2} under the extra condition of the reverse doubling condition imposed in the underlying spaces. The second main aim of this paper is to deliver a new result on maximal function characterizations of local Hardy spaces  associated to an operator. 

For convenience we recall the notion of (local) atomic Hardy spaces \cite{CW,Go,YZ}.

\begin{defn}
	Let $p\in (\f{n}{n+1},1]$ and $q\in [1,\vc]\cap (p,\vc]$. A function $a$ is called a  $(p,q)$-atom associated to the ball $B$ if
	\begin{enumerate}[{\rm (i)}]
		\item ${\rm supp}\, a\subset B$;
		\item $\|a\|_{L^q(X)}\leq \mu(B)^{1/q-1/p}$;
		\item $\displaystyle \int a(x)d\mu(x) =0$ if $r_B\leq 1$.
	\end{enumerate}
\end{defn}

We now define the atomic Hardy space on $X$.
\begin{defn}
	Given $p\in (\f{n}{n+1},1]$ and $q\in [1,\vc]\cap (p,\vc]$, we  say that $f=\sum
	\lambda_ja_j$ is an atomic $(p,q)$-representation if
	$\{\lambda_j\}_{j=0}^\infty\in l^p$, each $a_j$ is a $(p,q)$-atom,
	and the sum converges in $L^2(X)$. The space $h^{p,q}_{at}(X)$ is then defined as the completion of
	\[
		\left\{f\in L^2(X):f \ \text{has an atomic
		$(p,q)$-representation}\right\},
	\]
	with the norm given by
	$$
	\|f\|_{h^{p,q}_{at}(X)}=\inf\left\{\left(\sum|\lambda_j|^p\right)^{1/p}:
	f=\sum \lambda_ja_j \ \text{is an atomic $(p,q)$-representation}\right\}.
	$$
	\end{defn}

Assume now that the operator $\LL$ satisfies the following two additional conditions:

\begin{enumerate}
	\item[(A3)] There is a positive constant $\delta_1>0$ so that
\begin{equation}\tag{H}\label{H}
	|\widetilde{p}_t(x,y)-\widetilde{p}_t(\overline{x},y)|\leq \f{C}{\mu(B(x,\sqrt{t}))}\Big[\f{d(x,\overline{x})}{\sqrt{t}}\Big]^{\delta_1}\exp\Big(-\f{d(x,y)^2}{ct}\Big),
	\end{equation}
	whenever $d(x,\overline{x})\leq [\sqrt{t}+d(x,y)]/2$ and $t>0$;

	\item [(A4)] For every $x\in X$,
	\begin{equation}\tag{C}\label{C}
	\displaystyle \int_X \widetilde{p}_t(x,y)\dx= 1.
	\end{equation}
\end{enumerate}

Then we have the following. 
\begin{thm}\label{mainthm2}
	Let $\LL$ satisfy (A1), (A2), (A3) and (A4). Let $p\in (\f{n}{n+\delta_1},1]$ and $q\in [1,\vc]\cap (p,\vc]$. Then the Hardy spaces $h^{p,q}_{at}(X)$, $h^{p}_{\LL, {\rm max}}(X)$ and $h^{p}_{\LL, {\rm rad}}(X)$ coincide with equivalent norms. Hence, in this case, we shall write $h^p(X)$ for any $h^{p,q}_{at}(X)$, $h^{p}_{\LL, {\rm max}}(X)$ and $h^{p}_{\LL, {\rm rad}}(X)$ with $p\in (\f{n}{n+\delta_1},1]$ and $q\in [1,\vc]\cap (p,\vc]$.
\end{thm}
As mentioned earlier, the maximal function chracterization result for local Hardy spaces was proved in \cite{YZ2} under the presence of the reverse doubling condition. Hence, the main contribution of Theorem \ref{mainthm2} is to remove the reverse doubling condition. This allows us to apply the theorem to more general settings.

As a direct consequence of Theorem \ref{mainthm1} and Theorem \ref{mainthm2}, we obtain:
\begin{cor}
	Let $\LL$ satisfy (A1), (A2), (A3) and (A4). Let $p\in (\f{n}{n+\delta_1},1]$, $q\in [1,\vc]\cap (p,\vc]$ and $M>\f{n}{2}(\f{1}{p}-1)$. Then the Hardy spaces $h^{p,q}_{at}(X)$ and $h^{p,q}_{\LL,at,M}$ coincide with equivalent norms.
\end{cor}

We note that apart from examples given in Section 6, our results can be applied to certain operators defined on an open subset of $\mathbb{R}^2$. More precisely, when $X=\Omega$ is an unbounded domain of $\mathbb{R}^n$ with smooth boundary and $L=-\Delta_D$ is the Laplace operator  on $\Om$ with Dirichlet boundary condition, then $L$ satisfies (A1) and (A2). If instead we take $L=-\Delta_N$ to be the Laplace operator  on $\Om$ with Neumann boundary condition then $L$ satisfies (A1)--(A4).

\subsection{Local Hardy spaces asscociated to critical functions}\label{sec: local hardy rho}
A  function $\rho :X\to (0,\infty)$  is called a \emph{critical function} if there exist positive constants $C$ and $k_0$ so that
\begin{equation}\label{criticalfunction}
	\ry\leq C\rx\left(1 +\f{d(x,y)}{\rx}\right)^{\f{k_0}{k_0+1}}
\end{equation}
for all $x,y\in X$.

Note that the concept of critical functions was introduced in the setting of Schr\"odinger operators on $\mathbb{R}^n$ in \cite{F} (see also \cite{Sh}) and then was extended to the spaces of homogeneous type in \cite{YZ}.

A simple example of a critical function is $\rho\equiv 1$. Moreover, one of the most important classes of the critical functions is the one involving the weights satisfying the reverse H\"older's inequality. Recall that a nonnegative locally integrable function $w$ is said to be in the reverse H\"older class $RH_q(X)$ with $q>1$ if there exists a constant $C>0$ so that
$$
\Big(\fint_B w(x)^q\dx\Big)^{1/q}\leq \fint_B w(x)\dx
$$
for all balls $B\subset X$.
Note that if $w\in RH_q(X)$ then $w$ is a Muckenhoupt weight. See \cite{ST}.

Now suppose  $V\in RH_q(X)$ for some $q>1$ and, following \cite{Sh, YZ}, set
\begin{equation}\label{gammafunction}
	\rho(x)=\sup\Big\{r>0: \f{r^2}{\mu(B(x,r))}\int_{B(x,r)}V(y)\dy\leq 1\Big\}.
\end{equation}
Then it was proved in \cite{Sh,YZ} that $\rho$ is a critical function provided $n\geq 1$ and $q>\max\{1,n/2\}$.

We now introduce new local Hardy spaces associated to critical functions $\rho$
\begin{defn}\label{def: rho atoms}
	Let $\rho$ be a critical function on $X$. Let $p\in (\f{n}{n+1},1]$, $q\in [1,\vc]\cap (p,\vc]$ and $\epsilon\in (0,1]$. A function $a$ is called a  $(p,q,\rho,\epsilon)$-atom associated to the ball $B(x_0,r)$ if
	\begin{enumerate}[{\rm (i)}]
		\item ${\rm supp}\, a\subset B(x_0,r)$;
		\item $\|a\|_{L^q(X)}\leq \mu(B(x_0,r))^{1/q-1/p}$;
		\item $\displaystyle \int a(x)d\mu(x) =0$ if $r<\epsilon \rho(x_0)/4$.
	\end{enumerate}
\end{defn}
For the sake of convenience, when $\epsilon=1$ we shall write $(p,q,\rho)$ atom instead of $(p,q,\rho,\epsilon)$-atom. 

\begin{defn}
	Let $\rho$ be a critical function on $X$. Let $p\in (\f{n}{n+1},1]$, $q\in [1,\vc]\cap (p,\vc]$ and $\epsilon\in (0,1]$. We  say that $f=\sum
	\lambda_ja_j$ is an atomic $(p,q,\rho,\epsilon)$-representation if
	$\{\lambda_j\}_{j=0}^\infty\in l^p$, each $a_j$ is a $(p,q,\rho,\epsilon)$-atom,
	and the sum converges in $L^2(X)$. The space $h^{p,q}_{at,\rho,\epsilon}(X)$ is then defined as the completion of
	\[
	\left\{f\in L^2(X):f \ \text{has an atomic
		$(p,q,\rho,\epsilon)$-representation}\right\},
	\]
	with the norm given by
	$$
	\|f\|_{h^{p,q}_{at,\rho,\epsilon}(X)}=\inf\Big\{\Big(\sum|\lambda_j|^p\Big)^{1/p}:
	f=\sum \lambda_ja_j \ \text{is an atomic $(p,q,\rho,\epsilon)$-representation}\Big\}.
	$$
\end{defn}
In the particular case $\epsilon=1$ we write $h^{p,q}_{at,\rho}(X)$ instead of $h^{p,q}_{at,\rho,\epsilon}(X)$. It is clear when $\rho\equiv  1$ (or any fixed positive constant) we have $h^{p,q}_{at,\rho}(X)\equiv h^{p,q}_{at}(X)$.

Assume that the operator $\LL$ satisfies (A1)-(A4). Let $\rho$ be a critical function on $X$. For $f\in L^2(X)$ we define 
\[
f^*_{\LL,\rho}(x)= \sup_{0<t<\rho(x)^2}\sup_{d(x,y)<t}|e^{-t^2\LL}f(y)|
\]
and
\[
f^+_{\LL,\rho}(x) = \sup_{0<t<\rho(x)^2}|e^{-t^2\LL}f(x)|
\]
for all $x\in X$.\\

The maximal Hardy spaces associated to $\mathfrak{L}$ and $\rho$ are defined as follows.
\begin{defn}
	Let $\LL$ satisfy (A1)-(A4) and let $\rho$ be a critical function on $X$. Given $p\in (0,1]$,  the Hardy space $h^p_{\LL,\max,\rho}(X)$ is defined as the completion of
	\[
		\{f
	\in L^2(X): f^*_{\LL,\rho} \in L^p(X)\}
	\]
	under the norm given by
	$$
	\|f\|_{h^p_{\LL,\max,\rho}(X)}=\|f^*_{\LL,\rho}\|_{L^p(X)}.
	$$
	Similarly, the Hardy space $h^p_{\LL,\rad,\rho}(X)$ is defined as a completion of
	\[
	\{f
	\in L^2(X): f^+_{\LL,\rho} \in L^p(X)\}
	\]
	under the norm given by
	$$
	\|f\|_{h^p_{\LL,\rad,\rho}(X)}=\|f^+_{\LL,\rho}\|_{L^p(X)}.
	$$
\end{defn}
We have the following result.
\begin{thm}\label{mainthm2s}
Let $\LL$ satisfy (A1), (A2), (A3) and (A4) and let $\rho$ be a critical function on $X$. Let $p\in (\f{n}{n+\delta_1},1]$ and $q\in [1,\vc]\cap (p,\vc]$. Then we have
\[
h^{p,q}_{at,\rho}(X)\equiv h^p_{\LL,\max,\rho}(X)\equiv h^p_{\LL,\rad,\rho}(X).
\]
\end{thm}

We now consider another nonnegative self-adjoint operator $L$ on $L^2(X)$ which acts as a pertubation of the operator $\LL$. Denote by $p_t(x,y)$ the  kernels associated with $e^{-tL}$, and $q_t(x,y)=p_t(x,y)-\widetilde{p}_t(x,y)$, where $\widetilde{p}_t(x,y)$ is a kernel of $e^{-t\LL}$. We assume the following conditions:

\begin{enumerate}[(B1)]
	\item For all $N>0$, there exist positive constants $c$ and $C$ so that
	$$
	|p_t(x,y)|\leq \f{C}{\mu(B(x,\sqrt{t}))}\exp\Big(-\f{d(x,y)^2}{ct}\Big)\Big(1+\f{\sqrt{t}}{\rho(x)}+\f{\sqrt{t}}{\rho(y)}\Big)^{-N}
	$$
	for all $x,y\in X$ and $t>0$;
	
	
	\item There is a positive constant $\delta_2>0$ so that
	$$
	|q_t(x,y)|\leq \Big(\f{\sqrt{t}}{\sqrt{t}+\rho(x)}\Big)^{\delta_2}\f{C}{\mu(B(x,\sqrt{t}))}\exp\Big(-\f{d(x,y)^2}{ct}\Big)
	$$
	for all $x,y\in X$ and $t>0$
	
	\item There are a positive constant $\delta_3>0$ so that
	\[
	|q_t(x,y)-q_t(\overline{x},y)|\leq \min\left\{\Big[\f{d(x,\overline{x})}{\rho(y)}\Big]^{\delta_3},\Big[\f{d(x,\overline{x})}{\sqrt{t}}\Big]^{\delta_3} \right\}\f{C}{\mu(B(x,\sqrt{t}))}\exp\Big(-\f{d(x,y)^2}{ct}\Big)
	\]
	whenever $d(x,\overline{x})\leq \min\{d(x,y)/4,\rho(x)\}$ and $t>0$.
\end{enumerate}
\begin{rem}
	\label{rem1}
	The assumptions (A3) and (B3) imply that
	\begin{equation}
	\label{Holder ptxy}
	|p_t(x,y)-p_t(\overline{x},y)|\leq \Big[\f{d(x,\overline{x})}{\sqrt{t}}\Big]^{\delta_3\wedge \delta_1}\f{C}{\mu(B(x,\sqrt{t}))}\exp\Big(-\f{d(x,y)^2}{ct}\Big)
	\end{equation}
	whenever $d(x,\overline{x})\leq \min\{d(x,y)/4,\rho(x)\}$ and $t>0$, where $\delta_3\wedge \delta_1=\min\{\delta_1,\delta_3\}$.
\end{rem}

	

Let $\rho$ be a critical function on $X$. For $f\in L^2(X)$ we define

$$
\mathcal{M}_{\max, L}f(x) = \sup_{t>0} \sup_{d(x,y)<t}|e^{-t^2L}f(y)|
$$
and 
\[
\mathcal{M}_{{\rm rad}, L}f(x) = \sup_{t>0}|e^{-t^2L}f(x)|
\]
for all $x\in X$.\\

The maximal Hardy spaces associated to $L$ are defined as follows.
\begin{defn}
	Given $p\in (0,1]$, the Hardy space $H^p_{L, {\rm max}}(X)$ is defined as a completion of
	$$
	\left\{f
	\in L^2(X):  \mathcal{M}_{\max, L}f\in L^p(X)\right\},
	$$
	under the norm
	$$
	\|f\|_{H^p_{L, {\rm max}}(X)}=\|\mathcal{M}_{\max, L}\|_{L^p(X)}.
	$$
	Similarly, the Hardy space $H^p_{L, \rad}(X)$ is defined as a completion of
	$$
	\left\{f
	\in L^2(X):  \mathcal{M}_{{\rm rad}, L}f\in L^p(X)\right\},
	$$
	under the norm
	$$
	\|f\|_{H^p_{L, \rad}(X)}=\|\mathcal{M}_{{\rm rad}, L}f\|_{L^p(X)}.
	$$
	
\end{defn}
\begin{thm}\label{mainthm3}
	Let $\LL$ and $L$ satisfy (A1)-(A4) and (B1)-(B3), respectively. Let $p\in (\f{n}{n+\delta_0},1]$ and $q\in [1,\vc]\cap (p,\vc]$, where $\delta_0=\min\{\delta_1,\delta_2,\delta_3\}$. Then we have
	\[
	h^{p,q}_{at,\rho}(X)\equiv H^p_{L, {\rm max}}(X)\equiv H^p_{L, \rad}(X).
	\]

\end{thm}

\section{A covering lemma, critical functions, a Calder\'on reproducing formula and some kernel estimates}\label{sec: prelim}

For a measurable subset $E\subset X$ and $f\in L^1(E)$ we denote
\[
\fint_E f d\mu =\f{1}{\mu(E)}\int_E fd\mu.
\]
We denote by $\mathcal{M}$ the Hardy-Littlewood maximal function define by
\[
\mathcal{M}f(x)=\sup_{B\ni x} \fint_B|f|d\mu
\]
where the supremum is taken over all balls $B$ containing $x$.

We will now recall  an important covering lemma from \cite{C}. The open sets described in the lemma play the role of dyadic cubes in our setting.
\begin{lem}\label{Christ'slemma} There
	exists a collection of open sets $\{Q_\tau^k\subset X: k\in
	\mathbb{Z}, \tau\in I_k\}$, where $I_k$ denotes certain (possibly
	finite) index set depending on $k$, and constants $\rho\in (0,1),
	a_0\in (0,1]$ and $C_1\in (0,\vc)$ such that
	\begin{enumerate}[(i)]
		\item $\mu(X\backslash \cup_\tau Q_\tau^k)=0$ for all $k\in
		\mathbb{Z}$;
		\item if $i\geq k$, then either $Q_\tau^i \subset Q_\beta^k$ or $Q_\tau^i \cap
		Q_\beta^k=\emptyset$;
		\item for $(k,\tau)$ and each $i<k$, there exists a unique $\tau'$
		such $Q_\tau^k\subset Q_{\tau'}^i$;
		\item the diameter ${\rm diam}\,(Q_\tau^k)\leq C_1 \rho^k$;
		\item each $Q_\tau^k$ contains certain ball $B(x_{Q_\tau^k}, a_0\rho^k)$.
	\end{enumerate}
\end{lem}

The following elementary estimate will be used frequently. Its proof is simple and we omit it.
\begin{lem}\label{lem-ele est}
	Let $\epsilon >0$. We have
	$$
	\int_X\f{1}{\mu(B(x,s))\wedge \mu(B(y,s))}\Big(1+\f{d(x,y)}{s}\Big)^{-n-\epsilon}|f(y)|\dy\lesi \mathcal{M}f(x).
	$$
	for all $x\in X$ and $s>0$.
\end{lem}

\subsection{Critical functions}

For $x\in X$, we call the ball $B(x,\rx)$ a critical ball. We now give some basic properties for the critical functions and critical balls.
\begin{lem}\label{lem-criticalfunction} Let $\rho$ be a critical function on $X$.
\begin{enumerate}[{\rm (a)}]
\item For $\lambda>0$ and $x\in X$, we have
	$$
	(1+\lambda)^{-k_0}\rx\lesi \ry\lesi (1+\lambda)^{\f{k_0}{k_0+1}}\rx \ \text{for all $y\in B(x,\lambda \rx)$}.
	$$
	
\item For all $x,y\in X$, we have $\rx +d(x,y) \approx \ry + d(x,y)$.
	
\item There exists a constant $C$ so that
	$$
	\ry\geq C[\rx] ^{1+k_0}[\ry +d(x,y)]^{-k_0}
	$$
	for all $x,y \in X$.
	
\item Let $\epsilon\in (0,1]$ and $a>0$. For any $N>0$ we have
\begin{equation}
\label{eq1-rho exp}
\exp\Big(-\f{d(x,y)^2}{a[\epsilon\rho(x)]^2}\Big)\leq c(a,N)\Big(\f{\epsilon\rho(y)}{d(x,y)}\Big)^N,
\end{equation}
and
\begin{equation}
\label{eq-rho exp}
\exp\Big(-\f{d(x,y)^2}{a[\epsilon\rho(x)]^2}\Big)\f{1}{\rho(y)}\leq \f{c(a,N)}{\rho(x)}\Big(\f{\epsilon\rho(y)}{d(x,y)}\Big)^N,
\end{equation}

for all $x,y\in X$.
\end{enumerate}
\end{lem}
\begin{proof} (a) By (\ref{criticalfunction}), we have $\ry\lesi (1+\lambda)^{\f{k_0}{k_0+1}}\rx$ for all $y\in B(x,\lambda \rx)$. It remains to prove the first inequality. Indeed, if $\ry\geq \rx$, there is nothing to prove. If $\ry\leq \rx$, by (\ref{criticalfunction}), we write
	$$
	\rx\lesi [\ry] ^{\f{1}{1+k_0}}[\ry +d(x,y)]^{\f{k_0}{k_0+1}}\lesi (1+\lambda)^{\f{k_0}{k_0+1}}[\ry] ^{\f{1}{1+k_0}}[\rx]^{\f{k_0}{k_0+1}}
	$$
	It implies that $\ry \geq (1+\lambda)^{-k_0}\rx$. This completes the proof of (a).
	
	For the proofs of (b) and (c), we refer to \cite[Lemma 2.1]{YZ}.
	
	\noindent (d) We only provide the proof of \eqref{eq-rho exp}, since the proof of \eqref{eq-rho exp} is similar and  easier.
	
	We consider two cases. 
	
	\noindent {\bf Case 1: $d(x,y)\leq \ry$.} From (c) we have $\rx\lesi \ry$. This, along with the fact that $e^{-x^2}\lesi x^{-N}$, yields \eqref{eq-rho exp}.
	
	\noindent {\bf Case 2: $d(x,y)> \ry$.} From \eqref{criticalfunction} we have	
	\[
	\rx\leq C\ry\left(\f{d(x,y)}{\ry}\right)^{\f{k_0}{k_0+1}}.
	\]

	This, in combination with inequality $e^{-x^2}\lesi x^{-N(k_0+1)-k_0}$, implies
	\[	\begin{aligned}
	\exp\Big(-\f{d(x,y)^2}{a[\epsilon\rho(x)]^2}\Big)\f{1}{\rho(y)}&\lesi \f{1}{\rho(x)}\left(\f{d(x,y)}{\ry}\right)^{\f{k_0}{k_0+1}}\Big(\f{\epsilon\rho(x)}{d(x,y)}\Big)^{N(k_0+1)+k_0} \\
	&\lesi \f{1}{\rho(x)}\left(\f{d(x,y)}{\ry}\right)^{\f{k_0}{k_0+1}}\left[\f{\epsilon\rho(y)}{d(x,y)}\left(\f{d(x,y)}{\ry}\right)^{\f{k_0}{k_0+1}}\right]^{N(k_0+1)+k_0}\\
	&\lesi \f{\epsilon^{k_0(N+1)}}{\rho(x)}\left(\f{d(x,y)}{\ry}\right)^{\f{k_0}{k_0+1}}\left[\f{\epsilon\rho(y)}{d(x,y)}\right]^{N},
	\end{aligned}\]
	which implies \eqref{eq-rho exp}.
\end{proof}

A direct consequence of Lemma \ref{lem-criticalfunction} is that whenever $B:=B(x_0,\rxx)$ is a critical ball then $\rxx\sim \rx$ for all $x\in B.$

The following result will be useful in the sequel. See Lemma 2.3 and Lemma 2.4 in \cite{YZ}.
\begin{lem}\label{coveringlemm}
	\label{Lem2: rho}
	Let $\rho$ be a critical function on $X$. Then there exists a sequence of points $\{x_\alpha\}_{\alpha\in \mathcal{I}}\subset X$ and a family of functions $\{\psi_\alpha\}_{\alpha\in \mathcal{I}}$ satisfying for some $C>0$
	\begin{enumerate}[{\rm (i)}]
		\item $\displaystyle \bigcup_{\alpha\in \mathcal{I}} B(x_\alpha, \rho(x_\alpha)) = X$.
		\item For every $\lambda \geq 1$ there exist constants $C$ and $N_1$ such that $\displaystyle \sum_{\alpha\in \mathcal{I}} \chi_{B(x_\alpha, \rho(x_\alpha))}\leq C\lambda^{N_1}$.
		\item ${\rm supp}\, \psi\subset B(x_\alpha, \rho(x_\alpha)/2)$ and $0\leq \psi_\alpha(x)\leq 1$ for all $x\in X$;
		\item $\displaystyle|\psi_\alpha(x)-\psi_\alpha(y)|\leq C d(x,y)/\rho(x_\alpha)$;
		\item $\displaystyle \sum_{\alpha\in \mathcal{I}}\psi_\alpha(x)=1$ for all $x\in X$.
	\end{enumerate}
\end{lem}

\subsection{Calder\'on reproducing formula and some kernel estimates}
{\it In this subsection, we assume that $\LL$  satisfies (A1) and (A2) only.}

Denote by $E_\LL(\lambda)$ a spectral decomposition of $\LL$. Then by spectral theory, for any bounded Borel funtion $F:[0,\vc)\to \mathbb{C}$ we can define
$$
F(\LL)=\int_0^\vc F(\lambda)dE_\LL(\lambda)
$$
as a bounded operator on $L^2(X)$. It is well-known that the kernel $K_{\cos(t\sqrt{\LL})}$ of $\cos(t\sqrt{\LL})$ satisfies 
\begin{equation}\label{finitepropagation}
{\rm supp}\,K_{\cos(t\sqrt{\LL})}\subset \{(x,y)\in X\times X:
d(x,y)\leq t\}.
\end{equation}
See for example \cite{CS}.
We have the following useful lemma.
\begin{lem}[\cite{HLMMY}]\label{lem:finite propagation}
	Let $\varphi\in C^\vc_0(\mathbb{R})$ be an even function with {\rm supp}\,$\varphi\subset (-1, 1)$ and $\int \varphi =2\pi$. Denote by $\Phi$ the Fourier transform of $\varphi$. For every $\ell\in \mathbb{N}$, set $\Phi^{(\ell)}(\xi):=\f{d^\ell}{d\xi^\ell}\Phi(\xi)$. Then for every $k,\ell\in \mathbb{N}$ and $k+\ell\in 2\mathbb{N}$, the kernel $K_{(t\sqrt{\LL})^k\Phi^{(\ell)}(t\sqrt{\LL})}$ of $(t\sqrt{\LL})^k\Phi^{(\ell)}(t\sqrt{\LL})$ satisfies 
	\begin{equation}\label{eq1-lemPsiL}
	\displaystyle
	{\rm supp}\,K_{(t\sqrt{\LL})^k\Phi^{(\ell)}(t\sqrt{\LL})}\subset \{(x,y)\in X\times X:
	d(x,y)\leq t\},
	\end{equation}
	and
	\begin{equation}\label{eq2-lemPsiL}
	|K_{(t\sqrt{\LL})^k\Phi^{(\ell)}(t\sqrt{\LL})}(x,y)|\leq \f{C}{\mu(B(x,t))}.
	\end{equation}
\end{lem}
The following inhomogeneous Calder\'on reproducing formula related to $\LL$ will be crucial for the development of our paper. 
\begin{prop}
	\label{Calderon prop}
	Let $\varphi$ be as in Lemma \ref{lem:finite propagation}. Let $\psi\in C^\vc_0(\mathbb{R})$ be an even function with {\rm supp}\,$\psi \subset (-1, 1)$ and $\int \psi=2\pi$. For every $k,j\in \mathbb{N}$, set $\Phi_{k,j}(\xi):=\xi^j \Phi^{(k)}(\xi)$ and $\Psi_{k,j}(\xi):=\xi^j\Psi^{(k)}(\xi)$, where $\Phi$ and $\Psi$ are the Fourier transforms of $\varphi$ and $\psi$, respectively. Then for each $M\in \mathbb{N}$ and $f\in L^2(X)$ there exist numbers $c(M,k)$ and $c(M,k,j)$ so that
	\begin{equation}\label{eq-Calderon}
	\begin{aligned}
	f=&\sum_{k=0}^{M+1}c(M,k)\int_0^{1/2}(t^2\LL)^M\Phi_{2k,2}(t\sqrt{\LL})\Psi^{(2M-2k+2)}(t\sqrt{\LL})f\f{dt}{t}\\
	&+\sum_{k=0}^{M}c(M,k)\int_0^{1/2}(t^2\LL)^M\Phi_{2k+1,1}(t\sqrt{\LL})\Psi_{2M-2k+1,1}(t\sqrt{\LL})f\f{dt}{t}\\
	&+\sum_{k=1}^{2M+2}\sum_{j=0}^kc(M,k,j)\Phi_{j,j}(2^{-1}\sqrt{\LL})\Psi_{k-j,k-j}(2^{-1}\sqrt{\LL})f
	\end{aligned}
	\end{equation}
	in $L^2(X)$.
\end{prop}
\begin{proof}
	
	By Lebnitz's rule we have for any $k\in \mathbb{N}$
	\begin{equation}\label{eq-Lebnitz}
	\f{d^k}{ds^k}(\Phi(s)\Psi(s))=\sum_{j=0}^kC^k_j\Phi^{(j)}(s)\Psi^{(k-j)}(s).
	\end{equation}
	
	On the other hand, by integration by parts and a straightforward calculation we obtain
	\[
	\begin{aligned}
	\int_0^{1/2} (tz)^{2M+2}&(\Phi \Psi )^{(2M+2)}(tz)\f{dt}{t}\\
	&=\sum_{k=0}^{2M+1} (-1)^{k-1}\f{(2M+2)!}{(k+1)!}\Big(\f{z}{2}\Big)^{k}(\Phi\Psi)^{(k)}\Big(\f{z}{2}\Big)+(2M+2)!\\
	&=\sum_{k=0}^{2M+1}\sum_{j=0}^k (-1)^{k-1}C^k_j\f{(2M)!}{(k+1)!}\Phi_{j,j}\Big(\f{z}{2}\Big)\Psi_{k-j,k-j}\Big(\f{z}{2}\Big)+(2M+2)!.
	\end{aligned}
	\]
	This, along with spectral theory, implies
	\[
	\begin{aligned}
	f=&\f{1}{(2M+2)!}\int_0^{1/2} (t\sqrt{\LL})^{2M+2}(\Phi\Psi)^{(2M+2)}(t\sqrt{\LL}) \f{dt}{t}\\
	&+\sum_{k=0}^{2M+1}\sum_{j=0}^kc(M,k,j)\Phi_{j,j}(2^{-1}\sqrt{\LL})\Psi_{k-j,k-j}(2^{-1}\sqrt{\LL})f.
	\end{aligned}
	\]
	Moreover, from \eqref{eq-Lebnitz} we can find that
	\[
	\begin{aligned}
	\int_0^{1/2} (t\sqrt{\LL})^{2M+2}&(\Phi\Psi)^{(2M+2)}(t\sqrt{\LL}) \f{dt}{t}\\
	=&\sum_{k=0}^{2M+2} C^{2M+2}_k\int_0^{1/2}(t^2\LL)^{M+2}\Phi^{(k)}(t\sqrt{\LL})\Psi^{(2M-k+2)}(t\sqrt{\LL})f\f{dt}{t}\\
	=&\sum_{k=0}^{M+1}c(M,k)\int_0^{1/2}(t^2\LL)^M\Phi_{2k,2}(t\sqrt{\LL})\Psi^{(2M-2k+2)}(t\sqrt{\LL})f\f{dt}{t}\\
	&+\sum_{k=0}^{M}c(M,k)\int_0^{1/2}(t^2\LL)^M\Phi_{2k+1,1}(t\sqrt{\LL})\Psi_{2M-2k+1,1}(t\sqrt{\LL})f\f{dt}{t}.
	\end{aligned}
	\]
	Taking these two estimates we obtain \eqref{eq-Calderon}.
\end{proof}

We record the following result in \cite{DKP}.
\begin{lem}
	\label{lem-DKP}
	Let $\varphi\in\mathscr{S}(\mathbb{R})$ be even function with $\varphi(0)=1$ and let $N>0$. Then there exist even functions $\phi,\psi\in \mathscr{S}(\mathbb{R})$ with $\phi(0)=1$ and $ \psi^{(\nu)}(0)=0, \nu=0,1,\ldots, N$ so that for every $f\in L^2(X)$ and every $j\in \mathbb{Z}$ we have
	$$
	f=\phi(2^{-j}\sqrt{\LL})\varphi(2^{-j}\sqrt{\LL})f+\sum_{k\geq j}\psi(2^{-k}\sqrt{\LL})[\varphi(2^{-k}\sqrt{\LL})-\varphi(2^{-k+1}\sqrt{\LL})]f \ \ \text{in $L^2(X)$}.
	$$
\end{lem}

The following results give some kernel estimates which play an important role in the proof of main results.

\begin{lem}
	\label{lem1}
	\begin{enumerate}[{\rm (a)}]
		\item Let $\varphi\in \mathscr{S}(\mathbb{R})$ be an even function. Then for any $N>0$ there exists $C$ such that 
		\begin{equation}
		\label{eq1-lema1}
		|K_{\varphi(t\sqrt{\LL})}(x,y)|\leq \f{C}{\mu(B(x,t))+\mu(B(y,t))}\Big(1+\f{d(x,y)}{t}\Big)^{-n-N},
		\end{equation}
		for all $t>0$ and $x,y\in X$.
		\item Let $\varphi_1, \varphi_2\in \mathscr{S}(\mathbb{R})$ be even functions. Then for any $N>0$ there exists $C$ such that
		\begin{equation}
		\label{eq2-lema1}
		|K_{\varphi_1(t\sqrt{\LL})\varphi_2(s\sqrt{\LL})}(x,y)|\leq C\f{1}{\mu(B(x,t))+\mu(B(y,t))}\Big(1+\f{d(x,y)}{t}\Big)^{-n-N},
		\end{equation}
		for all $t\leq s<2t$ and $x,y\in X$.
		\item Let $\varphi_1, \varphi_2\in \mathscr{S}(\mathbb{R})$ be even functions with $\varphi^{(\nu)}_2(0)=0$ for $\nu=0,1,\ldots,2\ell$ for some $\ell\in\mathbb{Z}^+$. Then for any $N>0$ there exists $C$ such that
		\begin{equation}
		\label{eq3-lema1}
		|K_{\varphi_1(t\sqrt{\LL})\varphi_2(s\sqrt{\LL})}(x,y)|\leq C\Big(\f{s}{t}\Big)^{2\ell} \f{1}{\mu(B(x,t))+\mu(B(y,t))}\Big(1+\f{d(x,y)}{t}\Big)^{-n-N},
		\end{equation}
		for all $t\geq s>0$ and $x,y\in X$.
	\end{enumerate}
\end{lem}
\begin{proof}
	\noindent (a) The estimate \eqref{eq1-lema1} was proved in \cite[Lemma 2.3]{CD} in the particular case $X=\mathbb{R}^n$ but the proof is still valid in the spaces of homogeneous type. \\

	\noindent (b) We have
	\[
	\begin{aligned}
	K_{\varphi_1(t\sqrt{\LL})\varphi_2(s\sqrt{\LL})}(x,y)=\int_XK_{\varphi_1(t\sqrt{\LL})}(x,z)K_{\varphi_2(t\sqrt{\LL})}(z,y)dz.
	\end{aligned}
	\]
	This along with (a) implies that
	\[
	\begin{aligned}
	|K_{\varphi_1(t\sqrt{\LL})\varphi_2(s\sqrt{\LL})}(x,y)|&\lesi\int_X\f{1}{\mu(B(x,t))}\Big(1+\f{d(x,z)}{t}\Big)^{-2n-N}\f{1}{\mu(B(y,s))}\Big(1+\f{d(z,y)}{s}\Big)^{-3n-N-1}dz\\
	&\lesi\int_X\f{1}{\mu(B(x,t))}\Big(1+\f{d(x,z)}{t}\Big)^{-2n-N}\f{1}{\mu(B(y,t))}\Big(1+\f{d(z,y)}{t}\Big)^{-3n-N-1}dz\\	
	&\lesi\int_X\f{1}{\mu(B(x,t))}\Big(1+\f{d(x,y)}{t}\Big)^{-2n-N}\f{1}{\mu(B(y,t))}\Big(1+\f{d(z,y)}{t}\Big)^{-n-1}dz\\
	&\lesi \f{1}{\mu(B(x,t))}\Big(1+\f{d(x,y)}{t}\Big)^{-2n-N},
	\end{aligned}
	\]
	where in the second inequality we used the fact that $s\sim t$ and in the last inequality we used Lemma \ref{lem-ele est}.
	
	This, in combination with \eqref{doub2}, gives (b).\\

	\noindent (c) Set $\psi_{1}(\lambda)=\lambda^{2\ell}\varphi_{1}(\lambda)$ and $\psi_{2}(\lambda)=\lambda^{-2\ell}\varphi_{2}(\lambda)$. It is obvious that $\psi_1,\psi_2$ are even functions and $\psi_1\in \mathscr{S}(\mathbb{R})$. Moreover, since $\varphi^{(\nu)}_2(0)=0$ for $\nu=0,1,\ldots,2\ell$, one has $\psi_2\in \mathscr{S}(\mathbb{R})$. Moreover,
	\[
	K_{\varphi_1(t\sqrt{\LL})\varphi_2(s\sqrt{\LL})}(x,y)=\Big(\f{s}{t}\Big)^{2\ell}K_{\psi_1(t\sqrt{\LL})\psi_2(s\sqrt{\LL})}(x,y).
	\]
	At this stage, arguing similarly to (b) we obtain (c).
	
\end{proof}

\section{Maximal function characterizations for local Hardy spaces related to $\LL$ }\label{sec: proof local hardy}

The bulk of this section will devoted to the proof of Theorem \ref{mainthm1}. Theorem \ref{mainthm2} will then be deduced from Theorem \ref{mainthm1} at the end of the section.

\subsection{Proof of Theorem \ref{mainthm1}}
Due to \eqref{inclusion of HL}, to prove Theorem \ref{mainthm1} it suffices to prove that 
\begin{equation}\label{eq-Hrad subset Hmax}
h^{p}_{\LL, {\rm rad}}(X)\subset h^{p}_{\LL, {\rm max}}(X) 
\end{equation}
and
\begin{equation}\label{eq-Hmax subset HL}
h^{p}_{\LL, {\rm max}}(X)\subset h^{p,q}_{\LL,at, M}(X)
\end{equation}
for all $p\in (0,1]$, $q\in [1,\vc]\cap (p,\vc]$ and $M>\f{n}{2}(\f{1}{p}-1)$.\\

In order to prove \eqref{eq-Hrad subset Hmax} we need the following auxiliary results. \\

Let $F$ be a measurable function on $X\times (0,\vc)$. For $\alpha>0$ we set 
$$
F^*_\alpha(x)=\sup_{0<t<1}\sup_{d(x,y)<\alpha t}|F(y,t)|.
$$
In the particular case $\alpha=1$, we write $F^*$ instead of $F^*_\alpha$.\\

We have the following result whose proof is similar to that of \cite[Theorem 2.3]{CT}.
\begin{lem}
	\label{lem2} For any $p>0$ and $0<\alpha_2\leq \alpha_1$, there exists $C$ depending on $n$ and $p$ so that
	\[
	\|F^*_{\alpha_1}\|_{L^p(X)}\leq C\Big(1+\f{2\alpha_1}{\alpha_2}\Big)^{n/p}\|F^*_{\alpha_2}\|_{L^p(X)}.
	\]
\end{lem}

From the lemma above we immediately imply the following result.
\begin{lem}
	\label{lem3} For any $p\in (0,1]$ and $\lambda>n/p$, there exists $C$ depending on $n$ and $p$ so that
	\[
	\Big\|\sup_{0<t<1}\sup_{y}F(y,t)\Big(1+\f{d(x,y)}{t}\Big)^{-\lambda}\Big\|_{L_x^p(X)}\leq C\|F^*\|_{L^p(X)}.
	\]
\end{lem}
\begin{proof} The proof is standard but we provide it for the sake of completeness.
	
	We have
	\[
	\begin{aligned}
		\sup_{0<t<1}\sup_{y}F(y,t)\Big(1+\f{d(x,y)}{t}\Big)^{-\lambda}&\leq F^*(x)+\sum_{k=0}^\vc \sup_{0<t<1}\sup_{2^kt\leq d(x,y)< 2^{k+1}t}F(y,t)\Big(1+\f{d(x,y)}{t}\Big)^{-\lambda}\\
		&\leq F^*(x)+\sum_{k=0}^\vc 2^{-k\lambda}F^*_{2^{k+1}}(x).
	\end{aligned}
	\]
	For $p\in (0,1]$, we then imply
	\[
	\begin{aligned}
	\Big\|\sup_{0<t<1}\sup_{y}F(y,t)\Big(1+\f{d(x,y)}{t}\Big)^{-\lambda}\Big\|^p_{L_x^p(X)}\leq \sum_{k=0}^\vc 2^{-kp\lambda}\|F^*_{2^{k}}\|^p_{L^p(X)}.
	\end{aligned}
	\]
	This, in combination with Lemma \ref{lem2}, yields that
	\[
	\begin{aligned}
	\Big\|\sup_{0<t<1}\sup_{y}F(y,t)\Big(1+\f{d(x,y)}{t}\Big)^{-\lambda}\Big\|^p_{L_x^p(X)}&\leq c_{n,p}\sum_{k=0}^\vc 2^{kn}2^{-kp\lambda}\|F^*\|^p_{L^p(X)}\\
	&\lesi \|F^*\|^p_{L^p(X)},
	\end{aligned}
	\]
	as long as $\lambda>n/p$.
\end{proof}

For any even function  $\varphi \in \mathscr{S}(\mathbb{R})$, $\alpha>0$ and $f\in L^2(X)$ we define
$$
\varphi^*_{\LL,\alpha}(f)(x)=\sup_{0<t<1}\sup_{d(x,y)<\alpha t}|\varphi(t\sqrt{\LL})f(y)|,
$$ 
and 
$$
\varphi^+_{\LL,\alpha}(f)(x)=\sup_{0<t<1}|\varphi(t\sqrt{\LL})f(x)|.
$$
As usual, we drop the index $\alpha$ as $\alpha=1$.

We now are in position to prove the following estimate.

\begin{prop}
	\label{prop1}
	Let $p\in (0,1]$. Let $\varphi_1, \varphi_2\in \mathscr{\mathbb{R}}$ be even functions with $\varphi_1(0)=1$ and $\varphi_2(0)=0$ and $\alpha_1, \alpha_2>0$. Then for every $f\in L^2(X)$ we have
	\begin{equation}
	\label{eq1-prop1}
	\|(\varphi_2)^*_{\LL,\alpha_2}f\|_{L^p(X)}\lesi \|(\varphi_1)^*_{\LL,\alpha_1}f\|_{L^p(X)}.
	\end{equation}
	As a consequence, for every even function $\varphi$ with $\varphi(0)=1$ and $\alpha>0$ we have
	\begin{equation}
	\label{eq2-prop1}
	\|\varphi^*_{\LL,\alpha}f\|_{L^p(X)}\sim \|f^*_{\LL, loc}\|_{L^p(X)}.
	\end{equation}
\end{prop}
\begin{proof} From Lemma \ref{lem2} it suffices to prove the proposition with $\alpha_1=\alpha_2=1$.\\

Fix $t\in (0,1)$ and let $j_0\in \mathbb{Z}^+$ so that $2^{-j_0+1}\leq t<2^{-j_0+2}$. According to Lemma \ref{lem-DKP} there exist even functions $\phi,\psi\in \mathscr{\mathbb{R}}$ with $\phi(0)=1$ and $\psi^{(\nu)}(0)=0$ for $\nu=0,1,\ldots,2\ell$ ($\ell$ will be determined later) so that
\[
f=\phi(2^{-j_0}\sqrt{\LL})\varphi_1(2^{-j_0}\sqrt{\LL})f+\sum_{k\geq j_0}\psi(2^{-k}\sqrt{\LL})[\varphi_1(2^{-k}\sqrt{\LL})-\varphi_1(2^{-k+1}\sqrt{\LL})]
\]
which implies
\begin{align*}
\varphi_2(t\sqrt{\LL})f&=\varphi_2(t\sqrt{\LL})\phi(2^{-j_0}\sqrt{\LL})\varphi_1(2^{-j_0}\sqrt{\LL})f\\
&\qquad +\sum_{k\geq j_0}\varphi_2(t\sqrt{\LL})\psi(2^{-k}\sqrt{\LL})[\varphi_1(2^{-k}\sqrt{\LL})-\varphi_1(2^{-k+1}\sqrt{\LL})]f.
\end{align*}
Hence,
\[
\begin{aligned}
\sup_{d(x,y)<t}|\varphi_2(t\sqrt{\LL})f(y)|&\leq \sup_{d(x,y)<t}|\varphi_2(t\sqrt{\LL})\phi(2^{-j_0}\sqrt{\LL})\varphi_1(2^{-j_0}\sqrt{\LL})f(y)|\\
&\ \ \ +\sum_{k\geq j_0}\sup_{d(x,y)<t}|\varphi_2(t\sqrt{\LL})\psi(2^{-k}\sqrt{\LL})\varphi_1(2^{-k}\sqrt{\LL})f(y)|\\
& \ \ \ +  \sum_{k\geq j_0}\sup_{d(x,y)<t}|\varphi_2(t\sqrt{\LL})\psi(2^{-k}\sqrt{\LL})\varphi_1(2^{-k+1}\sqrt{\LL})f(y)|\\
&=: I_1+I_2+I_3.
\end{aligned}
\]

Fix $\lambda>n/p$ and $N>0$. Using \eqref{eq2-lema1} we have
\[
\begin{aligned}
I_1&\lesi \sup_{d(x,y)<t}\int_X \f{1}{V(y,2^{-j_0})}\Big(1+\f{d(y,z)}{2^{-j_0}}\Big)^{-n-N-\lambda}|\varphi_1(2^{-j_0}\sqrt{\LL})f(z)|d\mu(z).
\end{aligned}
\]
Since $d(x,y)<t<2^{-j_0+2}$, we have
$$
\Big(1+\f{d(y,z)}{2^{-j_0}}\Big)^{-\lambda}\sim \Big(1+\f{d(x,z)}{2^{-j_0}}\Big)^{-\lambda}.
$$
As a consequence, we have
\begin{equation}\label{eq-I1}
\begin{aligned}
I_1&\lesi \sup_{z}\Big(1+\f{d(x,z)}{2^{-j_0}}\Big)^{-\lambda}|\varphi_1(2^{-j_0}\sqrt{\LL})f(z)| \int_X \f{1}{V(y,2^{-j_0})}\Big(1+\f{d(y,z)}{2^{-j_0}}\Big)^{-n-N}d\mu(z)\\
&\lesi \sup_{z}\Big(1+\f{d(x,z)}{2^{-j_0}}\Big)^{-\lambda}|\varphi_1(2^{-j_0}\sqrt{\LL})f(z)|\\
&\lesi \sup_{0<t<1}\sup_{z}\Big(1+\f{d(x,z)}{t}\Big)^{-\lambda}|\varphi_1(t\sqrt{\LL})f(z)|.
\end{aligned}
\end{equation}

Note that $t\geq 2^{-k}$ as $k\geq j_0$. Hence, applying \eqref{eq3-lema1} we obtain
\[
\begin{aligned}
I_2&\lesi \sum_{k\geq j_0}\sup_{d(x,y)<t}\int_X \Big(\f{2^{-k}}{t}\Big)^{2\ell}\f{1}{\mu(B(y,t))}\Big(1+\f{d(y,z)}{t}\Big)^{-n-N-\lambda}|\varphi_1(2^{-k}\sqrt{\LL})f(z)|d\mu(z)\\
&\lesi \sum_{k\geq j_0}\sup_{d(x,y)<t}\int_X 2^{-2\ell(k-j_0)}\f{1}{\mu(B(y,t))}\Big(1+\f{d(y,z)}{t}\Big)^{-n-N-\lambda}|\varphi_1(2^{-k}\sqrt{\LL})f(z)|d\mu(z),
\end{aligned}
\]
where in the last inequality we used $t\sim 2^{-j_0}$.

On the other hand, we have
$$
\Big(1+\f{d(y,z)}{t}\Big)^{-\lambda}\sim \Big(1+\f{d(x,z)}{t}\Big)^{-\lambda}\sim \Big(1+\f{d(x,z)}{2^{-j_0}}\Big)^{-\lambda}\ \ \ \text{as $d(x,y)<t$}.
$$
Hence, by Lemma \ref{lem-ele est} we have
\[
\begin{aligned}
I_2
&\lesi \sum_{k\geq j_0} \int_X 2^{-2\ell(k-j_0)}\f{1}{\mu(B(y,t))}\Big(1+\f{d(y,z)}{t}\Big)^{-n-N}\Big(1+\f{d(x,z)}{2^{-j_0}}\Big)^{-\lambda}|\varphi_1(2^{-k}\sqrt{\LL})f(z)|d\mu(z)\\
&\lesi \sum_{k\geq j_0} \int_X 2^{-(2\ell-\lambda)(k-j_0)}\f{1}{\mu(B(y,t))}\Big(1+\f{d(y,z)}{t}\Big)^{-n-N}\Big(1+\f{d(x,z)}{2^{-k}}\Big)^{-\lambda}|\varphi_1(2^{-k}\sqrt{\LL})f(z)|d\mu(z)\\
&\lesi \sum_{k\geq j_0} 2^{-(2\ell-\lambda)(k-j_0)}\sup_{z}\Big(1+\f{d(x,z)}{2^{-k}}\Big)^{-\lambda}|\varphi_1(2^{-k}\sqrt{\LL})f(z)|.
\end{aligned}
\]
We now choose $\ell>\lambda/2$. Then from the inequality above we arrive at
\begin{equation}\label{eq-I2}
I_2\lesi \sup_{0<t<1}\sup_{z}\Big(1+\f{d(x,z)}{t}\Big)^{-\lambda}|\varphi_1(t\sqrt{\LL})f(z)|
\end{equation}
Similarly,
\begin{equation}\label{eq-I3}
I_3\lesi \sup_{0<t<1}\sup_{z}\Big(1+\f{d(x,z)}{t}\Big)^{-\lambda}|\varphi_1(t\sqrt{\LL})f(z)|.
\end{equation}
Taking these three estimates \eqref{eq-I1}, \eqref{eq-I2} and \eqref{eq-I3} into account and then applying Lemma \ref{lem3} we get \eqref{eq1-prop1} as desired.

\noindent

To prove \eqref{eq2-prop1}, we apply \eqref{eq1-prop1} for $\varphi_1(\lambda)=\varphi(\lambda)-e^{-\lambda^2}$, $\varphi_2(\lambda)=e^{-\lambda^2}$, $\alpha_1=
\alpha$ and $\alpha_2=1$ to obtain
$$
\Big\|\sup_{0<t<1}\sup_{d(x,y)<\alpha t}|\varphi(t\sqrt{\LL})f(y)-e^{-t^2\LL}f(y)|\Big\|_{L^p_x(X)}\lesi \|f^*_{\LL}\|_{L^p(X)}.
$$
This, along with Lemma \ref{lem2}, yields 
$$
\|\varphi^*_{\LL,\alpha}f\|_{L^p(X)}\lesi \|f^*_{\LL}\|_{L^p(X)}.
$$
Similarly, we obtain
 $$
 \|f^*_{\LL}\|_{L^p(X)}\lesi \|\varphi^*_{\LL,\alpha}f\|_{L^p(X)}.
 $$
 This proves \eqref{eq2-prop1}.
\end{proof}

For each $N>0$ and each even function $\varphi\in\mathscr{S}(\mathbb{R})$ we define
$$
M^*_{\LL,\varphi,N}f(x)=\sup_{0<t<1}\sup_{y\in X}\f{|\varphi(t\sqrt{\LL})f(y)|}{\Big(1+\f{d(x,y)}{t}\Big)^N},
$$
for each $f\in L^2(X)$.

Obviously, we have $\varphi^*_\LL f(x)\leq M^*_{\LL,\varphi,N}f(x)$ for all $x\in X, N>0$ and even functions $\varphi\in\mathscr{S}(\mathbb{R})$.

The inclusion \eqref{eq-Hrad subset Hmax} follows immediately from the following result.
\begin{prop}
	\label{prop2}
	Let $p\in(0,1]$. Let $\varphi\in\mathscr{S}(\mathbb{R})$ be an even function with $\varphi(0)=1$. Then we have, for every $f\in L^2(X)$,
	\begin{equation}\label{eq-petree}
	\Big\|M^*_{\LL,\varphi,N}f\Big\|_{L^p(X)}\lesi \|\varphi^+_{\LL}f\|_{L^p(X)},
	\end{equation}
provided $N>n/p$.

As a consequence, we have
\begin{equation*}
\Big\|\varphi^*_{\LL}f\Big\|_{L^p(X)}\lesi \|\varphi^+_{\LL}f\|_{L^p(X)},
\end{equation*}
\end{prop}
\begin{proof}
We fix $0<\theta<p$ and $\ell\in \mathbb{N}$ so that $N>n/\theta$  and $\ell>N/2$. Fix $t\in (0,1)$ and let $j_0\in \mathbb{Z}^+$ so that $2^{-j_0+1}\leq t<2^{-j_0+2}$. According to Lemma \ref{lem-DKP} there exist even functions $\phi,\psi\in \mathscr{\mathbb{R}}$ with $\phi(0)=1$ and $\psi^{(\nu)}(0)=0$ for $\nu=0,1,\ldots,2\ell$ so that
\[
\varphi(t\sqrt{\LL})f=\varphi(t\sqrt{\LL})\phi(2^{-j_0}\sqrt{\LL})\varphi(2^{-j_0}\sqrt{\LL})f+\sum_{k\geq j_0}\varphi(t\sqrt{\LL})\psi(2^{-k}\sqrt{\LL})[\varphi(2^{-k}\sqrt{\LL})-\varphi(2^{-k+1}\sqrt{\LL})]f.
\]
Hence, for any $y\in X$ we have
\begin{equation}
\label{eq-J123}
\begin{aligned}
\Big(1+\f{d(x,y)}{t}\Big)^{-N}|\varphi(t\sqrt{\LL})f(y)|&\leq \Big(1+\f{d(x,y)}{t}\Big)^{-N}|\varphi(t\sqrt{\LL})\phi(2^{-j_0}\sqrt{\LL})\varphi(2^{-j_0}\sqrt{\LL})f(y)|\\
&\ \ \ \ +\sum_{k\geq j_0}\Big(1+\f{d(x,y)}{t}\Big)^{-N}|\varphi(t\sqrt{\LL})\psi(2^{-k}\sqrt{\LL})\varphi(2^{-k}\sqrt{\LL})f(y)|\\
&\ \ \ \ +\sum_{k\geq j_0}\Big(1+\f{d(x,y)}{t}\Big)^{-N}|\varphi(t\sqrt{\LL})\psi(2^{-k}\sqrt{\LL})\varphi(2^{-k+1}\sqrt{\LL})f(y)|\\
&=:J_1+J_2+J_3.
\end{aligned}
\end{equation}
We now estimate the term $J_1$. Using \eqref{eq2-lema1} and the fact that $t\sim 2^{-j_0}$ we obtain
\begin{equation}\label{eq-J1}
\begin{aligned}
J_1&\lesi \int_X \f{1}{\mu(B(z,t))}\Big(1+\f{d(y,z)}{t}\Big)^{-N}\Big(1+\f{d(x,y)}{t}\Big)^{-N}|\varphi(t\sqrt{\LL})f(z)| d\mu(z)\\
&\lesi \int_X \f{1}{\mu(B(z,t))}\Big(1+\f{d(x,z)}{t}\Big)^{-N}|\varphi(t\sqrt{\LL})f(z)| d\mu(z)\\
&\lesi [M^*_{\LL,\varphi,N}f(x)]^{1-\theta}\times \int_X \f{1}{\mu(B(z,t))}\Big(1+\f{d(x,z)}{t}\Big)^{-N\theta}|\varphi(t\sqrt{\LL})f(z)|^{\theta} d\mu(z)\\
&\lesi [M^*_{\LL,\varphi,N}f(x)]^{1-\theta}\mathcal{M}(|\varphi^+_{\LL}f|^{\theta})(x),
\end{aligned}
\end{equation}
where we used Lemma \ref{lem-ele est} in the last inequality due to $N\theta>n$.

Since $t\geq 2^{-k}$ as $k\geq j_0$, using \eqref{eq3-lema1} we find that
\begin{equation}\label{eq1-proof prop2}
\begin{aligned}
J_2&\lesi \sum_{k\geq j_0}\int_X \Big(\f{2^{-k}}{t}\Big)^{2\ell}\f{1}{\mu(B(z,t))}\Big(1+\f{d(y,z)}{t}\Big)^{-N}\Big(1+\f{d(x,y)}{t}\Big)^{-N}|\varphi(2^{-k}\sqrt{\LL})f(z)|d\mu(z)\\
&\lesi \sum_{k\geq j_0}\int_X 2^{-2\ell(k-j_0)}\f{1}{V(z,2^{-j_0})}\Big(1+\f{d(x,z)}{2^{-j_0}}\Big)^{-N}|\varphi(2^{-k}\sqrt{\LL})f(z)|d\mu(z)\\
\end{aligned}
\end{equation}
where in the last inequality we used $t\sim 2^{-j_0}$.

Note that
$$
\Big(1+\f{d(x,z)}{2^{-j_0}}\Big)^{-N}\leq 2^{(k-j_0)N}\Big(1+\f{d(x,z)}{2^{-k}}\Big)^{-N}.
$$
Inserting this into \eqref{eq1-proof prop2}, we get that
\[
\begin{aligned}
J_2
&\lesi \sum_{k\geq j_0}2^{-(2\ell-N)(k-j_0)} \int_X \f{1}{V(z,2^{-k})}\Big(1+\f{d(x,z)}{2^{-k}}\Big)^{-N}|\varphi(2^{-k}\sqrt{\LL})f(z)|d\mu(z).
\end{aligned}
\]
Arguing similarly to \eqref{eq-J1} we obtain
\begin{equation}
\label{eq-J2}
\begin{aligned}
J_2&\lesi \sum_{k\geq j_0}2^{-(2\ell-N)(k-j_0)}[M^*_{\LL,\varphi,N}f(x)]^{1-\theta}\mathcal{M}(|\varphi^+_{\LL}f|^{\theta})(x)\\
&\lesi [M^*_{\LL,\varphi,N}f(x)]^{1-\theta}\mathcal{M}(|\varphi^+_{\LL}f|^{\theta})(x).
\end{aligned}
\end{equation}
Similarly,
\begin{equation}
\label{eq-J3}
\begin{aligned}
J_3&\lesi [M^*_{\LL,\varphi,N}f(x)]^{1-\theta}\mathcal{M}(|\varphi^+_{\LL}f|^{\theta})(x).
\end{aligned}
\end{equation}
Plugging the estimates $J_1$, $J_2$ and $J_3$ into \eqref{eq-J123} and then taking the supremum over $y\in X$ and $0<t<1$ we obtain
\[
M^*_{\LL,\varphi,N}f(x)\lesi [M^*_{\LL,\varphi,N}f(x)]^{1-\theta}\mathcal{M}(|\varphi^+_{\LL}f|^{\theta})(x).
\]
Hence,
\[
M^*_{\LL,\varphi,N}f(x)\lesi \left[\mathcal{M}(|\varphi^+_{\LL}f|^{\theta})(x)\right]^{\f{1}{\theta}}.
\]
Using the $L^{\f{p}{\theta}}$-boundedness of the maximal function $\mathcal{M}$ we get \eqref{eq-petree} as desired.
\end{proof}

To complete the proof of Theorem \ref{mainthm1}, we need only to show \eqref{eq-Hmax subset HL}. To do this, we need the following covering lemma in \cite{CW} (see also \cite{DKP}).

\begin{lem}
	\label{coveringlemma}
	Let $E\subset X$ be an open subset with finite measure. Then there exists a collection of balls $\{B_k:=B(x_{B_k},r_{B_k}): x_{B_k}\in E, r_{B_k}=d(x_{B_k},E^c)/2, k=0,1,\ldots\}$ so that
	\begin{enumerate}[{\rm (i)}]
		\item $\displaystyle E=\cup_k B(x_{B_k},r_{B_k})$;
		\item $\displaystyle  \{B(x_{B_k},r_{B_k}/5)\}_{k=1}^\vc$ are disjoint.
	\end{enumerate}
\end{lem}

\begin{proof}
	[Proof of \eqref{eq-Hmax subset HL}:]
	
	Since $h_{\LL,at,M}^{p,q}(X)\subset h_{\LL,at,M}^{p,\vc}(X)$ for all $p\in (0,1]$, $q\in [1,\vc]\cap (p,\vc]$ and $M>\f{n}{2}(\f{1}{p}-1)$, it suffices to prove that $h^p_{\max,\LL}\cap L^2(X)\subset h_{\LL,at,M}^{p,\vc}(X)$.
	
	Fix $f\in h^p_{\max,\LL}\cap L^2(X)$. Let $\Phi$ and $\Psi$ be functions in Proposition \ref{Calderon prop}. From Proposition \ref{Calderon prop}, for $M\in \mathbb{N}, M>\f{n}{2}(\f{1}{p}-1)$ we have
	\begin{equation}
	\label{Calderon forula}
	\begin{aligned}
	f=&\sum_{\ell=0}^{M+1}c(M,\ell)\int_0^{1/2}(t^2\LL)^M\Phi_{2\ell,2}(t\sqrt{\LL})\Psi^{(2M-2\ell+2)}(t\sqrt{\LL})f\f{dt}{t}\\
	&+\sum_{\ell=0}^{M}c(M,\ell)\int_0^{1/2}(t^2\LL)^M\Phi_{2\ell+1,1}(t\sqrt{\LL})\Psi_{2M-2\ell+1,1}(t\sqrt{\LL})f\f{dt}{t}\\
	&+\sum_{\ell=1}^{2M+2}\sum_{j=0}^\ell c(M,\ell,j)\Phi_{j,j}(2^{-1}\sqrt{\LL})\Psi_{\ell-j,\ell-j}(2^{-1}\sqrt{\LL})f\\
	=&:\sum_{\ell=0}^{M+1}f_{\ell,1} +\sum_{\ell=0}^{M}f_{\ell,2} +\sum_{\ell=1}^{2M+2}\sum_{j=0}^\ell g_{\ell,j}
	\end{aligned}
	\end{equation}
	in $L^2(X)$.	
	
	We will prove that functions $f_{\ell,1}, f_{\ell,2}$ and $g_{\ell,j}$ admit atomic $(p,\vc)$-representations.
	
	We now take care of $g_{\ell,j}$. Note that from Lemma \ref{Christ'slemma}, we can pick up a disjoint family of open sets $\{Q_k\}_{k=1}^\vc$ and $\{x_k\}_{k=1}^\vc$ so that $X=\cup_k Q_k$, $Q_k\subset B_k:=B(x_k,1/2)$ and $\mu(Q_k)\sim \mu(B_k)$ for all $k$. For each $m,\ell,j$ we decompose
	$$
	g_{\ell,j}=\sum_k c(M,\ell,j)\Phi_{j,j}(2^{-1}\sqrt{\LL})\left[\Psi_{\ell-j,\ell-j}(2^{-1}\sqrt{\LL})f.\chi_{Q_k}\right].
	$$
	We now set 
	$$
	\lambda_k =\mu(Q_k)^{1/p}\sup_{x\in Q_k}|\Psi_{\ell-j,\ell-j}(2^{-1}\sqrt{\LL})f(x)|,
	$$
	and
	$$
	 a_k=\f{c(M,\ell,j)}{\lambda_k}\Phi_{j,j}(2^{-1}\sqrt{\LL})\left[\Psi_{\ell-j,\ell-j}(2^{-1}\sqrt{\LL})f.\chi_{Q_k}\right].
	$$
	We then have 	$g_{\ell,j}=\sum_k\lambda_k a_k$, and
	\[
	\begin{aligned}
	|\lambda_k|^p&\leq \mu(Q_k)\inf_{x\in Q_k}\sup_{d(x,y)<1}|\Psi_{\ell-j,\ell-j}(2^{-1}\sqrt{\LL})f(y)|^p\\
	&\leq \mu(Q_k)\inf_{x\in Q_k} \left[\sup_{0<t<1}\sup_{d(x,y)<2t} |\Psi_{\ell-j,\ell-j}(t\sqrt{\LL})f(y)|\right]\\
	&\leq \int_{Q_k}|\Psi_{\ell-j,\ell-j} ^*f(x)|^pd\mu(x),
	\end{aligned}
	\]	
	where
	\[
	\Psi_{\ell-j,\ell-j} ^*f(x):=\sup_{0<t<1}\sup_{d(x,y)<2t} |\Psi_{\ell-j,\ell-j}(t\sqrt{\LL})f(y)|.
	\]
	This implies 
	\[
	\sum_k |\lambda_k|^p \leq \|\Psi_{\ell-j,\ell-j} ^*f\|^p_{L^p(X)}\lesi \|f\|^p_{h^p_{\LL,\max}(X)},
	\]
	where in the last inequality we used Lemma \ref{lem2}.
	
	It remains to show that $a_k$ is a multiple of a $(p,\vc,M)_{\LL}$ atom with a harmless constant for each $k$. Indeed, from \eqref{eq2-lemPsiL} we imply
	\[
	{\rm supp}\,a_k\subset B(x_k,1).
	\]
	Moreover, we have
	\[
	a_k(x)=
	\f{c(M,\ell,j)}{\lambda_k}\int_{Q_k}K_{\Phi_{j,j}(2^{-1}\sqrt{\LL})}(x,y)\Psi_{\ell-j,\ell-j}(2^{-1}\sqrt{\LL})f(y)d\mu(y).
	\]
	This, along with \eqref{eq1-lemPsiL} and the expression of $\lambda_k$, yields
	\[
	\begin{aligned}
	|a_k(x)|&\leq \mu(Q_k)^{-1/p}\int_{Q_k}|K_{\Phi_{j,j}(2^{-1}\sqrt{\LL})}(x,y)|d\mu(y)
	\lesi \mu(Q_k)^{-1/p}.
	\end{aligned}
	\]
	This shows $a_k$ is a multiple of a $(p,\vc,M)_{\LL}$ atom.\\
	
	We now take care of $f_{\ell,1}$. For a fixed $\ell\in \{0,1,\ldots, M+1\}$ we define
	\[
	\begin{aligned}
	\eta_{\ell}(x)=&\int_0^1(t^2x^2)^M\Phi_{2\ell,2}(tx)\Psi^{(2M-2\ell+2)}(tx)\f{dt}{t}
	=\int_0^xt^{2M+1}\Phi_{2\ell,2}(t)\Psi^{(2M-2\ell+2)}(t)dt.
	\end{aligned}
	\]
	Then $\eta_{\ell}\in \mathscr{S}(\mathbb{R})$ and $\eta_{\ell}(0)=0$ for each $\ell$.
	
	Moreover, we have, for any $a, b>0$,
\[
	\eta_{\ell}(b\sqrt{\LL})-\eta_{\ell}(a\sqrt{\LL})=\int_a^b (t^2\LL)^M\Phi_{2\ell,2}(t\sqrt{\LL})\Psi^{(2M-2\ell+2)}(t\sqrt{\LL})\f{dt}{t}.
	\]
	
	Define
	$$
	\mathbb{M}_{\LL}f(x)=\sup_{0<t<1}\sup_{d(x,y)<5t}\left[ |\eta_{\ell}(t\sqrt{\LL})f(y)| +|\Psi^{(2M-2\ell+2)}(t\sqrt{\LL})f(y)|\right].
	$$
	This along with Proposition \ref{prop1} yields
	\begin{equation}
	\label{eq1-proff mainthm1}
	\|\mathbb{M}_{\LL}f\|_{L^p(X)}\lesi \|f\|_{H^p_{\max, L}(X)}.
	\end{equation}

	The remainder of the proof is similar to that of \cite[Theorem 1.3]{SY}; hence we just sketch it here. For each $i\in \mathbb{Z}$ we set $O_i:=\{x\in X: \mathbb{M}_{\LL}f(x)>2^i\}$ and set $\widehat{O}_i:=(x,t)\in X\times (0,1): B(x,4t)\subset O$. Then we have
	$$
	X\times (0,1)= \bigcup_{i} \widehat{O}_i\backslash \widehat{O}_{i+1}=:\bigcup_{i} T_{i}.
	$$ 
	For each $O_i$ let $\{B_i^k\}_{k=1}^\vc$ be a family of balls covering $O_i$ as in Lemma \ref{coveringlemma}. For $i\in \mathbb{Z}$ and $k=0,1,\ldots$ we define
	$$
	\mathcal{R}(B_i^0):=\emptyset \ \ \ \text{and} \ \ \ \mathcal{R}(B_i^k):=\{(x,t)\in X\times (0,1): d(x,B_i^k)<t\}, \ \ \ k=1,2,\ldots
	$$
	Hence, $\widehat{O}_i\subset \cup_{k=0}^\vc \mathcal{R}(B_i^k)$. We now define, for $i\in \mathbb{Z}$ and $k=0,1,\ldots$,
	$$
	T^k_i=T_i\cap \left(\mathcal{R}(B_i^k)\backslash \cup_{j=0}^{k-1}\mathcal{R}(B_i^j)\right).
	$$
	It is obvious that $T_i^k\cap T_{j}^{l}=\emptyset$ either $i\neq j$ or $k\neq l$; moreover,
	$$
	X\times (0,1)=\bigcup_{i\in \mathbb{Z}}\bigcup_{k\in \mathbb{N}}T_{i}^k.
	$$
	We can write
		
	\[
	f_{\ell,1}=c(M,\ell)\sum_{i\in \mathbb{Z}, k\in \mathbb{N}}\int_0^{1/2}(t^2\LL)^M\Phi_{2\ell,2}(t\sqrt{\LL})\left[\Psi^{(2M-2\ell+2)}(t\sqrt{\LL})f .\chi_{T_i^k}\right]\f{dt}{t}.
	\] 
	We now define $\lambda_i^k=2^i\mu(B_i^k)^{1/p}$ and $a_i^k=\LL^Mb_k^i$ with
	\[
	b_i^k = \f{c(M,\ell)}{\lambda_i^k}\sum_{i\in \mathbb{Z}, k\in \mathbb{N}}\int_0^{1/2}t^{2M}\Phi_{2\ell,2}(t^2\sqrt{\LL})\left[\Psi^{(2M-2\ell+2)}(t\sqrt{\LL})f .\chi_{T_i^k}\right]\f{dt}{t}.
	\]

Hence, $f=\sum_{i\in \mathbb{Z}, k\in \mathbb{N}} \lambda_i^k a_i^k$ and it is not difficult to see that this series converges in $L^2(X)$. 

On the other hand, from the definition of the level set $O_i$ we obtain
\[
\begin{aligned}
\sum_{i\in \mathbb{Z}, k\in \mathbb{N}}|\lambda_i^k|^p&=\sum_{i\in \mathbb{Z}, k\in \mathbb{N}}2^{ip}\mu(B_i^k)\lesi \sum_{i\in \mathbb{Z}}2^{ip}\mu(O_i)
&\lesi \|\mathbb{M}_\LL f\|^p_{L^p(X)}
&\lesi \|f\|^p_{H^p_{\LL,\max}(X)}.
\end{aligned}
\]

It remains to prove that each $a_i^k$ is a multiple of $(p,\vc, M)_\LL$ atom with a universal constant. To see this, we observe that for $(y,t)\in T_i^k$ we then have $(y,t)\in \widehat{O}_i$ and hence $B(y,4t)\subset O_i$. This implies $d(y, O_i^c)>4t$. On the other hand, $(y,t)\in \mathcal{R}(B_i^k)$ and hence $d(y,B_i^k)<2t$. This leads to $d(y,x_{B_i^k})<t+r_{B_i^k}$. As a consequence, we have
\[
\begin{aligned}
4t&<d(y, O_i^c)\leq d(y, x_{B_i^k}) + d(x_{B_i^k},O_i^c)
&<t+r_{B_i^k} + 2r_{B_i^k},
\end{aligned}
\]
where in the last in equality we used the fact that $d(x_{B_i^k},O_i^c)=2r_{B_i^k}$.

This gives $t<r_{B_i^k}$. This along with \eqref{eq1-lemPsiL} implies that
\[
{\rm supp}\,\LL^{m}b_i^k\subset 3B_i^k, \ m=0, 1, \dots, M.
\]
Applying the argument in the proof of \cite[Theorem 1.3]{SY} mutatis mutandis we conclude that
\[
\|(r_{B_i^k}^2\LL)^{m}b_i^k\|_{L^\vc(X)}\leq
		r_{B_i^k}^{2M}\mu(B_i^k)^{-\f{1}{p}},\ m=0,1,\dots,M.\]

Similarly, we can prove that each $f_{\ell,2}$ admits a $(p,\vc,M)_\LL$-atom decomposition. This completes our proof of \eqref{eq-Hmax subset HL} and hence the proof of Theorem \ref{mainthm1} is complete.
\end{proof}

\subsection{Proof of Theorem \ref{mainthm2}}

	Since the proof of the inclusion $h^{p,q}_{at}(X)\subset h^{p}_{\LL, {\rm max}}(X)$ for $p\in (\f{n}{n+\delta_1},1]$ and $q\in [1,\vc]\cap (p,\vc]$ is standard and we will leave it to the interested reader. It remains to show that $h^{p}_{\LL, {\rm max}}(X)\subset h^{p,q}_{at}(X)$. Indeed, for $f\in h^{p}_{\LL, {\rm max}}(X)\cap L^2(X)$, from Theorem \ref{mainthm1} we can decompose $f=\sum_{j}\lambda_ja_j$ as an atomic $(p,q,M)_\LL$-representation with $M>\f{n}{2}(\f{1}{p}-1)$, where $a_j$ is a $(p,q,M)_\LL$-atom associated to a ball $B_j$ for $j\ge 1$. If $r_{B_j}\geq 1$, it is obvious that the $(p,\vc,M)_\LL$-atom $a_j$ is also a (local) $(p,q)$ atom. Otherwise, if $r_{B_j}< 1$, the argument used in Lemma 9.1 in \cite{HLMMY} shows that $\displaystyle \int a_jd\mu =0$. Hence, in this case  a $(p,q,M)_\LL$-atom $a_j$ is a (local) $(p,q)$ atom. As a consequence, $f=\sum_{j}\lambda_ja_j$ is an atomic $(p,q)$-repsentation, and hence $f\in  h^{p,q}_{at}(X)$.
	This completes the proof of Theorem \ref{mainthm2}.
	
\bigskip

As a byproduct, by a careful examination of the proof of Theorem \ref{mainthm1} we obtain the following result.
\begin{prop}
	Let $\LL$ satisfies (A1) and (A2). Let $p\in (0,1]$, $q\in [1,\vc]\cap (p,\vc]$ and $M>\f{n}{2}(\f{1}{p}-1)$. If $f\in h^p_{\LL}(X)\cap L^2(X)$ and ${\rm supp}\, f\subset B(x_0,r)$, then $f$ has an atomic $(p,q,M)_\LL$-repsentation $f=\sum_{j=1}^\vc\lambda_j a_j$  with ${\rm supp}\, a_j\subset B(x_0,r+1)$ for all $j$. 
\end{prop} 
In the proof of Theorem \ref{mainthm2} we have proved that if $\LL$ satisfies (A1)-(A4), then each $(p,q,M)_\LL$ atom is also $(p,q)$ atom. Hence, this along with 
the proposition above implies:
\begin{prop}
	\label{Hardyfunc-compactsupp}
	Let $\LL$ satisfies (A1)-(A4). Let $p\in (\f{n}{n+\delta_1},1]$. If $f\in h^p(X)\cap L^2(X)$ and ${\rm supp}\, f\subset B(x_0,r)$, then f has an atomic $(p,\vc)$-presentation $f=\sum_{j=1}^\vc\lambda_j a_j$  with ${\rm supp}\, a_j\subset B(x_0,r+1)$ for all $j$. 
	\end{prop} 

\section{Maximal function characterizations for local Hardy spaces associated to critical functions}\label{sec: proof local hardy rho}

This section is dedicated to the proof of Theorem \ref{mainthm2s} and Theorem \ref{mainthm3}.\\

We fix a family of balls $\{B_\alpha\}_{\alpha\in \mathcal{I}}$ and functions $\{\psi_\alpha\}_{\alpha\in \mathcal{I}}$ from Lemma \ref{Lem2: rho}. We then set, for each $\alpha$,
\begin{equation}\label{defIalpha}
\mathcal{I}_\alpha=\{j\in \mathcal{I}: B_j\cap B_\alpha\neq \emptyset\}.
\end{equation}
Then it follows from Lemma  \ref{Lem2: rho} and the doubling property that there exists $C>0$ so that
\begin{equation}\label{cardinalof Ialpha}
\sharp \mathcal{I}_\alpha\leq C, \ \ \text{for all $\alpha \in \mathcal{I}$}.
\end{equation}

From Lemma \ref{lem-criticalfunction}, we can see that  there exists $C_\rho$ so that if $y\in B(x,\rho(x))$ then $C_\rho^{-1}\rho(x)\leq \rho(y)\leq C_\rho\rho(x)$. We shall fix the constant $C_\rho$ and for any ball $B\subset$ we denote $B^* = 4C_\rho B$.

\begin{lem}\label{lem-Tatom}
	Let $\rho$ be a critical function on $X$. Let $p\in (\f{n}{n+1},1]$, $q\in [1,\vc]\cap (p,\vc]$ and $\epsilon\in (0,1]$. Assume that $T$ is a bounded sublinear operator on $L^2(X)$. If there exists $C$ so that 
	\[
	\|Ta\|_{L^p(X)}\leq C
	\]
	for all $(p,q,\rho,\epsilon)$ atom $a$, then $T$ can be extended to be bounded from $h^{p,q}_{at,\rho,\epsilon}(X)$ to $L^p(X)$.
\end{lem}
\begin{proof}
	The proof of the lemma is quite standard. See for example Lemma 4.1 in \cite{HM}. Hence, we omit details.
\end{proof}

We first concentrate on some localized maximal function estimates which will be useful in the proof of the main results.
\begin{lem}\label{Lem1: Hardy spaces}
	Let $\f{n}{n+\delta_1}<p\leq 1$ and $q\in (p,\vc]\cap [1,\vc]$. Then there exists $\kappa>0$ so that for any $0<\epsilon\leq 1$, we have
	$$
	\Big\|\sup_{0< t\leq [\epsilon \rho(x_\alpha)]^2}|e^{-t\LL}(f\psi_\alpha )(x)|\,\Big\|^p_{L^p(X\backslash B^*_\alpha)}\lesi \epsilon^{\kappa}\|f\|^p_{h^{p,q}_{at,\rho,\epsilon}(X)},
	$$
	for all $f\in h^{p,q}_{at,\rho,\epsilon}$ and each function $\psi_\alpha$ from Lemma \ref{Lem2: rho}.
\end{lem}
\begin{proof}
	It is obvious that
	\[
	\sup_{0< t\leq [\epsilon \rho(x_\alpha)]^2}|e^{-t\LL}(f\psi_\alpha )(x)\lesi \mathcal{M}f(x).
	\]
	Hence, from Lemma \ref{lem-Tatom} it suffices to prove that
	$$
	\Big\|\sup_{0< t\leq [\epsilon \rho(x_\alpha)]^2}|e^{-t\LL}(a\psi_\alpha)(x)|\,\Big\|_{L^p(X\backslash B^*_\alpha)}\lesi \epsilon^{\kappa},
	$$
	for all $(\rho,p,q,\epsilon)$ atoms associated to balls $B(x_0,r)$ so that $B(x_0,r)\cap B_\alpha\neq \emptyset$.
	
	To do this, we consider two cases:
	
	\textbf{Case 1: $\epsilon\rho(x_0)/4<r\leq \epsilon\rho(x_0)$}

	Using the Gaussian upper bound of $\widetilde{p}_t(x,y)$ and the fact that $d(x,y)\sim d(x,x_\alpha)$ for $x\in X\backslash B^*_\alpha$ and $y\in B_\alpha$,  we have, for $x\in X\backslash B^*_\alpha$,
	$$
	\begin{aligned}
	|e^{-t\LL}(a\psi_\alpha)(x)|
	&\lesi \int_{B_\alpha} \f{1}{\mu(B(x,\sqrt{t}))}\exp\Big(-\f{d(x,y)^2}{ct}\Big)|a(y)|d\mu(y)\\
	&\lesi \f{1}{\mu(B(x,d(x,x_\alpha)))}\exp\Big(-\f{d(x,x_\alpha)^2}{ct}\Big)\int_B|a(y)|d\mu(y)\\
	&
	\lesi \f{1}{\mu(B(x,d(x,x_\alpha)))}\exp\Big(-\f{d(x,x_\alpha)^2}{c[\epsilon\rho(x_\alpha)]^2}\Big)\mu(B(x_0,r))^{1-1/p}.
	\end{aligned}
	$$
	
	This implies that
	\begin{equation}\label{eq1-lem1Hardy}
	\begin{aligned}
	\Big\|\sup_{0<t\leq [\epsilon \rho(x_\alpha)]^2}|e^{-t\LL}(a\psi_\alpha)(x)|\,\Big\|_{L^p(X\backslash B^*_\alpha)}&\lesi e^{-c/\epsilon}\Big(\f{\mu(B_\alpha)}{\mu(B(x_0,r))}\Big)^{1/p-1}.
	\end{aligned}
	\end{equation}
	Note that since $B(x_0,r)\cap B_\alpha\neq \emptyset$ and $r\leq \rho(x_0)$, applying Lemma \ref{lem-criticalfunction} (i) we have $\rho(x_\alpha)\sim \rho(x_0)$. Hence, $\mu(B_\alpha)\sim \mu(B(x_0,\rho(x_0)))$. This together with (\ref{eq1-lem1Hardy}) yields that
	$$
	\begin{aligned}
	\Big\|\sup_{0<t\leq [\epsilon \rho(x_\alpha)]^2}|e^{-t\LL}(a\psi_\alpha)(x)|\,\Big\|_{L^p(X\backslash B^*_\alpha)}&\lesi e^{-c/\epsilon}\Big(\f{\mu(B(x_0,\rho(x_0)))}{\mu(B(x_0,r))}\Big)^{1/p-1}
	\\
	&
	\lesi e^{-c/\epsilon}\Big(\f{\rho(x_0)}{r}\Big)^{n(1/p-1)}
	\lesi e^{-c/\epsilon}.
	\end{aligned}
	$$

	\textbf{Case 2: $r\leq \epsilon \rho(x_0)/4$}
	
	Using the cancellation property of $a$ we obtain
	$$
	\begin{aligned}
	|e^{-t\LL}(a\psi_\alpha)(x)|&\lesi \int_{B(x_0,r)}|\widetilde{p}_t(x,y)-\widetilde{p}_t(x,x_0)|\, |a(y)|d\mu(y)\\
	& \ \ +\int_{B(x_0,r)}|\widetilde{p}_t(x,x_0)|\,|\psi_\alpha(y)-\psi_\alpha(x_0)|\, |a(y)|d\mu(y):=I_1(x)+I_2(x).
	\end{aligned}
	$$
	By \eqref{H} and a similar argument used in Case 1, we obtain that
	$$
	\begin{aligned}
	I_1(x)&\lesi \Big(\f{r}{\sqrt{t}}\Big)^{\delta_1}\f{1}{\mu(B(x,\sqrt{t}))}\exp\Big(-\f{d(x,x_0)^2}{ct}\Big)\mu(B(x_0,r))^{1-1/p}\\
	&\lesi \Big(\f{r}{d(x,x_0)}\Big)^{\delta_1}\f{1}{\mu(B(x_0,d(x,x_0)))}\exp\Big(-\f{d(x,x_0)^2}{ct}\Big)\mu(B(x_0,r))^{1-1/p}\\
	&\lesi \Big(\f{r}{d(x,x_0)}\Big)^{\delta_1} \f{1}{\mu(B(x_0,d(x,x_0)))}\mu(B(x_0,r))^{1-1/p}\\
	\end{aligned}
	$$
	which together with the doubling property implies that
	$$
	\begin{aligned}
	I_1(x)
	&\lesi \Big(\f{r}{d(x,x_0)}\Big)^{\delta_1+n-n/p} \f{1}{\mu(B(x_0,d(x,x_0)))} \mu(B(x_0,d(x,x_0)))^{1-1/p}\\
	&\lesi \Big(\f{r}{d(x,x_\alpha)}\Big)^{\delta_1+n-n/p} \mu(B(x_\alpha,d(x,x_\alpha)))^{-1/p}.
	\end{aligned}
	$$
	
	Similarly, by using the fact that $|\psi_\alpha(y)-\psi_\alpha(x_0)|\lesi d(y,x_0)/\rho(x_\alpha)$ we also obtain that
	$$
	I_2(x)\lesi \Big(\f{r}{d(x,x_\alpha)}\Big)^{\delta_1+n-n/p} \mu(B(x_\alpha,d(x,x_\alpha)))^{-1/p}.
	$$
	From these two estimates and the fact that $p>\f{n}{n+\delta_1}$ we deduce the desired estimate.
\end{proof}
From Lemma \ref{Lem1: Hardy spaces}, we deduce the following estimate.
\begin{cor}\label{Cor1: Hardy spaces}
	Let $\f{n}{n+\delta_1}<p\leq 1$ and $q\in (p,\vc]\cap [1,\vc]$. Then there exists $\kappa$ so that for any $0<\epsilon\leq 1$, we have
	$$
	\Big\|\sup_{0< t\leq [\epsilon \rho(x_\alpha)]^2}|e^{-t\LL}(f\psi_\alpha )(x)|\,\Big\|^p_{L^p(X\backslash B^*_\alpha)}\lesi \epsilon^{\kappa}\sum_{j\in \mathcal{I}_\alpha}\|f\psi_j\|^p_{h^{p,q}_{at,\rho,\epsilon}(X)},
	$$
	for all $f\in h^{p,q}_{at,\rho,\epsilon}(X)$ and each function $\psi_\alpha$ from Lemma \ref{Lem2: rho}, where $\mathcal{I}_\alpha$ is defined as in \eqref{defIalpha}.
\end{cor}
\begin{proof}
	We first note that
	$$
	f\psi_\alpha=\sum_{j\in \mathcal{I}_\alpha}\psi_\alpha (f\psi_j).
	$$
	Therefore,
	$$
	\Big\|\sup_{0< t\leq [\epsilon \rho(x_\alpha)]^2}|e^{-t\LL}(f\psi_\alpha )(x)|\,\Big\|^p_{L^p(X\backslash B^*_\alpha)}\leq \sum_{j\in \mathcal{I}_\alpha}\Big\|\sup_{0< t\leq [\epsilon \rho(x_\alpha)]^2}|e^{-t\LL}[\psi_\alpha(f\psi_j)](x)|\,\Big\|^p_{L^p(X\backslash B^*_\alpha)}
	$$
	which together with Lemma \ref{Lem1: Hardy spaces} implies that
	$$
	\Big\|\sup_{0< t\leq [\epsilon \rho(x_\alpha)]^2}|e^{-t\LL}(f\psi_\alpha )(x)|\,\Big\|^p_{L^p(X\backslash B^*_\alpha)}\lesi \epsilon^{\kappa}\sum_{j\in \mathcal{I}_\alpha}\|f\psi_j\|^p_{h^{p,q}_{at,\rho,\epsilon}}.
	$$
\end{proof}

For each $\epsilon\in(0,1]$ we define a sublinear operator $T_{\epsilon}$ by setting
\[
T_{\epsilon}f(x):=\sum_{\alpha\in \mathcal{I}}\sup_{0< t\leq [\epsilon \rho(x)]^2}|\psi_\alpha(x)e^{-t\LL}f(x) -e^{-t\LL}(f\psi_\alpha )(x)|.
\]

We first prove the $L^2$-boundedness of $T_{\epsilon}$. More precisely, we have the following result.
\begin{lem}
	\label{lem-Tpsi}
	For each $\epsilon\in(0,1]$, $T_{\epsilon}$ is bounded on $L^p(X)$ for all $1<p<\vc$.
\end{lem}
\begin{proof}
	Observe that
	\[
	T_{\epsilon}f(x)=\sum_{\alpha\in \mathcal{I}}\sup_{0< t\leq [\epsilon \rho(x)]^2}\Big|\int_X(\psi_\alpha(x)-\psi_\alpha(y))\widetilde{p}_t(x,y)f(y)d\mu(y)\Big|.
	\]
	From Lemma \ref{coveringlemm} we obtain
	\[
	\begin{aligned}
		\|T_{\epsilon}f\|_{L^p(X)}^p\lesi& \sum_{\beta\in \mathcal{I}}\int_{B_\beta}\Big[\sum_{\alpha\in \mathcal{I}}\sup_{0< t\leq [\epsilon \rho(x)]^2}\Big|\int_X(\psi_\alpha(x)-\psi_\alpha(y))\widetilde{p}_t(x,y)f(y)d\mu(y)\Big|\Big]^pd\mu(x)\\
		\lesi& \sum_{\beta\in \mathcal{I}}\int_{B_\beta}\Big[\sum_{\alpha\in \mathcal{I}_{\beta,1}}\sup_{0< t\leq [\epsilon \rho(x)]^2}\Big|\int_X(\psi_\alpha(x)-\psi_\alpha(y))\widetilde{p}_t(x,y)f(y)d\mu(y)\Big|\Big]^pd\mu(x)\\
		&+ \sum_{\beta\in \mathcal{I}}\int_{B_\beta}\Big[\sum_{\alpha\in \mathcal{I}_{\beta,2}}\sup_{0< t\leq [\epsilon \rho(x)]^2}\Big|\int_X(\psi_\alpha(x)-\psi_\alpha(y))\widetilde{p}_t(x,y)f(y)d\mu(y)\Big|\Big]^pd\mu(x)\\
		=:&I_1+I_2,
	\end{aligned}
	\]
	where 
	\[
	\mathcal{I}_{\beta,1}=\{\alpha\in \mathcal{I}: B_\alpha\cap B_\beta^*\neq \emptyset\}, \ \ \text{and} \ \  \mathcal{I}_{\beta,2}=\{\alpha\in \mathcal{I}: B_\alpha\cap B_\beta^*= \emptyset\}.
	\]
	For each $\alpha\in \mathcal{I}$ we have
	\[
	\sup_{0< t\leq [\epsilon \rho(x)]^2}\Big|\int_X(\psi_\alpha(x)-\psi_\alpha(y))\widetilde{p}_t(x,y)f(y)d\mu(y)\Big|\leq 2\sup_{t>0}\int_X\widetilde{p}_t(x,y)|f(y)|d\mu(y)\lesi \mathcal{M}f(x).
	\]
	This, in combination with the fact that $\sharp I_{\beta,1}\lesi 1$, implies
	\[
	I_1\lesi \sum_{\beta\in \mathcal{I}}\int_{B_\beta}|\mathcal{M}f(x)|^p\dx\sim \|\mathcal{M}f\|_{L^p(X)}^p\lesi \|f\|_{L^p(X)}^p.
	\]
	To estimate $I_2$, we can see that $\psi_\alpha(x)=0$ for $x\in B_{\beta}, \alpha\in I_{\beta,2}$. Hence,
	\[
	\begin{aligned}
	I_2&=\sum_{\beta\in \mathcal{I}}\int_{B_\beta}\Big[\sum_{\alpha\in \mathcal{I}_{\beta,2}}\sup_{0< t\leq [\epsilon \rho(x)]^2}\Big|\int_{B_\alpha}\psi_\alpha(y)\widetilde{p}_t(x,y)f(y)d\mu(y)\Big|\Big]^pd\mu(x)\\
	&\lesi \sum_{\beta\in \mathcal{I}}\int_{B_\beta}\Big[\sum_{\alpha\in \mathcal{I}_{\beta,2}}\int_{B_\alpha}\f{1}{\mu(B(x,d(x,y)))}\exp\Big(-\f{d(x,y)^2}{c\rx^2}\Big)|f(y)|d\mu(y)\Big]^pd\mu(x).
	\end{aligned}
	\]
	Since $x\in B_\beta$ and $y\in B_\alpha, \alpha \in \mathcal{I}_{\beta,2}$, then $d(x,y)\geq r_{B_\beta}\sim \rho(x)$. Hence, we find that
\begin{equation}\label{eq-proof L2 Te}
\begin{aligned}
	I_2
	&\lesi \sum_{\beta\in \mathcal{I}}\int_{B_\beta}\Big[\sum_{\alpha\in \mathcal{I}_{\beta,2}}\int_{B_\alpha}\f{1}{\mu(B(x,\rho(x)))}\exp\Big(-\f{d(x,y)^2}{c\rx^2}\Big)|f(y)|d\mu(y)\Big]^pd\mu(x)\\
	&\lesi \sum_{\beta\in \mathcal{I}}\int_{B_\beta}\Big[\int_{X}\f{1}{\mu(B(x,\rho(x)))}\exp\Big(-\f{d(x,y)^2}{c\rx^2}\Big)|f(y)|d\mu(y)\Big]^pd\mu(x).
	\end{aligned}
		\end{equation}
	
	Moreover,
	\[
	\int_{X}\f{1}{\mu(B(x,\rho(x)))}\exp\Big(-\f{d(x,y)^2}{c\rx^2}\Big)|f(y)|d\mu(y)\lesi \mathcal{M}f(x).
	\]
	Inserting this into \eqref{eq-proof L2 Te}, we arrive at
	\[
	\begin{aligned}
	I_2 \lesi \sum_{\beta\in \mathcal{I}}\int_{B_\beta}|\mathcal{M}f(x)|^p\dx\sim \|\mathcal{M}f\|_{L^p(X)}^p\lesi \|f\|_{L^p(X)}^p.
	\end{aligned}
	\]
	This completes our proof.
\end{proof}
\begin{lem}\label{Lem2: Hardy spaces}
	Let $\f{n}{n+\delta_1}<p\leq 1$ and $q\in (p,\vc]\cap [1,\vc]$. Then there exists $\kappa$ so that for any $0<\epsilon\leq 1$, we have
	\begin{equation}\label{eq1-Lem2Hardyspaces}
	\Big\|\sum_{\alpha\in \mathcal{I}}\sup_{0< t\leq [\epsilon \rho(x)]^2}|\psi_\alpha(x)e^{-t\LL}f(x) -e^{-t\LL}(f\psi_\alpha )(x)|\,\Big\|^p_{L^p(X)}\lesi \epsilon^{\kappa}\|f\|^p_{h^{p,q}_{at,\rho,\epsilon}(X)},
	\end{equation}
	for all $f\in h^{p,q}_{at,\rho,\epsilon}(X)$.
\end{lem}
\begin{proof}
	Due to Lemma \ref{lem-Tatom} and Lemma \ref{lem-Tpsi} it suffices to prove (\ref{eq1-Lem2Hardyspaces}) for any $(p,q,\rho,\epsilon)$ atom $a$. Assume that $a$ is a $(p,q,\rho,\epsilon)$ atom associated to $B:=B(x_0,r)$. We then set
	$$
	\begin{aligned}
	\mathcal{I}_{1,B} &:=\{\alpha: B_\alpha\cap B(x_0,\rho(x_0))^* \neq \emptyset\} \\
	\mathcal{I}_{2,B} &:=\{\alpha: B_\alpha\cap B(x_0,\rho(x_0))^* = \emptyset\}.
	\end{aligned}
	$$
	Hence,
	\begin{equation}
	\begin{aligned}
	\Big\|\sum_{\alpha\in \mathcal{I}}\sup_{0< t\leq [\epsilon \rho(x)]^2}&|\psi_\alpha(x)e^{-t\LL}a(x) -e^{-t\LL}(a\psi_\alpha )(x)|\,\Big\|^p_{L^p(X)}\\
	&\leq 
	\sum_{\alpha\in \mathcal{I}}\Big\|\sup_{0< t\leq [\epsilon \rho(x)]^2}|\psi_\alpha(x)e^{-t\LL}a(x) -e^{-t\LL}(a\psi_\alpha )(x)|\,\Big\|^p_{L^p(X)}\\
	&\lesi \sum_{\alpha\in \mathcal{I}_{1,B}} \ldots+\sum_{\alpha\in \mathcal{I}_{2,B}} \ldots\\
	&=:J_1+J_2.
	\end{aligned}
	\end{equation}
	Since $a\psi_\alpha =0$ for all $\alpha\in \mathcal{I}_{2,B}$, then from Lemma \ref{Lem2: rho} we conclude that
	\begin{equation}
	\begin{aligned}
	J_2&=\sum_{\alpha\in \mathcal{I}_{2,B}} \Big\|\sup_{0< t\leq [\epsilon \rho(x)]^2}|\psi_\alpha(x)e^{-t\LL}a(x)|\,\Big\|^p_{L^p(X)}\lesi \sum_{\alpha\in \mathcal{I}_{2,B}} \Big\|\sup_{0< t\leq [\epsilon \rho(x)]^2}|e^{-t\LL}a(x)|\,\Big\|^p_{L^p(B_\alpha)}\\
	&\lesi \Big\|\sum_{\alpha\in \mathcal{I}_{2,B}} \sup_{0< t\leq [\epsilon \rho(x)]^2}|e^{-t\LL}a(x)|\,\Big\|^p_{L^p(X\backslash B(x_0,\rho(x_0))^*)}.
	\end{aligned}
	\end{equation}
	We can argue as in Lemma \ref{Lem1: Hardy spaces} and  arrive at $J_2\lesi \epsilon^{\kappa}$.
	
	It remains to show that $J_1\lesi \epsilon^{\kappa}$. To do this, we first note that due to Lemma \ref{Lem2: rho}, $\sharp \,\mathcal{I}_{1, B}\leq C$ where $C$ is a constant independent of $a$. Hence, in order to prove $J_1\lesi \epsilon^{\kappa}$, it suffices to prove that for $\alpha\in \mathcal{I}_{1, B}$, we have
	$$
	\Big\|\sup_{0< t\leq [\epsilon \rho(x)]^2}|\psi_\alpha(x)e^{-t\LL}a(x) -e^{-t\LL}(a\psi_\alpha )(x)|\,\Big\|^p_{L^p(X)}\lesi \epsilon^{\kappa}.
	$$
	Obviously,
	$$
	\begin{aligned}
	\Big\|\sup_{0< t\leq [\epsilon \rho(x)]^2}&|\psi_\alpha(x)e^{-t\LL}a(x)-e^{-t\LL}(a\psi_\alpha )(x)|\,\Big\|^p_{L^p(X)}\\
	&\lesi \Big\|\sup_{0< t\leq [\epsilon \rho(x)]^2}|\psi_\alpha(x)e^{-t\LL}a(x) -e^{-t\LL}(a\psi_\alpha )(x)|\,\Big\|^p_{L^p(4B)}\\
	&\ \ \ + \Big\|\sup_{0< t\leq [\epsilon \rho(x)]^2}|\psi_\alpha(x)e^{-t\LL}a(x) -e^{-t\LL}(a\psi_\alpha )(x)|\,\Big\|^p_{L^p(X\backslash 4B)}:= J_{11}+J_{12}.
	\end{aligned}
	$$
	To take care of $J_{11}$, using the fact that $\rho(x)\sim \rho(x_\alpha)$ and H\"older's inequality, we write
	$$
	\begin{aligned}
	J_{11}&=\int_{4B}\sup_{0< t\leq [\epsilon \rho(x)]^2}\Big|\int_{B}\widetilde{p}_{t}(x,y)(\psi_\alpha(x)-\psi_\alpha(y))a(y)d\mu(y)\Big|^pd\mu(x)\\
	&\lesi \int_{4B}\sup_{0< t\leq [\epsilon \rho(x)]^2}\Big[\int_{B}\f{1}{\mu(B(x,\sqrt{t}))}\exp\Big(-\f{d(x,y)^2}{ct}\Big)\f{d(x,y)}{\rho(x_\alpha)}|a(y)|d\mu(y)\Big]^pd\mu(x)\\
	&\lesi \int_{4B}\sup_{0< t\leq [\epsilon \rho(x)]^2}\Big[\int_{B}\f{1}{\mu(B(x,\sqrt{t}))}\exp\Big(-\f{d(x,y)^2}{ct}\Big)\f{\sqrt{t}}{\rho(x_\alpha)}|a(y)|d\mu(y)\Big]^pd\mu(x)\\
	&\lesi \epsilon^p\mu(B)^{1-p/q}\Big[\int_{4B}\Big(\sup_{0< t\leq [\epsilon \rho(x)]^2}\int_{B}\f{1}{\mu(B(x,\sqrt{t}))}\exp\Big(-\f{d(x,y)^2}{ct}\Big)|a(y)|d\mu(y)\Big)^qd\mu(x)\Big]^{p/q}\\
	&\lesi \epsilon^p\mu(B)^{1-p/q}\Big[\int_{4B}\Big(\mathcal{M}(|a|)\Big)^qd\mu(x)\Big]^{p/q}\\
	&\lesi \epsilon^p \mu(B)^{1-p/q}\|a\|_{L^q}^p \;\lesi  \epsilon^p,
	\end{aligned}
	$$
	where $\mathcal{M}$ is the Hardy-Littlewood maximal function.
	
	We now take care of $J_{12}$. To do this, we consider two cases:
	
	\textbf{Case 1: $\epsilon\rho(x_0)/4\leq r\leq \epsilon\rho(x_0)$}
	
	In this situation, we use the fact that $d(x,y)\sim d(x,x_0)$ and $\rho(x_\alpha)\sim\rho(x_0)$ for $x\in X\backslash 4B$ and $x_0,y\in B$ and the argument above to obtain that
	$$
	\begin{aligned}
	J_{12}&\lesi \int_{X\backslash  4B}\sup_{0< t\leq [\epsilon \rho(x)]^2}\Big|\int_{B}\f{1}{\mu(B(x,\sqrt{t}))}\exp\Big(-\f{d(x,y)^2}{ct}\Big)\f{d(x,y)}{\rho(x_\alpha)}|a(y)|d\mu(y)\Big|^pd\mu(x)\\
	&\lesi \int_{X\backslash 4B}\sup_{0< t\leq [\epsilon \rho(x)]^2}\Big|\int_{B}\f{1}{\mu(B(x_0,d(x,x_0)))}\exp\Big(-\f{d(x,x_0)^2}{c[\epsilon\rho(x)]^2}\Big)\f{\sqrt{t}}{\rho(x_0)}|a(y)|\,d\mu(y)\Big|^pd\mu(x).
	\end{aligned}
	$$
	This, in combination with \eqref{eq-rho exp}, yields that for $N>n(1-p)/p$ we have
	$$
	\begin{aligned}
	J_{12}&\lesi \epsilon^p\int_{X\backslash 4B}\Big|\int_{B}\f{1}{\mu(B(x_0,d(x,x_0)))}\Big(\f{\epsilon\rho(x_0)}{d(x,x_0)}\Big)^{N}|a(y)|\,d\mu(y)\Big|^pd\mu(x)\\
	&\lesi \epsilon^p \mu(B)^{p-1}\sum_{j\geq 3}\int_{S_j(B)}\Big[\f{1}{\mu(B(x_0,d(x,x_0)))}\Big(\f{r}{d(x,x_0)}\Big)^{N}\Big]^pd\mu(x)\\
	&\lesi \epsilon^p \sum_{j\geq 3}\Big(\f{\mu(2^jB)}{\mu(B)}\Big)^{1-p}2^{-jNp}\\
	&\lesi \epsilon^p \sum_{j\geq 3} 2^{-j(Np-(1-p)n)}
	\;\lesi \epsilon^p.
	\end{aligned}
	$$
	
	\textbf{Case 2: $r<\epsilon\rho(x_0)/4$}
	
	In this case, since $\int a(y)d\mu(y) =0$, we have
	$$
	\begin{aligned}
	J_{12}&=\int_{X\backslash 4B}\sup_{0< t\leq [\epsilon \rho(x)]^2}\Big|\int_{B}[\widetilde{p}_{t}(x,y)-\widetilde{p}_{t}(x,x_0)](\psi_\alpha(x)-\psi_\alpha(x_0))a(y)d\mu(y)\Big|^pd\mu(x)\\
	& \ \ +\int_{X\backslash 4B}\sup_{0< t\leq [\epsilon \rho(x)]^2}\Big|\int_{B}\widetilde{p}_{t}(x,y)(\psi_\alpha(x_0)-\psi_\alpha(y))a(y)d\mu(y)\Big|^pd\mu(x)\\
	&:=K_1+K_2.
	\end{aligned}
	$$
	By (\ref{H}) and the fact that $\rho(x_\alpha)\sim \rho(x_0)$ we have
	$$
	\begin{aligned}
	K_1&\lesi \int_{X\backslash 4B}\sup_{0< t\leq [\epsilon \rho(x)]^2}\Big|\int_{B}\Big(\f{d(y,x_0)}{\sqrt{t}}\Big)^{\delta_1} \f{1}{\mu(B(x,\sqrt{t}))}\exp\Big(-\f{d(x,x_0)^2}{ct}\Big)\f{d(x,x_0)}{\rho(x_0)} a(y)d\mu(y)\Big|^pd\mu(x)\\
	&\lesi \int_{X\backslash 4B}\Big|\int_{B}\Big(\f{d(y,x_0)}{\rho(x_0)}\Big)^{\delta_1} \f{1}{\ \mu(B(x_0,d(x,x_0)))}\exp\Big(-\f{d(x,x_0)^2}{c[\epsilon\rho(x)]^2}\Big)a(y)d\mu(y)\Big|^pd\mu(x)\\
	\end{aligned}
	$$
	which along with \eqref{eq1-rho exp}  gives
	$$
	\begin{aligned}
	&\lesi \int_{X\backslash 4B}\Big|\int_{B}\Big(\f{d(y,x_0)}{\rho(x_0)}\Big)^{\delta_1} \f{1}{\ \mu(B(x_0,d(x,x_0)))}\Big(\f{\epsilon\rho(x_0)}{d(x,x_0)}\Big)^{\delta}a(y)d\mu(y)\Big|^pd\mu(x)\\
	&\lesi \epsilon^{\delta_1 p}\int_{X\backslash 4B}\Big|\int_{B}\Big(\f{r}{d(x,x_0)}\Big)^{\delta_1} \f{1}{\mu(B(x_0,d(x,x_0)))}a(y)d\mu(y)\Big|^pd\mu(x)\\
	&\lesi \epsilon^{\delta p}\mu(B)^{p-1}\sum_{j\geq 3}\int_{S_j(B)}\Big[\Big(\f{r}{d(x,x_0)}\Big)^{\delta_1} \f{1}{\mu(B(x_0,d(x,x_0)))}\Big]^pd\mu(x)\\
	&\lesi \epsilon^{\delta p}\sum_{j\geq 3}\Big(\f{r}{2^jr}\Big)^{\delta_1 p} \Big(\f{\mu(2^jB)}{\mu(B)}\Big)^{1-p}\\
	&\lesi \epsilon^{\delta_1 p}\sum_{j\geq 3}2^{-j(\delta p -(1-p)n)}\lesi \epsilon^{\delta_1 p}.
	\end{aligned}
	$$
	
	This completes our proof.
\end{proof}

\subsection{Proof of Theorem \ref{mainthm2s}}
For each $t>0$ we define
\begin{equation*}
K_{t,\rho}(x,y)=\widetilde{p}_t(x,y)\exp\left[-\Big(\f{\sqrt{t}}{\rho(x)}\Big)^{\delta_1}\right]
\end{equation*}
for all $x,y\in X$ (where $\delta_1$ is the constant  in (A3)) and its associated operator by
\[
T_{t,\rho}f(x)=\int_{X}K_{t,\rho}(x,y)f(y)\dy.
\]
For each $t>0$ we set 
\begin{equation}\label{eq-Qt rho}
Q_{t,\rho}(x,y)=\widetilde{p}_t(x,y)-K_{t,\rho}(x,y)=\widetilde{p}_t(x,y)\left[1-e^{-\left[\f{\sqrt{t}}{\rho(x)}\right]^{\delta_1}}\right].
\end{equation}
 Then we have the following estimate.
\begin{lem}
	\label{lem-etL Kt ho}
	Let $Q_{t,\rho}$ be defined in \eqref{eq-Qt rho}. Then we have the following estimates.
	\begin{equation}\label{eq1-Q t rho}
	|Q_t(x,y)|\leq \Big[\f{\sqrt{t}}{\rho(x)}\Big]^{\delta_1}\f{1}{\mu(B(x,\sqrt{t}))}\exp\Big(-\f{d(x,y)^2}{ct}\Big)
	\end{equation}
	for all $x,y\in X$ and $t>0$, and 
	\begin{equation}\label{eq2-Q t rho}
	|Q_t(x,y)-Q_t(x,y_0)|\leq \Big[\f{d(y,y_0)}{\rho(x)}\Big]^{\delta_1} \f{C}{\mu(B(x,\sqrt{t}))}\exp\Big(-\f{d(x,y)^2}{ct}\Big).
	\end{equation}
	whenever $d(y,y_0)\leq d(x,y)/2$ and $t>0$.
\end{lem}
\begin{proof}
	Using (A2) and the inequality $1-e^{-x}\lesi x$, valid for all  $x>0$, we obtain \eqref{eq1-Q t rho}.
	
		We now take care of \eqref{eq2-Q t rho}. Observe that 
	\[
	Q_t(x,y)-Q_t(x,y_0)=(\widetilde{p}_t(x,y)-\widetilde{p}_t(x,y_0))\left[1-e^{-\left[\f{\sqrt{t}}{\rho(x)}\right]^{\delta_1}}\right].
	\]
	This, along with (A3) and the inequality $1-e^{-x}\lesi x$ again implies \eqref{eq2-Q t rho}.
\end{proof}

\begin{lem}\label{Lem3: Hardy spaces}
	Let $\f{n}{n+\delta_1}<p\leq 1$ and $q\in (p,\vc]\cap [1,\vc]$. Then there exists $\kappa$ so that for any $0<\epsilon\leq 1$, we have
	\begin{equation}\label{eq1-Lem3Hardy}
	\Big\|\sup_{0< t\leq [\epsilon \rho(x)]^2}|(e^{-t\LL}-T_{t,\rho})f(x)|\,\Big\|^p_{L^p(X)}\lesi \epsilon^{\kappa}\|f\|^p_{h^{p,q}_{at,\rho,\epsilon}(X)},
	\end{equation}
	for all $f\in h^{p,q}_{at,\rho,\epsilon}(X)$.
\end{lem}
\begin{proof}
	Observe that
	\[
	\sup_{0< t\leq [\epsilon \rho(x)]^2}|(e^{-t\LL}-T_{t,\rho})f(x)|\lesi \sup_{t>0}|e^{-t\LL}f(x)|\lesi \mathcal{M}f(x).
	\]
	Hence, from this and Lemma \ref{lem-Tatom} it suffices to prove \eqref{eq1-Lem3Hardy} for all $(p,q,\rho,\epsilon)$ atoms. Let  $a$ be $(p,q,\rho,\epsilon)$ atom associated to a ball $B:=B(x_0,r)$. We write
	$$
	\begin{aligned}
	\Big\|\sup_{0< t\leq [\epsilon \rho(x)]^2}|(e^{-t\LL}-T_{t,\rho})f(x)|\,\Big\|^p_{L^p(X)}&\leq \Big\|\sup_{0< t\leq [\epsilon \rho(x)]^2}|(e^{-t\LL}-T_{t,\rho})f(x)|\,\Big\|^p_{L^p(4B)}\\
	& \  \  \ +\Big\|\sup_{0< t\leq [\epsilon \rho(x)]^2}|(e^{-t\LL}-T_{t,\rho})f(x)|\,\Big\|^p_{L^p(X\backslash 4B)}\\
	&=I_1+I_2.
	\end{aligned}
	$$
	Using \eqref{eq1-Q t rho}, H\"older's inequality and the $L^q$-boundedness of $\mathcal{M}$, we get that
	$$
	\begin{aligned}
	I_1&\lesi \int_{4B}\Big[\sup_{0< t\leq [\epsilon \rho(x)]^2}\Big[\f{\sqrt{t}}{\rho(x)}\Big]^{\delta_1}\int_B\f{1}{\mu(B(x,\sqrt{t}))}\exp\Big(-\f{d(x,y)^2}{ct}\Big)|a(y)|d\mu(y)\Big]^pd\mu(x)\\
	&\lesi \epsilon^{p\delta_1}\mu(B)^{1-p/q} \Big[\int_{4B}\Big[\sup_{0< t\leq [\epsilon \rho(x)]^2}\int_B\f{1}{\mu(B(x,\sqrt{t}))}\exp\Big(-\f{d(x,y)^2}{ct}\Big)|a(y)|d\mu(y)\Big]^qd\mu(x)\Big]^{p/q}\\
	&\lesi \epsilon^{p\delta_1}\mu(B)^{1-p/q} \Big[\int_{4B}\Big[\mathcal{M}(|a|)(x)\Big]^qd\mu(x)\Big]^{p/q}\\
	&\lesi \epsilon^{p\delta_1}.
	\end{aligned}
	$$
	
	The estimate of $I_2$ can be done by considering the following two cases.
	
	\textbf{Case 1: $\epsilon\rho(x_0)/4\leq r\leq \epsilon\rho(x_0)$}
	
	By \eqref{eq1-Q t rho} again, we can write
	$$
	\begin{aligned}
	I_2&=\int_{X\backslash 4B}\Big[\sup_{0< t\leq [\epsilon \rho(x)]^2}\int_B \Big(\f{\sqrt{t}}{\rho(x)}\Big)^{\delta_1}\f{1}{\mu(B(x,\sqrt{t}))}\exp\Big(-\f{d(x,y)^2}{ct}\Big)|a(y)|d\mu(y)\Big]^pd\mu(x)\\
	&\lesi \epsilon^{p\delta_1}\|a\|_{L^1}^p\int_{X\backslash 4B}\Big[\f{1}{\mu(B(x_0,d(x,x_0)))}\exp\Big(-\f{d(x,x_0)^2}{c[\epsilon \rho(x)]^2}\Big)\Big]^pd\mu(x)
	\end{aligned}
	$$
	which along with \eqref{eq1-rho exp} implies that, for $N>n(1-p)/p$,
	$$
	\begin{aligned}
	I_2&\lesi \epsilon^{p\delta_1}\mu(B)^{p-1}\int_{X\backslash 4B}\Big[\f{1}{\mu(B(x_0,d(x,x_0)))}\Big(\f{\epsilon \rho(x_0)}{d(x,x_0)}\Big)^N\Big]^pd\mu(x)\\
	&\lesi \epsilon^{p\delta_1}\mu(B)^{p-1}\int_{X\backslash 4B}\Big[\f{1}{\mu(B(x_0,d(x,x_0)))}\Big(\f{r}{d(x,x_0)}\Big)^N\Big]^pd\mu(x)\\
	&\lesi \epsilon^{p\delta_1}.
	\end{aligned}
	$$
	
	\textbf{Case 2: $r< \epsilon\rho(x_0)/4$}
	
	In this situation, $\int a(y)d\mu(y)=0$. This implies that
	$$
	I_2=\int_{X\backslash 4B}\Big|\sup_{0< t\leq [\epsilon \rho(x)]^2}\int_B (Q_t(x,y)-Q_t(x,x_0))a(y)d\mu(y)\Big|^pd\mu(x).
	$$
	Hence, by \eqref{eq2-Q t rho} we obtain that
	$$
	\begin{aligned}
	I_2&=\int_{X\backslash 4B}\Big[\sup_{0< t\leq [\epsilon \rho(x)]^2}\int_B \Big(\f{d(y,x_0)}{\rx}\Big)^{\delta_1} \f{1}{\mu(B(x,\sqrt{t}))}\exp\Big(-\f{d(x,y)^2}{ct}\Big)|a(y)|d\mu(y)\Big]^pd\mu(x)\\
	&\lesi \int_{X\backslash 4B}\Big[\int_B \Big(\f{r}{\rho(x)}\Big)^{\delta_1} \f{1}{\mu(B(x_0,d(x,x_0)))}\exp\Big(-\f{d(x,x_0)^2}{c[\epsilon\rho(x)]^2}\Big)|a(y)|d\mu(y)\Big]^pd\mu(x)\\
	&\lesi \int_{X\backslash 4B}\Big[\int_B \Big(\f{\epsilon r}{d(x,x_0)}\Big)^{\delta_1} \f{1}{\mu(B(x_0,d(x,x_0)))}|a(y)|d\mu(y)\Big]^pd\mu(x)\\
	&\lesi \epsilon^{\delta_1 p},
	\end{aligned}
	$$
	as long as $p>n/(n+\delta_1)$.
	This completes our proof.
\end{proof}

\begin{proof}[Proof of Theorem \ref{mainthm2s}]
	We first prove the continuous embedding $h^{p,q}_{at,\rho}(X)\hookrightarrow h^p_{\LL,{\rm max},\rho}(X)$. Since the space $L^2(X)$ is dense in both $h^{p,q}_{at,\rho}(X)$ and $h^p_{\LL,{\rm rad},\rho}(X)$,  it suffices to show that  $h^{p,q}_{at,\rho}(X)\cap L^2(X)\hookrightarrow h^p_{\LL,{\rm rad},\rho}(X)$. Since $f^*_{\LL,\rho}$ is dominated by $\mathcal{M}f$, from Lemma \ref{lem-Tatom} it suffices to show that there exists $C$ so that
	\begin{align}\label{maximal atoms}
	\|a^*_{\LL,\rho}\|^p_{L^p}\leq C,
	\end{align}
	for all $(p,q,\rho)$-atoms associated to balls $B=B(x_0,r)$.
	
	To prove \eqref{maximal atoms}, we first write
	$$
	\|a^*_{\LL,\rho}\|^p_{L^p}\leq \|a^*_{\LL,\rho}\|^p_{L^p(4B)}+\|a^*_{\LL,\rho}\|^p_{L^p(M\backslash 4B)}: =I_1+I_2.
	$$
	The first term can be handled easily by H\"older's inequality and the $L^q$-boundedness of $\mathcal{M}$:
	$$
	\begin{aligned}
	I_1&\lesi \mu(B)^{1-p/q}\|a^*_{\LL,\rho}a\|^p_{L^q(4B)}\lesi \mu(B)^{1-p/q}\|\mathcal{M}a\|^p_{L^q(X)}
	\leq C.
	\end{aligned}
	$$
	To take care of the term $I_2$ we consider two cases.
	
	\textbf{Case 1: $\rho(x_0)/4 \leq r\leq \rho(x_0)$}
	
	From (A2) we have
	\[
	I_2\lesi \int_{X\backslash 4B}\sup_{0<t<\rho(x)^2}\sup_{d(x,y)<t}\Big[\int_B \f{1}{\mu(B(z,\sqrt{t}))}\exp\Big(-\f{d(y,z)^2}{ct}\Big)|a(z)|d\mu(z)\Big]^pd\mu(x).
	\]
	This together with the fact that
	\[
	\exp\Big(-\f{d(y,z)^2}{ct}\Big)\sim \exp\Big(-\f{d(x,z)^2}{ct}\Big), \ \ \ \text{as $d(x,y)<t$},
	\]
	implies that
	\[
	I_2\lesi \int_{X\backslash 4B}\sup_{0<t<\rho(x)^2}\Big[\int_B \f{1}{\mu(B(z,\sqrt{t}))}\exp\Big(-\f{d(x,z)^2}{ct}\Big)|a(z)|d\mu(z)\Big]^pd\mu(x).
	\]
	We then apply \eqref{doub2} to obtain further
	\[
		\begin{aligned}
	&\lesi \int_{X\backslash 4B}\sup_{0<t<\rho(x)^2}\Big[\int_B \f{1}{\mu(B(z,d(x,z)))}\exp\Big(-\f{d(x,z)^2}{ct}\Big)|a(z)|d\mu(z)\Big]^pd\mu(x).
		\end{aligned}
	\]
	Moreover, obverse that in this situation $d(x,z)\sim d(x,x_0)$. Hence, for $N>n(1-p)/p$,
	\[
	\begin{aligned}
	I_2
	&\lesi \int_{X\backslash 4B}\Big[\int_B \f{1}{\mu(B(x_0,d(x,x_0)))}\exp\Big(-\f{d(x,x_0)^2}{c\rho(x)^2}\Big)|a(y)|d\mu(y)\Big]^pd\mu(x)\\
	&\lesi \int_{X\backslash 4B}\Big[\int_B \f{1}{\mu(B(x_0,d(x,x_0)))}\Big(\f{\rho(x_0)}{d(x,x_0)}\Big)^N|a(y)|d\mu(y)\Big]^pd\mu(x)\\
	&\lesi \int_{X\backslash 4B}\Big[\int_B \f{1}{\mu(B(x_0,d(x,x_0)))}\Big(\f{r}{d(x,x_0)}\Big)^N|a(y)|d\mu(y)\Big]^pd\mu(x)\\
	&\leq C,
	\end{aligned}
	\]
	where in the second inequality we used \eqref{eq1-rho exp}. \\

	\textbf{Case 2: $r<\rho(x_0)/4$}
	
	Observe that
	$$
	\begin{aligned}
	I_2&\lesi \int_{X\backslash 4B}\sup_{0<t\leq 4r^2}\sup_{d(x,y)<t}\Big|\int_B \widetilde{p}_t(y,z)a(z)d\mu(z)\Big|^pd\mu(x)\\
	& \ \ + \int_{X\backslash 4B}\sup_{ 4r^2\le t<\rho(x)^2}\sup_{d(x,y)<t}\Big|\int_B \widetilde{p}_t(y,z)a(z)d\mu(z)\Big|^pd\mu(x)\\
	&=I_{21}+I_{22}.
	\end{aligned}
	$$
	Fix $N>n(1-p)/p$. Arguing similarly as above we obtain
	$$
	\begin{aligned}
	I_{21}&\lesi \int_{X\backslash 4B}\sup_{0<t<\rho(x)^2}\Big[\int_B \f{1}{\mu(B(z,d(x,z)))}\exp\Big(-\f{d(x,z)^2}{ct}\Big)|a(z)|d\mu(z)\Big]^pd\mu(x)\\
	&\lesi \int_{X\backslash 4B}\sup_{0<t\leq 4r^2}\Big[\int_B \f{1}{\mu(B(z,d(x,z)))}\Big(\f{\sqrt{t}}{d(x,z)}\Big)^N|a(z)|d\mu(z)\Big]^pd\mu(x)\\
	&\lesi \int_{X\backslash 4B}\Big[\int_B \f{1}{\mu(B(x_0,d(x,x_0)))}\Big(\f{r}{d(x,x_0)}\Big)^N|a(y)|d\mu(y)\Big]^pd\mu(x)\\
	&\leq C.
	\end{aligned}
	$$
	To take care of $I_{22}$ we use the cancellation property of $a$ to arrive at
	\begin{equation}\label{eqI22}
	I_{22}=\int_{X\backslash 4B}\sup_{ 4r^2\le t<\rho(x)^2}\sup_{d(x,y)<t}\Big|\int_B [\widetilde{p}_t(y,z)-\widetilde{p}_t(y,x_0)]a(z)d\mu(z)\Big|^pd\mu(x).
	\end{equation}
	From (A3), we have
	\[
	|\widetilde{p}_t(y,z)-\widetilde{p}_t(y,x_0)|\leq \Big(\f{d(z,x_0)}{\sqrt{t}}\Big)^{\delta_1}\Big[\f{1}{\mu(B(z,\sqrt{t}))}\exp\Big(-\f{d(y,z)^2}{ct}+\f{1}{\mu(B(x_0,\sqrt{t}))}\exp\Big(-\f{d(y,x_0)^2}{ct}\Big)\Big]
	\]
	Hence, for any $x\in (4B)^c$ with $d(x,y)<t$ we have, by \eqref{doub2},
	\[
	\begin{aligned}
	|\widetilde{p}_t(y,z)&-\widetilde{p}_t(y,x_0)|\\
	&\leq \Big(\f{d(z,x_0)}{\sqrt{t}}\Big)^{\delta_1}\Big[\f{1}{\mu(B(z,\sqrt{t}))}\exp\Big(-\f{d(x,z)^2}{ct}\Big)+\f{1}{\mu(B(x_0,\sqrt{t}))}\exp\Big(-\f{d(x,x_0)^2}{ct}\Big)\Big]\\
	&\leq \Big(\f{d(z,x_0)}{\sqrt{t}}\Big)^{\delta_1}\Big[\f{1}{\mu(B(z,d(x,z)))}\exp\Big(-\f{d(x,z)^2}{ct}\Big)+\f{1}{\mu(B(x_0,d(x,x_0)))}\exp\Big(-\f{d(x,x_0)^2}{ct}\Big)\Big]\\
	\end{aligned}
	\]
	Note that $d(x,z)\sim d(x,x_0)$ and $\mu(B(z,d(x,z)))\sim \mu(B(x_0,d(x,x_0)))$ as $z\in B$ and $x\in (4B)^c$. From this and the inequality above we obtain
	\[
	\sup_{d(x,y)<t
		}|\widetilde{p}_t(y,z)-\widetilde{p}_t(y,x_0)|\leq \Big(\f{d(z,x_0)}{\sqrt{t}}\Big)^{\delta_1}\f{1}{\mu(B(x_0,d(x,x_0)))}\exp\Big(-\f{d(x,x_0)^2}{ct}\Big).
	\]
	
	Inserting this into \eqref{eqI22} we have
	$$
	\begin{aligned}
	I_{22}&\lesi \int_{X\backslash 4B}\sup_{ 4r^2\le t}\Big[\int_B \Big(\f{d(z,x_0)}{\sqrt{t}}\Big)^{\delta_1}\f{1}{\mu(B(x_0,d(x,x_0)))}\exp\Big(-\f{d(x,x_0)^2}{ct}\Big)|a(z)|d\mu(z)\Big]^pd\mu(x)\\
	&\lesi \int_{X\backslash 4B}\Big[\int_B \Big(\f{d(z,x_0)}{d(x,x_0)}\Big)^{\delta_1}\f{1}{\mu(B(x_0,d(x,x_0)))}|a(z)|d\mu(z)\Big]^pd\mu(x)\\
	&\lesi \int_{M\backslash 4B}\Big[\int_B \Big(\f{r}{d(z,x_0)}\Big)^{\delta_1}\f{1}{\mu(B(x_0,d(x,x_0)))}|a(z)|d\mu(z)\Big]^pd\mu(x)\\
	&\leq C,
	\end{aligned}
	$$
	provided $p>n/(n+\delta_1)$. Therefore \eqref{maximal atoms} holds.\\

	Since $h^p_{\LL,{\rm max},\rho}(X)\subset h^p_{\LL,{\rm rad}, \rho}(X)$, to complete the proof we  remain to prove  that $h^p_{\LL,{\rm rad}, \rho}(X)\cap L^2(X)\hookrightarrow  h^{p,q}_{at,\rho}(X)\cap L^2(X)$. To do this we first note that for fixed numbers $\epsilon_1, \epsilon_2\in (0,1]$, there exists $C=C(\epsilon_1,\epsilon_2)$ so that
	\begin{equation}\label{eq0-mainthm}
	C^{-1}\|\cdot\|_{h^{p,q}_{at,\rho,\epsilon_1}(X)}\leq \|\cdot\|_{h^{p,q}_{at,\rho,\epsilon_2}(X)} \leq C\|\cdot\|_{h^{p,q}_{at,\rho,\epsilon_1}(X)}.
	\end{equation}
	Hence, it suffices to prove that there exists $\epsilon_0\in (0,1]$ so that
	\begin{equation}\label{eq1-mainthm}
	\|f\|_{h^{p,q}_{at,\rho,\epsilon_0}(X)}\lesi \|f\|_{h^p_{\LL,{\rm rad}, \rho}(X)},
	\end{equation}
	for all $f\in h^p_{\LL,{\rm rad}, \rho}(X)\cap L^2(X)$.

	Indeed, we note that for each $\alpha\in \mathcal{I}$, $f\psi_\alpha$ is supported in the ball $B_\alpha=B(x_\alpha,\rho(x_\alpha))$. Hence, by Proposition \ref{Hardyfunc-compactsupp} and a scaling argument, we can decompose $f\psi_\alpha$ into an atomic $(p,q,\rho,\epsilon)$-representation with $(p,q,\rho,\epsilon)$-atoms supported in $B^*_\alpha$. Moreover we have, from Lemma \ref{lem-criticalfunction} (a), the existence of $c_0$ so that 
	\[
	c_0^{-1}\rho(x_\alpha)\leq  \rx\leq c_0\rho(x_\alpha), \ \ \ \text{for all $x\in B_\alpha^*$ and all $ \alpha\in\mathcal{I}$}.
	\] 
	Therefore, from Theorem \ref{mainthm2} and Lemma \ref{Lem1: Hardy spaces}, by a scaling argument we obtain
	\begin{equation*}
	\begin{aligned}
	\sum_{\alpha\in \mathcal{I}}\|\psi_\alpha f\|^p_{h^{p,q}_{at,\rho,\epsilon}(X)}&\lesi \sum_{\alpha\in \mathcal{I}}\Big\|\sup_{0<t<[\epsilon \rho(x_\alpha)]^2}|e^{-t\LL}\psi_\alpha f|\,\Big\|^p_{L^p(X)}\\
	&\lesi \sum_{\alpha\in \mathcal{I}}\Big\|\sup_{0<t<[\epsilon \rho(x_\alpha)]^2}|e^{-t\LL}\psi_\alpha f|\,\Big\|^p_{L^p(B_\alpha^*)}+\sum_{\alpha\in \mathcal{I}}\Big\|\sup_{0<t<[\epsilon \rho(x_\alpha)]^2}|e^{-t\LL}(\psi_\alpha f)|\,\Big\|^p_{L^p(X\backslash B_\alpha^*)}\\
	&\lesi \sum_{\alpha\in \mathcal{I}}\Big\|\sup_{0<t<[\widetilde{\epsilon} \rho(\cdot)]^2}|e^{-t\LL}(\psi_\alpha f)(\cdot)|\,\Big\|^p_{L^p(B_\alpha^*)}+\sum_{\alpha\in \mathcal{I}}\epsilon^{\kappa}\sum_{j\in \mathcal{I}_\alpha}\|\psi_j f\|^p_{h^{p,q}_{at,\rho,\epsilon}(X)}\\
	&\lesi \sum_{\alpha\in \mathcal{I}}\Big\|\sup_{0<t<[\widetilde{\epsilon} \rho(\cdot)]^2}|e^{-t\LL}(\psi_\alpha f)(\cdot)|\,\Big\|^p_{L^p(B_\alpha^*)}+\epsilon^{\kappa}\sum_{\alpha\in \mathcal{I}}\|\psi_\alpha f\|^p_{h^{p,q}_{at,\rho,\epsilon}(X)},
	\end{aligned}
	\end{equation*}
	where $\widetilde{\epsilon}=c_0\epsilon$.
	
	As a consequence,
	\begin{equation*}
	\sum_{\alpha}\|\psi_\alpha f\|^p_{h^{p,q}_{at,\rho,\epsilon}(X)}\lesi \sum_{\alpha}\Big\|\sup_{0<t<[\widetilde{\epsilon} \rho(\cdot)]^2}|e^{-t\LL}(\psi_\alpha f)(\cdot)|\,\Big\|^p_{L^p(B_\alpha^*)},
	\end{equation*}
	provided that $\epsilon$ is small enough.

	This, along with the inequality
	\[
	\|f\|^p_{h^{p,q}_{at,\rho,\epsilon}(X)}\leq \sum_{\alpha\in \mathcal{I}}\|\psi_\alpha f\|^p_{h^{p,q}_{at,\rho,\epsilon}(X)},
	\]
	further implies that
	\begin{equation*}
	\begin{aligned}
	\|f\|^p_{h^{p,q}_{at,\rho,\epsilon}(X)}&\lesi \sum_{\alpha}\Big\|\sup_{0<t<[\widetilde{\epsilon} \rho(\cdot)]^2}|e^{-t\LL}(\psi_\alpha f)(\cdot)|\,\Big\|^p_{L^p(B_\alpha^*)}\\
	&\lesi \sum_{\alpha}\Big\|\sup_{0<t<[\widetilde{\epsilon} \rho(\cdot)]^2}|e^{-t\LL}(\psi_\alpha f)(\cdot)-\psi_\alpha(\cdot)e^{-t\LL} f(\cdot)|\,\Big\|^p_{L^p(B_\alpha^*)}\\
	& \ \ \ +\sum_{\alpha}\Big\|\sup_{0<t<[\widetilde{\epsilon} \rho(\cdot)]^2}|\psi_\alpha(\cdot)[e^{-t\LL}-T_{t,\rho} ]f(\cdot)|\,\Big\|^p_{L^p(B_\alpha^*)}\\
	& \ \ \ + \sum_{\alpha}\Big\|\sup_{0<t<[\widetilde{\epsilon} \rho(\cdot)]^2}|\psi_\alpha(\cdot)T_{t,\rho}f(\cdot)|\,\Big\|^p_{L^p(B_\alpha^*)}\\
	&=:I_1+I_2+I_3.
	\end{aligned}
	\end{equation*}
	From Lemma \ref{coveringlemm}, we conclude that
	\[
	I_1\lesi \Big\|\sum_{\alpha}\sup_{0<t<[\widetilde{\epsilon} \rho(\cdot)]^2}|e^{-t\LL}(\psi_\alpha f)(\cdot)-\psi_\alpha(\cdot)e^{-t\LL} f(\cdot)|\,\Big\|^p_{L^p(X)}.
	\]
	Hence, we have
	\begin{equation*}
	\begin{aligned}
	\|f\|^p_{h^{p,q}_{at,\rho,\epsilon}(X)}
	&\lesi \sum_{\alpha}\Big\|\sup_{0<t<[\widetilde{\epsilon} \rho(\cdot)]^2}|e^{-t\LL}(\psi_\alpha f)(\cdot)-\psi_\alpha(\cdot)e^{-t\LL} f(\cdot)|\,\Big\|^p_{L^p(X)}\\
	& \ \ \ +\Big\|\sup_{0<t<[\widetilde{\epsilon} \rho(\cdot)]^2}|[e^{-t\LL}-T_{t,\rho} ]f(\cdot)]|\,\Big\|^p_{L^p(X)}+ \Big\|\sup_{0<t<\widetilde{\epsilon}\rho(\cdot)^2}|T_{t,\rho}f(\cdot)|\,\Big\|^p_{L^p(X)}.
	\end{aligned}
	\end{equation*}
	This along with Lemma \ref{Lem2: Hardy spaces}, Lemma \ref{Lem3: Hardy spaces} deduces that
	\begin{equation}\label{eq2-lem3Hardy}
	\begin{aligned}
	\|f\|^p_{h^{p,q}_{at,\rho,\epsilon}(X)}&\lesi \epsilon^{\kappa}\|f\|^p_{h^{p,q}_{at,\rho,\widetilde{\epsilon}}(X)}+\Big\|\sup_{0<t<\rho(\cdot)^2}|T_{t,\rho}f(\cdot)|\,\Big\|^p_{L^p(X)}
	\end{aligned}
	\end{equation}
	as long as $\tilde{\epsilon}=c_0\epsilon<1$.

	Note that since $\widetilde{\epsilon}=c_0\epsilon$, from the definition of Hardy spaces $h^{p,q}_{at,\rho,\epsilon}(X)$, there exists $C$ independent of $\epsilon$ so that
	$$
	\|f\|^p_{h^{p,q}_{at,\rho,\widetilde{\epsilon}}(X)}\leq C\|f\|^p_{h^{p,q}_{at,\rho,\epsilon}(X)}.
	$$
	This together with (\ref{eq2-lem3Hardy}) implies that
	$$
	\|f\|^p_{h^{p,q}_{at,\rho,\epsilon}(X)}\lesi \epsilon^{\kappa}\|f\|^p_{h^{p,q}_{at,\rho,\epsilon}(X)}+\Big\|\sup_{0<t<\widetilde{\epsilon}\rho(\cdot)^2}|T_{t,\rho}f(\cdot)|\,\Big\|^p_{L^p(X)}.
	$$
	Therefore, there exists $\epsilon_0$ so that
	\[
	\|f\|_{h^{p,q}_{at,\rho,\epsilon_0}(X)}\lesi \Big\|\sup_{0<t<\rho(\cdot)^2}|T_{t,\rho}f(\cdot)|\,\Big\|^p_{L^p(X)}.
	\]
	On the other hand, from the expression of $T_{t,\rho}$ we have
	\[
	\sup_{0<t<\rho(x)^2}|T_{t,\rho}f(x]|\leq f^+_{\LL,\rho}(x), \ \ \text{for all $x\in X$}.
	\]
	Therefore,
	$$
	\|f\|_{h^{p,q}_{at,\rho,\epsilon_0}(X)}\lesi \Big\|f^+_{\LL,\rho}\Big\|_{L^p(X)}=\|f\|_{h^p_{\LL,{\rm rad}, \rho}(X)}.
	$$
	This completes our proof of Theorem \ref{mainthm2s}.
\end{proof}

\begin{rem}
	\label{rem-YZ}
	Assume that the measure $\mu$ satisfies the extra condition of `reverse doubling'. In \cite{YZ} the authors characterized the local Hardy spaces $h^{1,q}_{at,\rho}, q\in [1,\vc]$ in terms of radial maximal functions
	\[
	S_{\rho}^+f(x):=\sup_{k\in\mathbb{Z}, 2^{-k}<\rho(x)}S_kf(x)
	\]
	where $\{S_k\}_{k\in\mathbb{Z}}$ is an approximation of the identity. See \cite[Theorem 2.1]{YZ}. By replacing the semigroup $\{e^{-t\LL}\}$ by the family $\{S_k\}_{k\in\mathbb{Z}}$, our approach can be adapted easily to give the  radial maximal function $S_{\rho}^+$ characterization for the local Hardy spaces $h^{p,q}_{at,\rho}$ with $p,q$ as in Theorem \ref{mainthm2s}. We leave the details to the interested reader.
\end{rem}
\bigskip

\subsection{Proof of Theorem \ref{mainthm3}} The proof of Theorem \ref{mainthm3} is quite similar to that of Theorem \ref{mainthm2s} and hence we just sketch the main points. We first prove the following estimates which is similar to that in Lemma \ref{Lem3: Hardy spaces}.

\begin{lem}\label{Lem3: Hardy spaces-thm3}
	Let $\f{n}{n+\delta_2\wedge \delta_3}<p\leq 1$ and $q\in (p,\vc]\cap [1,\vc]$. Then there exists $\kappa$ so that for any $0<\epsilon\leq 1$, we have
	\begin{equation}\label{eq1-Lem3Hardy}
	\Big\|\sup_{0< t\leq [\epsilon \rho(x)]^2}|(e^{-t\LL}-e^{-tL})f(x)|\,\Big\|^p_{L^p(X)}\lesi \epsilon^{\kappa}\|f\|^p_{h^{p,q}_{at,\rho,\epsilon}(X)},
	\end{equation}
	for all $f\in h^{p,q}_{at,\rho,\epsilon}(X)$.
\end{lem}
\begin{proof}
	The proof is completely analogous to that of Lemma \ref{Lem3: Hardy spaces} with a minor modification of using (B2) and (B3) in place of \eqref{eq1-Q t rho} and \eqref{eq2-Q t rho}, respectively.
\end{proof}

We now turn to the proof of Theorem \ref{mainthm3}.

\begin{proof}[Proof of Theorem \ref{mainthm3}:]
As in the proof of Theorem \ref{mainthm2s}, we first prove  $h^{p,q}_{at,\rho}(X)\cap L^2(X)\hookrightarrow H^p_{L, {\rm max}}(X)$. We note that the maximal operator $\mathcal{M}_{\max, L}$ is dominated by the Hardy-Littlewood maximal function $\mathcal{M}f$. This fact along with Lemma \ref{lem-Tatom} reduces our task to showing 
\begin{align}\label{maximal atoms-thm3}
\|\mathcal{M}_{\max, L}a\|^p_{L^p}\leq C,
\end{align}
for some uniform constant $C$ and any $(p,q,\rho)$-atom $a$ associated to a ball $B=B(x_0,r)$.

We write
$$
\|\mathcal{M}_{\max, L}a\|^p_{L^p}\leq \|\mathcal{M}_{\max, L}a\|^p_{L^p(4B)}+\|\mathcal{M}_{\max, L}a\|^p_{L^p(M\backslash 4B)}: =I_1+I_2.
$$
The first term can be estimated exactly the same as the term $I_1$ in the proof of Theorem \ref{mainthm2s}.
 
For the term $I_2$ we consider two cases.

\textbf{Case 1: $\rho(x_0)/4 \leq r\leq \rho(x_0)$.} From (B1) and the fact that $r\sim \rho(x_0)\sim \rho(z)$ for $z\in B$ and $d(x,y)\sim d(x,x_0)$ for $y,x_0\in B$ and $x\in (4B)^c$, we have, for $N>n(1-p)/p$,
$$
\begin{aligned}
I_2&\lesi \int_{M\backslash 4B}\sup_{t>0}\sup_{d(x,y)<t}\Big[\int_B \f{1}{\mu(B(z,\sqrt{t}))}\exp\Big(-\f{d(y,z)^2}{ct}\Big)\Big(\f{\rho(z)}{\sqrt{t}}\Big)^N|a(z)|d\mu(z)\Big]^pd\mu(x)\\
&\lesi \int_{M\backslash 4B}\sup_{t>0}\sup_{d(x,y)<t}\Big[\int_B \f{1}{\mu(B(z,\sqrt{t}))}\exp\Big(-\f{d(y,z)^2}{ct}\Big)\Big(\f{\rho(x_0)}{\sqrt{t}}\Big)^N|a(z)|d\mu(z)\Big]^pd\mu(x).
\end{aligned}
$$
Due to
\[
\exp\Big(-\f{d(y,z)^2}{ct}\Big)\sim \exp\Big(-\f{d(x,z)^2}{ct}\Big), \ \ \text{as $d(x,y)<t$},
\]
we have
$$
\begin{aligned}
I_2&\lesi \int_{X\backslash 4B}\sup_{t>0}\Big[\int_B \f{1}{\mu(B(z,\sqrt{t}))}\exp\Big(-\f{d(x,z)^2}{t}\Big)\Big(\f{\rho(x_0)}{\sqrt{t}}\Big)^N|a(z)|d\mu(z)\Big]^pd\mu(x)\\
&\lesi \int_{X\backslash 4B}\sup_{t>0}\Big[\int_B \f{1}{\mu(B(z,d(x,z)))}\exp\Big(-\f{d(x,z)^2}{t}\Big)\Big(\f{\rho(x_0)}{\sqrt{t}}\Big)^N|a(z)|d\mu(z)\Big]^pd\mu(x).
\end{aligned}
$$
This along with the fact that $d(x,z)\sim d(x,x_0)$ implies
$$
\begin{aligned}
I_2
&\lesi \int_{X\backslash 4B}\sup_{t>0}\Big[\int_B \f{1}{\mu(B(x,d(x,x_0)))}\exp\Big(-\f{d(x,x_0)^2}{t}\Big)\Big(\f{\rho(x_0)}{\sqrt{t}}\Big)^N|a(z)|d\mu(z)\Big]^pd\mu(x)\\
&\lesi \int_{X\backslash 4B}\Big[\int_B \f{1}{\mu(B(x,d(x,x_0)))}\Big(\f{r}{d(x,x_0)}\Big)^N|a(z)|d\mu(z)\Big]^pd\mu(x)\\
&\leq C.
\end{aligned}
$$

\textbf{Case 2: $r<\rho(x_0)/4$.}

We now break $I_2$ into 2 terms as follows:
$$
\begin{aligned}
I_2&\lesi \int_{X\backslash 4B}\sup_{0<t\leq 4r^2}\sup_{d(x,y)<t}\Big|\int_B p_t(y,z)a(z)d\mu(z)\Big|^pd\mu(x)\\
& \ \ + \int_{X\backslash 4B}\sup_{ 4r^2\le t}\sup_{d(x,y)<t}\Big|\int_B p_t(y,z)a(z)d\mu(z)\Big|^pd\mu(x)\\
&=I_{21}+I_{22}.
\end{aligned}
$$
The remaining parts can be done in the same manner as those in the proof of Theorem \ref{mainthm2s} using (B1) and \eqref{Holder ptxy} in place of (A2) and (A3), and so we omit the details. This completes the proof of \eqref{maximal atoms-thm3}.

Due to the fact that $H^p_{L, \max}(X)\subset H^p_{L, \rad}(X)$, it remains to verify that $H^p_{L, \rad}(X)\cap L^2(X)\hookrightarrow  h^{p,q}_{at,\rho}(X)$. This part can be done mutatis mutandis as in the proof of Theorem \ref{mainthm2s} by replacing $T_{t,\rho}$ and Lemma \ref{Lem3: Hardy spaces} by $e^{-tL}$ and Lemma \ref{Lem3: Hardy spaces-thm3}, respectively.
\end{proof}

\section{Some applications}\label{sec: applications}

We now give some applications to the main results. The list is not exhaustive but is intended to
show the variety of possible
applications and the generality of our assumptions.
\subsection{Schr\"odinger operators on noncompact Riemannian manifolds}
Let $X$ be a complete connected Riemannian manifold, $\mu$ be the
Riemannian measure and $\nabla$ be the Riemannian gradient. Let
$-\Delta$ be the Laplace-Beltrami operator. It is well-known that $-\Delta$ satisfies (A4). The geodesic distance between $x\in X$ and $y\in X$ will be denoted by $d(x,y)$. Denote by $B(x, r)$ the open ball
of radius $r > 0$ and center $x \in X$. Assume that the measure $\mu$
satisfies the doubling property, that is, there exists a constant
$C>0$ and $n\geq 0$ so that
\begin{equation}\label{doubling-manifold}
\mu(B(x,\lambda r))\leq C\lambda^n \mu(B(x,r)),
\end{equation}
for all $x\in X, r>0$ and $\lambda\geq 1$.

We also assume that $X$ admits a Poincar\'e inequality.  That is, there exists a constant $C>0$ such that for every function $f\in C_0^{\infty}(X)$ and every ball $B\subset X$, we have
\begin{equation}\label{Poincare-p}
\begin{aligned}
\Big(\fint_B|f-f_B|^2\,d\mu\Big)^{1/2} \le C r_B \Big(\fint_B|\nabla f|^2\,d\mu\Big)^{1/2}.
\end{aligned}
\end{equation}

Denote by $\widetilde{p}_t(x,y)$ the associated kernel to the semigroup $e^{t\Delta}$. 
It is well-known that  the doubling condition \eqref{doubling-manifold} and the Poincar\'e inequality \eqref{Poincare-p} imply Gaussian  and  H\"older continuity estimates for $-\Delta$. More precisely, there exist $C,c>0$ and $\delta_1$ so that 
\begin{equation}\label{eq-G manifold}
0\leq \widetilde{p}_t(x,y)\leq \f{C}{\mu(B(x,\sqrt{t}))}\exp\Big(-\f{d(x,y)^2}{c t}\Big)
\end{equation}
for all $t>0, x,y\in X$, and
\begin{equation}\label{eq-H manifold}
|\widetilde{p}_t(x,y)-\widetilde{p}_t(x',y)|\leq C\Big(\f{d(x,x')}{\sqrt{t}}\Big)^{\delta_1}\f{1}{\mu(B(x,\sqrt{t}))}\exp\Big(-\f{d(x,y)^2}{ct}\Big),
\end{equation}
for all $t>0$ and $d(x,x')<(d(x,y)+\sqrt{t})/2$. See for example \cite{Sc1, Sc2}.

Note that the conditions \eqref{eq-G manifold} and \eqref{eq-H manifold} imply that for any $\delta\in (0,\delta_1]$ we have
\begin{equation}\label{eq2-H manifold}
|\widetilde{p}_t(x,y)-\widetilde{p}_t(x',y)|\leq C\Big(\f{d(x,x')}{\sqrt{t}}\Big)^{\delta}\f{1}{\mu(B(y,\sqrt{t}))}\Big[\exp\Big(-\f{d(x,y)^2}{ct}\Big)+\exp\Big(-\f{d(x',y)^2}{ct}\Big)\Big],
\end{equation}
for all $t>0$ and $x,x',y\in X$.

Let $\rho$ be a critical function on $X$. Hence we may apply Theorem \ref{mainthm2s} to $\LL=-\Delta$  to obtain

\begin{thm}
 For $p\in (\f{n}{n+\delta_1},1]$ and $q\in [1,\vc]\cap (p,\vc]$, we have
\[
h^{p,q}_{at,\rho}(X)\equiv  h^{p}_{ -\Delta,{\rm max}, \rho}(X)\equiv  h^{p}_{-\Delta,{\rm rad},  \rho}(X).
\] 
\end{thm}
This result is new even for $p=1$. Moreover, in the particular case $\rho\equiv 1$, we recover the result in \cite{YZ2}.


We now consider consider a Schr\"odinger operator $L=-\Delta+V$ where $V\in A_\vc\cap RH_\sigma$ with $\sigma>\max\{1,n/2\}$. See Subsection 6.1 for the definitions of the class $A_\vc$ and $RH_q$. Following the idea in \cite{Sh} we define the critical function $\rho$ on $X$ by setting
\begin{equation}\label{rhofunction}
\rho(x)=\sup\Big\{r>0: \f{r^2}{\mu(B(x,r))}\int_{B(x,r)}V(y)d\mu(y)\leq 1\Big\}.
\end{equation}
We then have the following result.
\begin{thm}
	\label{thm-Schrodinger manifolds} 
	Let $(X,d,\mu)$ satisfy \eqref{doubling-manifold} and \eqref{Poincare-p}. Let $L=-\Delta+V$ where $V\in A_\vc\cap RH_q$ with $q>\max\{1,n/2\}$, and let $\rho$ be defined as in \eqref{rhofunction}. Then $L$ satisfies (B1)-(B3) with $\LL=-\Delta$, $\delta_2=2-n/q$ and any $\delta_3<\min\{\delta_1,\delta_2\}$.
\end{thm}
The proof of this theorem is quite long and will be given in Subsection \ref{subsect-proof of 6.2}. More precisely, the proof of (B1), (B2) and (B3) will be addressed in Propositions \ref{Proposition1: heat kernel bound}, \ref{Proposition2: heat kernel bound} and \ref{Proposition3: heat kernel bound}, respectively.\\

As a direct consequence of Theorem \ref{mainthm3} and Theorem \eqref{thm-Schrodinger manifolds} we obtain:
\begin{thm}
	Assume that $(X,d,\mu)$ satisfy \eqref{doubling-manifold} and \eqref{Poincare-p}. Let $L=-\Delta +V$ with $V\in A_\vc\cap RH_q$ with $q>\max\{1,n/2\}$. Let $p\in (\f{n}{n+\delta_0},1]$  and $q\in [1,\vc]\cap(p,\vc]$ with $\delta_0=\min\{\delta_1,2-n/q\}$. Then we have
	\[
	h^{p,q}_{at,\rho}(X)\equiv H^p_{L, {\rm max}}(X)\equiv H^p_{L, \rad}(X).
	\] 
\end{thm}
The result in this theorem is new even for $p=1$.
\subsection{Laguerre operators}

Let $X=((0,\vc)^m, d\mu(x))$ where 
$d\mu(x)=d\mu_1(x_1)\ldots d\mu_m(x_m)$
and $d\mu_k = x_k^{\alpha_k}dx_k$, $\alpha_k >-1$, for 
$k=1,\ldots, m$ ($dx_{j}$ being the one dimensional Lebesgue measure).
We endow $\mathcal{X}$ with the distance $d$ defined for
$x=(x_1,\ldots,x_m)$ and $y=(y_1,\ldots,y_m) \in X$ as
\begin{equation*}
d(x,y):=|x-y|=\Big(\sum_{k=1}^m|x_k-y_k|^2\Big)^{1/2}.
\end{equation*}
Then it is clear that 
\begin{equation}\label{volume-Bessel}
\mu(B(x,r))\sim r^m
\prod_{k=1}^m (x_k+r)^{\alpha_k}.
\end{equation}
Note that this estimate implies the doubling property \eqref{doub2} with $n=m+\alpha_1+\ldots+\alpha_m$.

For an element $x\in \mathbb{R}^m$, unless specified otherwise,
we shall write $x_k$ for the $k$-th component of $x$, 
$k=1,\ldots,m$. Moreover, for $\lambda\in \mathbb{R}^m$, 
we write $\lambda^2=(\lambda^2_1,\ldots, \lambda_m^2)$.

We consider the second order Bessel differential operator
\begin{align}\label{bessel}
\LL=-\Delta-\sum_{k=1}^m \f{\alpha_k}{x_k}\f{\partial}{\partial x_k}
\end{align}
whose system of eigenvectors is defined by
$$
E_\lambda(x):= \prod_{k=1}^mE_{\lambda_k}(x_k),
\quad
E_{\lambda_k}(x_k):=
(x_k\lambda_k)^{-(\alpha_k-1)/2}J_{(\alpha_k-1)/2}
(x_k\lambda_k),
\quad
\lambda, x\in \mathcal{X}
$$
where $J_{(\alpha_k-1)/2}$ is the Bessel function of the first 
kind of order $(\alpha_k-1)/2$ (see \cite{L}). 
It is known that $\LL(E_\lambda)=|\lambda|^2 E_{\lambda}$. 
Moreover, the functions $E_{\lambda_k}$ are  eigenfunctions 
of the one-dimension Bessel operators
$$
\LL_k=-\f{\partial^2}{\partial x_k\,^2}-\f{\alpha_k}
{x_k}\f{\partial}{\partial x_k}
$$
and indeed $\LL_k(E_{\lambda_k})=\lambda_k^2 E_{\lambda_k}$ for 
$k=1,\ldots, m$.

It is well known that $\LL$ satisfies (A1)-(A4) with $\delta_1=1$. See for example \cite{L}. Hence, as a consequence of Theorem \ref{mainthm2s} we have:
\begin{thm}
	Let $\LL$ be the Bessel operator defined in \eqref{bessel} and $\rho$ be a critical function on $X$. If $p\in (\f{n}{n+1},1]$ and $q\in [1,\vc]\cap (p,\vc]$ then we have
		\[
		h^{p,q}_{at,\rho}(X)\equiv  h^p_{\LL,\max,\rho}(X)\equiv  h^p_{\LL,\rad,\rho}(X).
		\]
\end{thm}

We next consider the Laguerre operator defined by
\begin{equation}\label{eq-Laguerre}
L:=
\sum_{k=1}^m\LL_k+|x|^2 =\LL+|x|^2.
\end{equation}
It is well-known that the heat kernel $p_t(x,y)$ associated to the semigroup $e^{-tL}$ is given by
\begin{equation}
\label{eq-heatkernelLaguerre}
p_t(x,y)=\prod_{j=1}^m \f{2e^{-2t}}{1-e^{-4t}}\exp{\Big(-\f{1}{2}\f{1+e^{-4t}}{1-e^{-4t}}(x_j^2+y_j^2)\Big)}(x_jy_j)^{-(\alpha_j-1)/2}I_{(\alpha_j-1)/2}\Big(\f{2e^{-2t}}{1-e^{-4t}}x_jy_j\Big),
\end{equation}
for all $t>0$, $x,y\in X$ and $I_{(\alpha_j-1)/2}$ being the Bessel function. See for example \cite{L}.

We define the critical function $\rho$ on $X$ by setting
\begin{equation}\label{rhofunction-Lague}
\rho(x)=\sup\Big\{r>0: \f{r^2}{\mu(B(x,r))}\int_{B(x,r)}|y|^2d\mu(y)\leq 1\Big\}.
\end{equation}
Then by a simple calculation we can find that
\begin{equation}\label{eq-rho laguer}
\rho(x)\sim \min\{1, |x|^{-1}\}.
\end{equation}

We have the following result.
\begin{thm}
	Let $L$ be a Laguerre operator defined in \eqref{eq-Laguerre}. Then $L$ satisfies (B1)-(B3) with $\LL$ from \eqref{bessel} and any $\delta_2=1$, any $\delta_3<1$ and with $\rho$ defined in \eqref{rhofunction-Lague}.
\end{thm}
\begin{proof}
	We only prove that $L$ satisfies (B1). Once this is proved, arguing similarly to the proof of Theorem \ref{thm-Schrodinger manifolds} we can show that $L$ satisfies (B2)-(B3).

We first recall some basic properties of Bessel functions $I_\nu, \nu>-1$. It is well known that
\begin{equation}\label{eq1-Bessel}
z^{-\nu}I_\nu(z)\sim 2^{-\nu}, z\in (0,1];
\end{equation}
\begin{equation}\label{eq2-Bessel}
I_\nu(z)\lesi z^{-1/2}e^z, z\geq 1;
\end{equation}
and
\begin{equation}\label{eq3-Bessel}
\f{d}{dz}(z^{-\nu}I_\nu(z))=z^{-\nu}I_{\nu+1}(z).
\end{equation}
See for example \cite{L}.

Due to \eqref{eq-heatkernelLaguerre} it suffices to prove (B1) for the one dimensional case $m=1$. More precisely we  claim that for all $N>0$, there exist positive constants $c$ and $C$ so that
\begin{equation}
\label{B1-Lague}
|p_t(x,y)|\leq \f{C}{\mu(B(x,\sqrt{t}))}\exp\Big(-\f{d(x,y)^2}{ct}\Big)\Big(1+\f{\sqrt{t}}{\rho(x)}+\f{\sqrt{t}}{\rho(y)}\Big)^{-N}
\end{equation}
for all $x,y\in X$ and $t>0$. Here $d\mu=x^\alpha dx$ for $\alpha>-1$ and 
\[
p_t(x,y)=\f{2e^{-2t}}{1-e^{-4t}}\exp{\Big(-\f{1}{2}\f{1+e^{-4t}}{1-e^{-4t}}(x^2+y^2)\Big)}(xy)^{-(\alpha-1)/2}I_{(\alpha-1)/2}\Big(\f{2e^{-2t}}{1-e^{-4t}}xy\Big).
\]

	Setting $s=\f{1-e^{-4t}}{2e^{-2t}}$, we rewrite
	\[
	p_t(x,y)=\f{1}{s}\exp{\Big(-\f{1}{2}\f{1+e^{-4t}}{1-e^{-4t}}(x^2+y^2)\Big)}(xy)^{-(\alpha-1)/2}I_{(\alpha-1)/2}\Big(\f{xy}{s}\Big).
	\]
	
	We consider two cases.
	
	\textbf{Case 1. $xy<s$}
	
	In this situation, we have $x<\sqrt{s}$ or $y<\sqrt{s}$. Without the loss of generality, we may assume that $x<\sqrt{s}$ and hence $\mu(B(x,\sqrt{s}))\sim s^{(\alpha+1)/2}$.
	
	Moreover, by (\ref{eq1-Bessel}),
	$$
	(xy)^{-(\alpha-1)/2}I_{(\alpha-1)/2}\Big(\f{xy}{s}\Big)\sim s^{-(\alpha-1)/2}.
	$$
	Hence,
	\begin{equation}\label{eq1-proof B1}
	\begin{aligned}
	p_t(x,y)&\lesi \f{s^{-(\alpha-1)/2}}{s}\exp{\Big(-\f{1}{2}\f{1+e^{-4t}}{1-e^{-4t}}(|x|^2+|y|^2)\Big)}\\
	&\lesi \f{1}{\mu(B(x,\sqrt{s}))}\exp{\Big(-\f{1}{4}(|x|^2+|y|^2)\Big)}.
	\end{aligned}
	\end{equation}
	
	On the other hand, we also have
	\begin{equation}\label{eq2-proof B1}
	\begin{aligned}
	p_t(x,y)&\lesi \f{s^{-(\alpha-1)/2}}{s}\exp{\Big(-\f{1}{2}\f{1+e^{-4t}}{1-e^{-4t}}(|x|^2+|y|^2)\Big)}\\
	&\lesi \f{1}{\mu(B(x,\sqrt{s}))}\exp{\Big(-\f{1}{2}\f{1+e^{-4t}}{1-e^{-4t}}(|x-y|^2)\Big)}.
	\end{aligned}
	\end{equation}
	From \eqref{eq1-proof B1} and \eqref{eq2-proof B1} we conclude that
\begin{equation}\label{eq3-proof B1}
	p_t(x,y)\lesi \f{1}{\mu(B(x,\sqrt{s}))}\exp{\Big(-\f{1}{4}\f{1+e^{-4t}}{1-e^{-4t}}(|x-y|^2)\Big)}\exp{\Big(-\f{1}{8}(|x|^2+|y|^2)\Big)}.
	\end{equation}
	
	If $0<t\leq 1$, then $1+e^{-4t}\sim 1$, $1-e^{-4t}\sim t$ and $s\sim t$. This together with \eqref{eq-rho laguer} and \eqref{eq3-proof B1} yields, for any $N>0$,
	\[
	\begin{aligned}
	p_t(x,y)\lesi \f{1}{\mu(B(x,\sqrt{t}))}\exp{\Big(-\f{|x-y|^2}{ct}\Big)}\exp{\Big(-\f{1}{8}(|x|^2+|y|^2)\Big)}\\
	\lesi \f{1}{\mu(B(x,\sqrt{t}))}\exp{\Big(-\f{|x-y|^2}{ct}\Big)}\Big(1+\f{\sqrt{t}}{\rx}+\f{\sqrt{t}}{\rx}\Big)^{-N}.
	\end{aligned}
	\]
	If $t> 1$ then $1+e^{-4t}\sim 1$, $1-e^{-4t}\sim 1$ and $s\sim e^{2t}>te^t$. This along with  \eqref{eq3-proof B1} implies
	\begin{equation}\label{eq4-proof B1}
	\begin{aligned}
	p_t(x,y)\lesi \f{1}{\mu(B(x,\sqrt{te^t}))}\exp{\Big(-c|x-y|^2\Big)}\exp{\Big(-\f{1}{8}(|x|^2+|y|^2)\Big)}\\
	\lesi \f{1}{\mu(B(x,\sqrt{te^t}))}\exp{\Big(-\f{|x-y|^2}{ct}\Big)}\exp{\Big(-\f{1}{8}(|x|^2+|y|^2)\Big)}
	\end{aligned}
	\end{equation}
	
	Moreover, we can see that
	\begin{equation}\label{reverse doubling Lague}
	\lambda^\kappa \mu(B(x,r))\leq C\mu(B(x,\lambda r)) 
	\end{equation}
	for all $x\in X, r>0$ and $\lambda >1$, where $\kappa = \min\{1, 1+\alpha\}$.
	This, in combination with \eqref{eq4-proof B1}, implies
	\[
	p_t(x,y)
	\lesi \f{e^{-\kappa t/2}}{\mu(B(x,\sqrt{t}))}\exp{\Big(-\f{|x-y|^2}{ct}\Big)}\exp{\Big(-\f{1}{8}(|x|^2+|y|^2)\Big)}.
	\]
	This and \eqref{eq-rho laguer} gives \eqref{B1-Lague}.
	\\

	\textbf{Case 2. $xy\geq s$} \ \
	
	By (\ref{eq2-Bessel}), we have
	\begin{equation}\label{eq5-proof B1}
	\begin{aligned}
	p_t(x,y)&\lesi \f{1}{s}\exp{\Big(-\f{1}{2}\f{1+e^{-4t}}{1-e^{-4t}}(|x-y|^2)\Big)}\exp{\Big(-\f{1+e^{-4t}}{1-e^{-4t}}xy\Big)}(xy)^{-(\alpha-1)/2}\Big(\f{s}{xy}\Big)^{1/2}\exp\Big(\f{xy}{s}\Big)\\
	&=\f{1}{\sqrt{s}}\exp{\Big(-\f{1}{2}\f{1+e^{-4t}}{1-e^{-4t}}(|x-y|^2)\Big)}\exp{\Big[\Big(-\f{1+e^{-4t}}{1-e^{-4t}}+\f{1}{s}\Big)xy\Big]}(xy)^{-\alpha/2}\\
	&=\f{1}{\sqrt{s}}\exp{\Big(-\f{1}{2}\f{1+e^{-4t}}{1-e^{-4t}}(|x-y|^2)\Big)}\exp{\Big(-\f{1-e^{-2t}}{1+e^{-2t}}xy\Big)}(xy)^{-\alpha/2}.
	\end{aligned}
	\end{equation}
	Moreover,
	\[
	\begin{aligned}
	\exp{\Big(-\f{1}{2}\f{1+e^{-4t}}{1-e^{-4t}}(|x-y|^2)\Big)}&\exp{\Big(-\f{1-e^{-2t}}{1+e^{-2t}}xy\Big)}\\
	&\leq \exp{\Big(-\f{1}{2}\f{1-e^{-2t}}{1+e^{-2t}}(|x-y|^2)\Big)}\exp{\Big(-\f{1-e^{-2t}}{1+e^{-2t}}xy\Big)}\\
	&=\exp{\Big(-\f{1-e^{-2t}}{2(1+e^{-2t})}(|x|^2+|y|^2)\Big)}.
	\end{aligned}
	\]
	This along with \eqref{eq5-proof B1} implies
	\begin{equation}\label{eq6-proof B1}
	\begin{aligned}
	p_t(x,y)\lesi \f{1}{\sqrt{s}}&\exp{\Big(-\f{1+e^{-4t}}{4(1-e^{-4t})}(|x-y|^2)\Big)}\exp{\Big(-\f{1-e^{-2t}}{2(1+e^{-2t})}xy\Big)}(xy)^{-\alpha/2}\\
	&\times \exp{\Big(-\f{1-e^{-2t}}{2(1+e^{-2t})}(|x|^2+|y|^2)\Big)}.
	\end{aligned}
	\end{equation}
	
	We consider three subcases.\\

	\noindent \emph{Subcase 2.1: $x, y\geq \sqrt{s}$.} In this situation, we have $\mu(B(x,\sqrt{s}))\sim x^\alpha\sqrt{s}$ and $\mu(B(y,\sqrt{s}))\sim y^\alpha\sqrt{s}$.
	
	If $0<t\leq 1$, we find that
	\begin{equation*}
	\begin{aligned}
	p_t(x,y)&\lesi \f{1}{\sqrt{t(xy)^{\alpha}}}\exp{\Big(-\f{|x-y|^2}{ct}\Big)}\exp\Big(-c(t|x|^2+t|y|^2)\Big)\\
	&\lesi \f{1}{\Big[\mu_\alpha(B(x,\sqrt{t}))\mu_\alpha(B(y,\sqrt{t}))\Big]^{1/2}}\exp{\Big(-c\f{|x-y|^2}{t}\Big)}\exp\Big(-c(t|x|^2+t|y|^2)\Big).
	\end{aligned}
	\end{equation*}
	This proves \eqref{B1-Lague}.\\

	If $t\geq 1$, similarly we have
	$$
	\begin{aligned}
	p_t(x,y)&\lesi \f{1}{e^{2t}(xy)^{\alpha/2}}\exp{\Big(-\f{|x-y|^2}{ct}\Big)}\exp\Big(-c(|x|^2+|y|^2)\Big)\\
	&\lesi \f{1}{e^t\, \Big[\mu_\alpha(B(x,\sqrt{t}))\mu_\alpha(B(y,\sqrt{t}))\Big]^{1/2}}\exp{\Big(-\f{|x-y|^2}{ct}\Big)}\exp\Big(-c(|x|^2+|y|^2)\Big).
	\end{aligned}
	$$
	This implies  \eqref{B1-Lague}.\\

	\noindent \emph{Subcase 2.2: $x\geq \sqrt{s}\geq y$}
	
	If $-1<\alpha\leq 0$ then
	$$
	\exp{\Big(-\f{1-e^{-2t}}{1+e^{-2t}}xy\Big)}(xy)^{-\alpha/2}\lesi \Big(\f{1-e^{-2t}}{1+e^{-2t}}\Big)^{\alpha/2}.
	$$
	Substituting into \eqref{eq6-proof B1} we get that
	\[
	p_t(x,y)\lesi \f{1}{\sqrt{s}}\exp{\Big(-\f{1+e^{-4t}}{4(1-e^{-4t})}(|x-y|^2)\Big)}\Big(\f{1-e^{-2t}}{1+e^{-2t}}\Big)^{\alpha/2} \exp{\Big(-\f{1-e^{-2t}}{2(1+e^{-2t})}(|x|^2+|y|^2)\Big)}.
	\]
	At this stage, using the same argument as above we conclude \eqref{B1-Lague}.
		
	If $\alpha>0$, then
	$$
	\exp{\Big(-\f{1-e^{-2t}}{1+e^{-2t}}xy\Big)}(xy)^{-\alpha/2}\leq (xy)^{-\alpha/2}\leq s^{-\alpha/2}.
	$$
	Inserting into \eqref{eq6-proof B1} and using the argument as above, we also obtain the desired estimate.\\

	\noindent \emph{Subcase 2.3: $y\geq \sqrt{s}\geq x$}
	
	This subcase can be done in the same manner as in Subcase 2.2 and we omit details.
\end{proof}	

From the above result and Theorem \ref{mainthm3} we imply:
\begin{thm}
		Let $L$ be a Laguerre operator defined in \eqref{eq-Laguerre} and $\rho$ be a critical function as in \eqref{eq-rho laguer}. Let $p\in (\f{n}{n+1},1]$ and $q\in [1,\vc]\cap(p,\vc]$. Then we have
		\[
		h^{p,q}_{at,\rho}(X)\equiv H^p_{L, {\rm max}}(X)\equiv H^p_{L, \rad}(X).
		\]
\end{thm}	
Note that the particular case $m=1$ and $\alpha \geq 0$ was obtained in \cite{D2}. Hence, the theorem is new even for the case $m=1$.

\subsection{Degenerate Schr\"odinger operators}
Let $w$ be a weight in Muckenhoupt class $A_2(\mathbb{R}^d), d\geq 3$. That is,  there exist a constant $C>0$ so that
$$
\Big(\f{1}{|B|}\int_Bw(x)dx\Big)\Big(\f{1}{|B|}\int_Bw^{-1}(x)dx\Big)\leq C
$$
for all balls $B\subset \mathbb{R}^d$. Then the triple $(X,d,d\mu)=\big(\RR{d}, |\cdot|, wdx\big)$ satisfies (\ref{doublingcondition}). Moreover, there exist $0<\kappa\leq n<\vc$ so that
$$
\lambda^\kappa w(B(x,r))\lesi w(B(x,\lambda r))\lesi \lambda^n w(B(x,r))
$$
for all $x\in \mathbb{R}^d, r>0$ and $\lambda\geq 1$, where $w(E)=\int_{E}w(x)dx$ for any measurable subset $E\subset \mathbb{R}^d$.

Let $\{a_{i, j}\}_{i,j=1}^d$ be a real symmetric matrix function satisfying, for some $C>0$ and every
$x, \xi \in \mathbb{R}^d$,
$$
C^{-1}|\xi|^2 w(x)\leq \sum_{i,j}a_{i,j}(x)\xi_i \overline{\xi}_j\leq C|\xi|^2w(x).
$$
We consider the degenerate elliptic operator $\LL$ defined by
\begin{align}\label{degen}
\LL f(x)=-\f{1}{w(x)}\sum_{i,j}\partial_i(a_{i,j}(\cdot)\partial_j f)(x).
\end{align}
Then the operator $\LL$ satisfies the assumptions (A1)-(A4) with some $\delta_1\in (0,1)$. See for example \cite{HS}. 

Let $\rho$ be a critical function on $\mathbb{R}^d$. For $p\in (0,1]$ we define the Hardy space $h^p_{\LL,{\rm rad}, \rho}(\mathbb{R}^d, |\cdot|, w(x)dx)$ as the completion of 
\[
\left\{f\in L^2(\mathbb{R}^d, |\cdot|, w(x)dx): f^+_{\LL,\rho}\in L^p(\mathbb{R}^d, |\cdot|, w(x)dx)\right\}
\]
with respect to the norm
\[
\|f\|_{h^p_{\LL,{\rm rad}, \rho}(\mathbb{R}^d, |\cdot|, w(x)dx)}:=\left\|f^+_{\LL,\rho}\right\|_{L^p(\mathbb{R}^d, |\cdot|, w(x)dx)}.
\]
Similarly, the Hardy space $h^p_{\LL,{\rm max},\rho}(\mathbb{R}^d, |\cdot|, w(x)dx)$ as the completion of 
\[
\left\{f\in L^2(\mathbb{R}^d, |\cdot|, w(x)dx): f^*_{\LL,\rho}\in L^p(\mathbb{R}^d, |\cdot|, w(x)dx)\right\}
\]
with respect to the norm
\[
\|f\|_{h^p_{\LL,{\rm max},\rho}(\mathbb{R}^d, |\cdot|, w(x)dx)}:=\left\|f^+_{\LL,\rho}\right\|_{L^p(\mathbb{R}^d, |\cdot|, w(x)dx)}.
\]

Hence  Theorem \eqref{mainthm2} implies the following.
\begin{thm}
	Let $\LL$ be the operator in \eqref{degen} and $\rho$ be a critical function on $\mathbb{R}^d$. Let $p\in (\f{n}{n+\delta_1},1]$ and $q\in [1,\vc]\cap (p,\vc]$. Then we have
	\[
	h^{p,q}_{at,\rho}\big(\mathbb{R}^d, |\cdot|, wdx\big)\equiv h^p_{\LL,\max,\rho}\big(\mathbb{R}^d, |\cdot|, wdx\big)\equiv h^p_{\LL,\rad,\rho}\big(\mathbb{R}^d, |\cdot|, wdx\big).
	\]
\end{thm}

Let $L=\LL+V$ be a so-called degenerate Schr\"odinger operator with $V\in RH_{q}(\mathbb{R}^d, |\cdot|, w(x)dx)$ with $q>n/2$. We  define the critical function $\rho$ by setting
\begin{equation}\label{eq-rho degenerate}
\rho(x)=\sup\Big\{r>0: \f{r^2}{w(B(x,r))}\int_{B(x,r)}V(y)w(y)dy\leq 1\Big\}.
\end{equation}

It was proved that the degenerate Schr\"odinger operator $L$ satisfies the conditions (B1) and (B2) with $\delta_2= 2-n/q$. See for example \cite{D3, YZ}. By a similar argument to that in the proof of Proposition \ref{Proposition3: heat kernel bound}, we conclude that $L$ satisfies (B3) for any $\delta_3<\min\{\delta_1, \delta_2\}$. 
Therefore Theorem \ref{mainthm3} may be applied to deduce the following result.
\begin{thm}
	Let $L=\LL+V$ be a Schr\"odinger operator with $\LL$ from \eqref{degen} and $V\in RH_{q}(X)$ for some $q>n/2$. Let $\rho$ be defined as in \eqref{eq-rho degenerate}. If $p\in (\f{n}{n+\delta_0},1]$ and $q\in [1,\vc]\cap(p,\vc]$ with $\delta_0=\min\{\delta_1,2-n/q\}$ then we have
	\[
	h^{p,q}_{at,\rho}\big(\mathbb{R}^d, |\cdot|, wdx\big)\equiv H^p_{L, {\rm max}}\big(\mathbb{R}^d, |\cdot|, wdx\big)\equiv H^p_{L, \rad}\big(\mathbb{R}^d, |\cdot|, wdx\big).
	\]
\end{thm}
The equivalence between $H^p_{at,\rho}(X)$ and $H^p_{L, \rad}(X)$ for $p=1$ was obtained in \cite{D3}. .

\subsection{Schr\"odinger operators on Heisenberg groups}
Let $\mathbb{H}^n$ be a $(2n + 1)$-dimensional Heisenberg group. Recall that $\mathbb{H}^n$ is a connected and simply connected
nilpotent Lie group with the underlying manifold $\mathbb{R}^{2n}\times \mathbb{R}$. The group structure is defined by
$$
(x, s)(y, t) = (x + y, s + t + 2 \sum_{j=1}^{n}(x_{d+j}y_j - x_jy_{d+j}))
$$
The homogeneous norm on $\mathbb{H}^d$ is defined by
$$
|(x, t)| = (|x|^4 + |t|^2)^{1/4}  \ \text{for all} \  (x, t) \in \mathbb{H}^n.
$$
See for example \cite{LL}.

This norm satisfies the triangle inequality and hence induces a left-invariant metric $d((x, t), (y, s)) = |(-x, -t)(y, s)|$. Moreover,
there exists a positive constant $C$ such that $|B((x, t), r)| = Cr^Q$, where $Q = 2d + 2$ is
the homogeneous dimension of $\mathbb{H}^n$ and $|B((x, t), r)|$ is the Lebesgue measure of the ball
$B((x, t), r)$. Obviously, the triplet $(\mathbb{H}^n, d, dx)$ satisfies the doubling condition (\ref{doublingcondition}).

A basis for the Lie algebra of left-invariant vector fields on $\mathbb{H}^d$ is given by
$$
X_{2n+1}=\f{\partial}{\partial t}, X_{j}=\f{\partial}{\partial x_j}+2x_{n+j}\f{\partial}{\partial t}, X_{n+j}=\f{\partial}{\partial x_{n+j}}-2x_{j}\f{\partial}{\partial t}, \ \ j=1,\ldots,n.
$$
and the sub-Laplacian $-\Delta_{\mathbb{H}^n}$  defined by
$$
\Delta_{\mathbb{H}^n}=-\sum_{j=1}^{2n}X_j^2.
$$
Furthermore, it was proved in \cite{LL} that the sub-Laplacian $\Delta_{\mathbb{H}^n}$ satisfies (A1)-(A4) with $\delta_1=1$. Therefore from Theorem \ref{mainthm2s} we have:
\begin{thm}
	Let $\rho$ be a critical function on $\mathbb{H}^n$. Let $p\in (\f{Q}{Q+1},1]$ and $q\in [1,\vc]\cap (p,\vc]$. Then we have
	\[
	h^{p,q}_{at,\rho}(\mathbb{H}^n)\equiv h^{p}_{\Delta_{\mathbb{H}^n},{\rm max}, \rho}(\mathbb{H}^n)\equiv h^{p}_{ \Delta_{\mathbb{H}^n},{\rm rad},\rho}(\mathbb{H}^n).
	\]
\end{thm}

We now consider the Schr\"odinger operator on $\mathbb{H}^n$ defined by $L= \Delta_{\mathbb{H}^n} + V$ where $V\in RH_q(\mathbb{H}^n), q>Q/2$. We define the critical function $\rho$ associated to $V$ by setting
\begin{equation}\label{eq-rho-Heisenberg}
\rho(x)=\rho(x)=\sup\Big\{r>0: \f{1}{r^{Q-2}}\int_{B((x,t),r)}V(y,s)dyds\leq 1\Big\}.
\end{equation}
Then, the Schr\"odinger operator $L$ satisfies conditions (B1) and (B2) with $\LL=\Delta_{\mathbb{H}^n}$, with any $\delta_2=2-Q/q$. See for example \cite{LL}. Arguing similarly to the proof of Proposition \ref{Proposition3: heat kernel bound} we imply that $L$ satisfies (B3) with any $0<\delta_3<\min\{1, 2-Q/q\}$. Then from Theorem \ref{mainthm3} we obtain:
\begin{thm}
	Let $L=\Delta_{\mathbb{H}^n} + V$ where $V\in RH_q(\mathbb{H}^n), q>Q/2$ and let $\rho $ be defined in \eqref{eq-rho-Heisenberg}. Let $p\in (\f{Q}{Q+\delta_0},1]$ and $q\in [1,\vc]\cap (p,\vc]$ where $\delta_0=\min\{1, 2-Q/q\}$. Then we have
	\[
	h^{p,q}_{at,\rho}(\mathbb{H}^n)\equiv H^p_{L, {\rm max}}(\mathbb{H}^n)\equiv H^p_{L, \rad}(\mathbb{H}^n).
	\]
\end{thm}
In the particular case $p=1$, the theorem is in line with that in \cite{LL}. Our result corresponding to $p<1$ is new.
\subsection{Schr\"odinger operators on connected and simply connected nilpotent
	Lie groups}
For background on connected and simply connected nilpotent Lie groups see \cite{V, NSW}. Let $\mathbb{G}$ be a connected and simply connected nilpotent Lie group. Let $X\equiv\{X_1,\ldots, X_k\}$ be left invariant vector fields on $\mathbb{G}$ satisfying the H\"omander condition. Let $d$ be the Carnot-Carath\'eodory distance on $\mathbb{G}$ associated to $X$ and $\mu$ be a left invariant Haar measure on $\mathbb{G}$. Then, there exist $0<\kappa\leq n<\vc$ such that $\mu(B(x,r))\approx r^\kappa$ when $0<r\leq 1$, and  $\mu(B(x,r))\approx r^n$ when $r\geq  1$, see for example \cite{NSW}.

The sub-Laplacian is defined by $\Delta_{\mathbb{G}}=-\sum_{j=1}^k X^2_j$. Then the operator $\Delta_{\mathbb{G}}$ generates the analytic semigroup $\{e^{-t\Delta_{\mathbb{G}}}\}_{t>0}$ whose kernels $\widetilde{p}_t(x,y)$ satisfy (A1)-(A4) with $\delta_1=1$. See for example \cite{V}. Hence, Theorem \ref{mainthm2s} implies the following.
\begin{thm}
	Let $\rho$ be a critical function on $\mathbb{G}$. Let $p\in (\f{n}{n+1},1]$ and $q\in [1,\vc]\cap (p,\vc]$. Then we have
	\[
	h^{p,q}_{at,\rho}(\mathbb{G})\equiv h^{p}_{ \Delta_{\mathbb{G}},{\rm max},\rho}(\mathbb{G})\equiv h^{p}_{ \Delta_{\mathbb{G}},{\rm rad},\rho}(\mathbb{G}).
	\]
\end{thm}

Let $V$ be a nonnegative locally integrable function on $\mathbb{G}$. Assume that $V\in RH_q(\mathbb{G}), q>n/2$ with its associated critical function $\rho$ defined by
\begin{equation}\label{eq-rho-Lie group}
\rho(x)=\sup\Big\{r>0: \f{r^2}{\mu(B(x,r))}\int_{B(x,r)}V(y)\dy\leq 1\Big\}.
\end{equation}
Then the operator $L=\Delta_{\mathbb{G}} +V$ generates the semigroup $\{e^{-tL}\}_{t>0}$ satisfying (B1) and (B2) with $\LL=\Delta_{\mathbb{G}}$ and $\delta_2= 2-n/q$. See for example \cite{YZ}. The argument used in the proof of Proposition \ref{Proposition3: heat kernel bound} yields that $L$ satisfies (B3) with any $0<\delta_3<\min\{1, 2-q/n\}$. Therefore, Theorem \ref{mainthm3} deduces the following result.
\begin{thm}
	Let $L=\Delta_{\mathbb{G}} +V$ be a Schr\"odinger operator with $V\in RH_{q}(\mathbb{G})$ with $q>n/2$ and let $\rho$ be as in \eqref{eq-rho-Lie group}. Let $p\in (\f{n}{n+\delta_0},1]$ and $q\in [1,\vc]\cap (p,\vc]$ with $\delta_0=\min\{1,2-n/q\}$. Then we have
	\[
	h^{p,q}_{at,\rho}(\mathbb{G})\equiv H^p_{L,{\rm max}}(\mathbb{G})\equiv H^p_{L,{\rm rad}}(\mathbb{G}).
	\]
\end{thm}
In \cite{YZ}, the authors prove the equivalence between $H^p_{at,\rho}(\mathbb{G})$ and $H_L^p(\mathbb{G})$ for $p=1$. Our result is new for $p\leq 1$.
\section{Appendices}\label{sec: appendix}
\subsection{Muckenhoupt weights}
Let $X$ be a space of homogeneous type as in Section 1. A weight $w$ is a non-negative measurable and locally integrable function on $X$.
We say that $w \in A_p$, $1 < p < \infty$, if there exists a
constant $C$ such that for every ball $B \subset X$,
$$
\Big(\fint_B w(x)\dx\Big)\Big(\fint_B w(x)^{-1/(p-1)}\dx\Big)^{p-1}\leq C.
$$
For $p = 1$, we say that $w \in A_1$ if there is a constant $C$ such
that for every ball $B \subset X$,
$$
\fint_B w(y)\dy \leq Cw(x) \ \text{for a.e. $x\in B$}.
$$
We set $A_\vc=\cup_{p\geq 1}A_p$.\\

The reverse H\"older classes are defined in the following way: $w
\in RH_q, 1 < q < \infty$, if there is a constant $C$ such that for
any ball $B \subset X$,
$$
\Big(\fint_B w(y)^q \dy\Big)^{1/q} \leq C \fint_B w(x)\dx.
$$
The endpoint $q = \infty$ is given by the condition: $w \in
RH_\infty$ whenever, there is a constant $C$ such that for any ball
$B \subset X$,
$$
w(x)\leq C \fint_B w(y)\dy  \ \text{for a.e. $x\in B$}.
$$
It is well known that $w\in A_\infty$ if and only if $w\in RH_q$ for some $q>1$.

\subsection{Proof of Theorem \ref{thm-Schrodinger manifolds}}\label{subsect-proof of 6.2}

In this subsection we always assume that $X$ is a manifold satisfying the doubling condition (\ref{doubling-manifold}) and a Poincar\'e inequality \eqref{Poincare-p}.

Our aim here is to give the proof of (B1)-(B3) in this setting. It is worth mentioning that we in fact prove something more general than (B1) in Theorem \ref{Th: improved heat kernel} by assuming $V\in A_\infty$. Estimate (B1) will then be deduced from Theorem \ref{Th: improved heat kernel} by restricting $V$ to $RH_q$ with $q>\max\{1,n/2\}$ (Proposition \ref{Proposition1: heat kernel bound}). The approach is based on the approach in \cite{K} and recently improved in \cite{RT} in the setting of Euclidean spaces. The main idea is to use the Fefferman--Phong inequality in \cite{BB} in place of the Fefferman--Phong inequality from \cite{Sh}. To keep our article self contained we give full details below.

Before giving the proof to the theorem we need some technical results. The first is the improved Fefferman-Phong inequality in \cite{BB}.
\begin{lem}\label{Lem: Feff-Phong manifold}
	Let $V\in A_{\infty}$ and $1\le p<\infty$. Then there are constants $C>0$ and $\beta\le 1$ depending only on the $A_{\infty}$ constant of $V$, on $p$, and on the constants in \eqref{doubling-manifold} and \eqref{Poincare-p}, such that for every ball $B$ of radius $r_B>0$ and $u\in W^1_{p,loc}$
	$$ \int_B \ContainB{\AbbsA{\nabla u}^2+V \AbbsA{u}^2}\,d\mu \ge C\frac{m_{\beta}\ContainB{r_B^2\fint_B V}}{r_B^2}\int_B\AbbsA{u}^2\,d\mu $$
	where
	$$ m_{\beta}(x):= \left\lbrace
	\begin{array}{ll}
	x^{\beta}\qquad &x\ge 1\\
	x \qquad &x\le 1
	\end{array}
	\right.$$
	
\end{lem}
We now consider some estimates related to weak subsolutions and weak solutions of the heat equation involving Schr\"odinger operators.

We fix the following notation. The set $Q$ will denote the parabolic cylinder
\begin{align*}
Q:=Q(x_Q,r_Q,t_Q)=\BracketB{(x,t)\in X\times (0,\infty): \;d(x_Q,x)<r_Q \quad\text{and}\quad t_Q-r_Q^2<t<t_Q}
\end{align*}
Given a fixed cylinder $Q$, we also write
\begin{align*}
B_Q:= B(x_Q,r_Q), && I_Q:=[t_Q-r_Q^2,t_Q] && I_Q^t := [t_Q-r_Q^2,t]
\end{align*}

\begin{defn}\label{Def: subs 1}
	Let $I$ be a closed interval in $\RR{}$ and $\Omega$ an open subset of $X$. Let $V$ be a non-negative function on $X$. We say $u$ is a weak subsolution of $(\partial_t-\Delta+V)$ in $I\times\Omega$ if $u\in W^{1,1}_2(I\times\Omega)\cap L^{\infty}(I\times\Omega)$ and
	\begin{align}\label{weak subsolution}
	\int_{I\times \Omega} (u_t\phi + \nabla u \cdot \nabla\phi + Vu\phi ) \,d\mu\,dt \le 0
	\end{align}
	for every $\phi\in C^{\infty}_0(I\times\Omega)$.
\end{defn}

\begin{defn}\label{Def: subs 2}
	We call $u(x,t)$ a weak solution to $(\partial_t-\Delta+V)u=0$ in $Q$ if
	\begin{enumerate}[(a)]
		\item
		$u\in L^{\infty}\ContainB{I_Q;W^{1,2}(B_Q)}\cap L^2\ContainB{I_Q; W^{1,2}(B_Q)}$ and
		\item $u$ satisfies for each $t\in I_Q$,
		\begin{equation*}
		\int_{B_Q} u(x,t)\phi(x,t)\,d\mu
		-\int\limits_{t_Q-r_Q^2}^t\int_{B_Q} \ContainB{u\phi_s + \nabla u\cdot\nabla \phi + Vu\phi }\,d\mu\,ds = 0
		\end{equation*}
		for all $\phi \in \mathcal{D}$, where
		$$ \mathcal{D}:=\BracketB{\varphi\in L^2\ContainB{I_Q;W^{1,2}(B_Q)} :\; \varphi_s\in L^2\ContainB{I_Q; L^2(B_Q)} \quad\text{and}\quad \varphi(x,t_Q,r_Q^2)=0} $$
	\end{enumerate}
\end{defn}

We have the following simple result.
\begin{lem}
	Let $I$ be a closed interval in $\RR{}$ and $\Omega$ an open subset of $X$. Let $V$ be a non-negative function on $X$.
	Suppose
	$u\in W^{1,1}_2(I\times\Omega)\cap L^{\infty}(I\times\Omega)$ is a weak subsolution of $(\partial_t-\Delta+V)$ in $I\times\Omega$.
	
	Let $g:\mathbb{R}\rightarrow \mathbb{R}$ satisfy $g'' \geq 0$, $g'\geq 0$ and $g(0)=0$.
	Then $g(u)$ is a weak subsolution of $(\partial_t -\Delta+V)$ in $I\times\Omega$.
\end{lem}
\begin{proof}
	Let $\varphi\in C^{\infty}_0(I\times\Omega)$ be non-negative. We will check that $g(u)$ and $\varphi$ satisfy \eqref{weak subsolution}.
	
	Firstly take $\phi = \varphi g'(u)$ as a test function   in \eqref{weak subsolution}. This gives
	$$ \int_{I\times\Omega} ( u_t\varphi g'(u) + \nabla u\cdot\nabla(\varphi g'(u))+Vu\varphi g'(u))\,d\mu\,dt\le 0$$
	or
	$$ \int_{I\times\Omega} \ContainB{ u_t\varphi g'(u) + \varphi g''(u) \AbbsA{\nabla u}^2 + g'(u)\nabla u\cdot\nabla\varphi+Vu\varphi g'(u)}\,d\mu\,dt\le 0$$
	Therefore the proceeding inequality implies
	\begin{align*}
	\int_{I\times\Omega} &\varphi \partial_t g(u) + \nabla g(u)\cdot\nabla\varphi + Vg(u)\varphi \\
	&=
	\int_{I\times\Omega} \varphi g'(u) \partial_t u + \nabla g(u)\cdot\nabla\varphi + Vg(u)\varphi \\
	&\le
	-\int_{I\times\Omega} \varphi g''(u)\AbbsA{\nabla u}^2 + \int_{I\times\Omega}V\varphi\ContainB{g(u)-ug'(u)}\\
	&\le \int_{I\times\Omega}V\varphi\ContainB{g(u)-ug'(u)} \\
	&\le 0
	\end{align*}
	The second inequality holds because $g''\ge 0$. The final inequality holds because $V, \varphi \ge 0$ and the hypotheses $g'\ge 0$ and $g(0)=0$ imply that $g(s)-sg'(s)\le 0$ for all $s\in\RR{}$.
\end{proof}	

The following estimate can be viewed as a Cacciopoli's inequality related to Schr\"odinger operators in manifolds.

\begin{lem}\label{Lem: cacciopoli manifold}
	Let $V$ a non-negative function on $X$. Suppose that $u$ is a weak solution to
	$ (\partial_t -\Delta +V)u=0 $
	in $2Q$. Then there exists $C>0$ such that for every $\sigma\in (0,1)$,
	$$\sup_{t\in[t_Q-(\sigma r_Q)^2,t_Q]} \int_{\sigma B_Q}\AbbsA{u}^2\,d\mu + \int_{\sigma Q}\AbbsA{\nabla u}^2 + V\AbbsA{u}^2\,d\mu\,ds \le \frac{C}{r_Q^2(1-\sigma)^2}\int_Q \AbbsA{u}^2\,d\mu\,ds $$
\end{lem}

\begin{proof}
	We adapt some ideas in \cite{K} and proceed in the following steps.
	
\emph{Step 1:  The cutoff functions}. We begin by defining some auxiliary functions:
		the spatial cut-off $\chi\in C^{\infty}_0(B_Q)$ with
		\begin{align*}
		0\le \chi \le 1, && \chi \equiv 1 \quad\text{on}\quad\sigma B_Q, && \AbbsA{\nabla\chi}\lesssim\frac{1}{r_Q(1-\sigma)},
		\end{align*}
		and the temporal cutoff $\eta\in C^{\infty}(\RR{})$ with
		\begin{align*}
		0\le\eta\le 1,  &&\AbbsA{\eta_t}\lesssim \frac{1}{r_Q^2(1-\sigma)^2}&&
		\eta(t)=\left\lbrace \begin{array}{ll}
		1 &\quad t\ge t_Q-(\sigma r_Q)^2 \\
		... &\quad \text{else}\\
		0 &\quad t\le t_Q-r_Q^2
		\end{array}\right.
		\end{align*}
	
\emph{Step 2: The test function.}
	Let $u$ be a weak solution to $(\partial_t-\Delta+V)u=0$ in $2Q$ in the sense of Definition \ref{Def: subs 2}. We may assume that $u_t\in L^2(2Q)$, since we can remove this assumption by the argument in Aronson and Serrin (1967).
	
	Take $\phi(x,t):=\eta(t)^2\chi(x)^2 u(x,t)$. Let us show that $\phi\in\mathcal{D}$.
		Firstly since $u$ is a weak solution then $u\in L^2\ContainB{I_Q;W^{1,2}(B_Q)}$ and hence $\eta^2\chi^2 u\in L^2\ContainB{I_Q;W^{1,2}(B_Q)}$.
		Secondly $\phi_t = \ContainB{\eta^2\chi^2 u}_t = 2u\chi^2\eta\eta_t+\eta^2\chi^2u_t \in L^2\ContainB{I_Q; L^2(B_Q)}$ since $u_t\in L^2(B_Q)$.
		Finally $\phi(x,t_Q-r_Q^2) = \eta^2(t_Q-r_Q^2)\chi^2(x)u(x,t_Q-r_Q^2)=0$ since $\eta(t_Q-r_Q^2)=0$.
	Therefore we may conclude $\phi\in \mathcal{D}$.
	
 \emph{Step 3: The identity for weak solutions}.
	Fix $t\in [t_Q-(\sigma r_Q)^2,t_Q]$. We use the notation $I_Q^t=[t_Q-r_Q^2,t]$. By parts with respect to the variable $s$ gives
	\begin{align}
	\int_{I_Q^t}\int_{B_Q} \eta^2\chi^2 u u_s \,d\mu\,ds
	&=
	\int_{B_Q}\GroupingB{\eta^2\chi^2 u^2}_{t_Q-r_Q^2}^t d\mu - \int_{I_Q^t}\int_{B_Q} u\, \partial_s\ContainB{\eta^2\chi^2 u}\,d\mu\,ds \notag\\
	&= \int_{B_Q}\chi^2(x)u^2(x,t)\,d\mu - \int_{I_Q^t}\int_{B_Q} u\,\partial_s\ContainB{\eta^2\chi^2 u^2}\,d\mu\,ds \label{cacciopoli 3.1}
	\end{align}
	since $\eta^2(t) =1$ because $t\ge t_Q^2-(\sigma r_Q)^2$.
	
	Next the product rule respect to $s$ gives
	\begin{align}
	\int_{I_Q^t}\int_{B_Q} u\partial_s\ContainB{\eta^2\chi^2 u}\,d\mu\,ds
	&=
	2\int_{I_Q^t}\int_{B_Q}\chi^2 u^2\eta\eta_s \,d\mu\,ds + \int_{I_Q^t}\int_{B_Q} \eta^2\chi^2 u u_s\,d\mu\,ds \label{cacciopoli 3.2}
	\end{align}
	Inserting \eqref{cacciopoli 3.1} into \eqref{cacciopoli 3.2} gives
	\begin{align*}
	\int_{I_Q^t}\int_{B_Q} u \partial_s\ContainB{\eta^2\chi^2 u }\,d\mu\,ds
	= 2\int_{I_Q^t}\int_{B_Q} \chi^2 u^2\eta\eta_s\,d\mu\,ds + \int_{B_Q}\chi^2 u^2\,d\mu - \int_{I_Q^t}\int_{B_Q} u\partial_s\ContainB{\eta^2\chi^2 u}\,d\mu\,ds
	\end{align*}
	Rearrange this to get
	\begin{align}\label{cacciopoli 3.3}
	\int_{I_Q^t}\int_{B_Q} u \partial_s\ContainB{\eta^2\chi^2 u }\,d\mu\,ds
	=\int_{I_Q^t}\int_{B_Q} \chi^2 u^2\eta\eta_s\,d\mu\,ds+\tfrac{1}{2}\int_{B_Q}\chi^2 u^2\,d\mu
	\end{align}
	Now take $\phi=\eta^2\chi^2 u$ as a test function in Definition \ref{Def: subs 2} (b) to get
	\begin{align*}
	\int_{B_Q}u^2\chi^2\,d\mu
	-\int_{I_Q^t}\int_{B_Q} u\partial_s\ContainB{\eta^2\chi^2 u}\,d\mu\,ds
	+\int_{I_Q^t}\int_{B_Q}\nabla u\cdot\nabla\ContainB{\eta^2\chi^2 u}+Vu\ContainB{\eta^2\chi^2 u}\,d\mu\,ds = 0
	\end{align*}
	Insert \eqref{cacciopoli 3.3} into the above to get
	\begin{align}\label{cacciopoli 3.4}
	\tfrac{1}{2}\int_{B_Q}\chi^2u^2\,d\mu - \int_{I_Q^t}\int_{B_Q}\chi^2u^2\eta\eta_s\,d\mu\,ds + \int_{I_Q^t}\int_{B_Q}\nabla u\cdot\nabla\ContainB{\eta^2\chi^2 u}+Vu\ContainB{\eta^2\chi^2 u}\,d\mu\,ds = 0
	\end{align}
	Noting that
	\begin{align*}
	\nabla u\cdot\nabla\ContainB{\eta^2\chi^2 u} = \eta^2\chi^2\nabla u\cdot\nabla u + \eta^2 u \nabla u\cdot\nabla(\chi^2)
	\end{align*}	
	and inserting this into the third term of \eqref{cacciopoli 3.4} gives
	\begin{align}
	\tfrac{1}{2}\int_{B_Q}\chi^2u^2\,d\mu
	+\int_{I_Q^t}\int_{B_Q} \eta^2\chi^2\AbbsA{\nabla u}^2\,d\mu\,ds
	+\int_{I_Q^t}\int_{B_Q} Vu^2\eta^2\chi^2\,d\mu\,ds \notag \\
	=
	\int_{I_Q^t}\int_{B_Q} \chi^2u^2\eta\eta_s\,d\mu\,ds
	-\int_{I_Q^t}\int_{B_Q}\eta^2 u \nabla u\cdot\nabla (\chi^2)\,d\mu\,ds
	\label{cacciopoli 3.5}
	\end{align}

\emph{Step 4: Control of the  $\int \AbbsA{\nabla u}^2$ term}.	
	By the non-negativity of $V$, Cauchy--Schwarzt's inequality, H\"older's inequality , and by taking $t=t_Q$ in \eqref{cacciopoli 3.5} we obtain
	\begin{align*}
	\int_Q \AbbsA{\nabla u}^2 \chi^2\eta^2\,d\mu\,ds
	&\le
	\int_Q\chi^2 u^2\eta\AbbsA{\eta_s} \,d\mu\,ds +\int_Q\eta^2 \AbbsA{u}\nabla u\cdot\nabla(\chi^2)\,d\mu\,ds \\
	&\le
	\int_Q u^2\AbbsA{\eta_s}\,d\mu\,ds + 2\int_Q\chi\eta^2\AbbsA{u}\AbbsA{\nabla u}\AbbsA{\nabla \chi}\,d\mu\,ds\\
	&\le
	\int_Q u^2\AbbsA{\eta_s}\,d\mu\,ds + 2\int_{I_Q}\eta^2\ContainC{\int_{B_Q}\AbbsA{\nabla u}^2\chi^2\,d\mu}^{1/2}\ContainC{\int_{B_Q}\AbbsA{\nabla \chi}^2 u^2\,d\mu}^{1/2}\,ds \\
	&\le
	\int_Q u^2\AbbsA{\eta_s}\,d\mu\,ds + 2\int_{I_Q}\eta^2\GroupingD{\frac{1}{4\varepsilon}\int_{B_Q}\AbbsA{\nabla u}^2\chi^2\,d\mu + \varepsilon\int_{B_Q}\AbbsA{\nabla \chi}^2 u^2\,d\mu}\,ds,
	\end{align*}
	which along with the fact that $\sqrt{A}\sqrt{B}\le \tfrac{1}{4\varepsilon}\sqrt{A}+\varepsilon\sqrt{B}$ gives
	\begin{align*}
	\int_Q \AbbsA{\nabla u}^2 \chi^2\eta^2\,d\mu\,ds
	&\le
	\frac{C}{r_Q^2(1-\sigma)^2}\int_Q u^2\,d\mu\,ds + \frac{1}{2\varepsilon}\int_Q\AbbsA{\nabla u}^2\chi^2\eta^2\,d\mu\,ds + \frac{C'\varepsilon}{r_Q^2(1-\sigma)^2}\int_Q u^2\,d\mu\,ds.
	\end{align*}
	By using the properties $\AbbsA{\eta_s}\lesssim r_Q^{-2}(1-\sigma)^{-2}$ and $\AbbsA{\nabla \chi}\lesssim r_Q^{-1}(1-\sigma)^{-1}$ and taking $\varepsilon=1$ we arrive at
	\begin{align*}
	\int_Q \AbbsA{\nabla u}^2 \chi^2\eta^2\,d\mu\,ds 	&\leq
	\frac{C''}{r_Q^2(1-\sigma)^2}\int_Qu^2\,d\mu\,ds + \tfrac{1}{2}\int_Q\AbbsA{\nabla u}^2\chi^2\eta^2\,d\mu\,ds.
	\end{align*}
	Rearranging this inequality gives
	\begin{align}\label{cacciopoli 4.1}
	\int_Q\chi^2\eta^2\AbbsA{\nabla u}^2\,d\mu\,ds
	\le \frac{2C''}{r_Q^2(1-\sigma)^2}\int_Q u^2\,d\mu\,ds
	\end{align}

\emph{Step 5: Control of the $\int Vu^2$ term}.
	Taking $t=t_Q$ in \eqref{cacciopoli 3.5} and applying a similar argument to Step 4 we obtain
	\begin{align}
	\int_Q Vu^2\chi^2\eta^2\,d\mu\,ds
	&\le
	\int_Q\chi^2 u^2\eta\AbbsA{\eta_s}\,d\mu\,ds +\int_Q\eta^2\AbbsA{u}^2\nabla u\cdot\nabla(\chi^2)\,d\mu\,ds \notag\\
	&\le
	\frac{C+1}{r_Q^2(1-\sigma)^2}\int_Qu^2\,d\mu\,ds + \tfrac{1}{2}\int_Q\AbbsA{\nabla u}^2\chi^2\eta^2\,d\mu\,ds.
	\end{align}
	We now apply \eqref{cacciopoli 4.1} to the second term in the inequality above to conclude that
	\begin{align}
	\int_Q Vu^2\chi^2\eta^2\,d\mu\,ds
	&\lesssim
	\frac{1}{r_Q^2(1-\sigma)^2}\int_Qu^2\,d\mu\,ds \label{cacciopoli 5.1}
	\end{align}

\emph{Step 6: Control of the $\int u^2$ term}.
		Since $V\ge 0$, the left hand side of \eqref{cacciopoli 3.5} is positive and hence by applying a similar argument to Step 4 we have
	\begin{align}
	\sup_{t\in[t_Q-(\sigma r_Q)^2,t_Q]} &\int_{B_Q}u^2(x,t)\chi^2(x)\,d\mu \notag\\
	&\le
	2\int_Q\chi^2 u^2\eta\AbbsA{\eta_s} \,d\mu\,ds +2\int_Q\eta^2 \AbbsA{u}\nabla u\cdot\nabla(\chi^2)\,d\mu\,ds \notag\\
	&\le
	\frac{C}{r_Q^2(1-\sigma)^2}\int_Q u^2\,d\mu\,ds + \frac{1}{\varepsilon}\int_Q\AbbsA{\nabla u}^2\chi^2\eta^2\,d\mu\,ds + \frac{C'\varepsilon}{r_Q^2(1-\sigma)^2}\int_Q u^2\,d\mu\,ds \notag\\
	\end{align}
	Taking $\varepsilon=1$ and applying \eqref{cacciopoli 4.1}, we derive
	\begin{align}
	\sup_{t\in[t_Q-(\sigma r_Q)^2,t_Q]} \int_{B_Q}u^2(x,t)\chi^2(x)\,d\mu \notag
	&\lesssim
	\frac{C''}{r_Q^2(1-\sigma)^2}\int_Qu^2\,d\mu\,ds + \int_Q\AbbsA{\nabla u}^2\chi^2\eta^2\,d\mu\,ds \notag\\
	&\lesssim
	\frac{1}{r_Q^2(1-\sigma)^2}\int_Q u^2\,d\mu\,ds. \label{cacciopoli 6.1}
	\end{align}
	
\emph{Step 7: Putting it all together}.
	Equations \eqref{cacciopoli 4.1} and \eqref{cacciopoli 5.1} give
	\begin{align}\label{cacciopoli 7.1}
	\int_Q\AbbsA{\nabla u}^2\chi^2\eta^2\,d\mu\,ds + \int_QV u^2\chi^2\eta^2\,d\mu\,ds
	\le
	\frac{C}{r_Q^2(1-\sigma)^2}\int_Q u^2\,d\mu\,ds
	\end{align}
	Since $\chi = 1$ on $\sigma B_Q$, $\eta =1$ for $t\ge t_Q-(\sigma r_Q)^2$, and noting that
	$$ \sigma Q= \BracketB{(y,s)\in M\times (0,\infty):\quad d(x_Q,y)<\sigma r_Q \quad\text{and}\quad t_Q-(\sigma r_Q)^2<s<t_Q}$$
	we have by \eqref{cacciopoli 7.1}
	\begin{align*}
	\int_{\sigma Q} \AbbsA{\nabla u}^2+Vu^2 \,d\mu\,ds
	\le
	\int_Q\ContainB{\AbbsA{\nabla u}^2+Vu^2}\eta^2\chi^2\,d\mu\,ds 
	\le
	\frac{C}{r_Q^2(1-\sigma)^2}\int_Q u^2\,d\mu\,ds
	\end{align*}
	Combining this final inequality with \eqref{cacciopoli 6.1} we obtain the required result.
\end{proof}

We now record the following mean value inequality related to Laplace-Beltrami operators.

\begin{lem}\label{Lem: MVI for Laplace}
	Let $u$ be a weak subsolution of $(\partial_t-\Delta)u\le 0$ in $Q$. Then
	$$ \sup_{(x,t)\in\frac{1}{2}Q}\AbbsA{u(x,t)}\le \ContainC{\frac{C}{r_Q^2 \mu(B_Q)}\int_{\frac{2}{3}Q}u^2\,d\mu\, dt}^{1/2} $$
\end{lem}
\begin{proof}
	This is proved by Saloff-Coste in \cite{Sc1, Sc2}.
\end{proof}

\begin{lem}[Mean value inequality for Schr\"odinger]\label{Lem: MVI for Schrodinger}
	
	Let $u$ be a weak solution of $(\partial_t-\Delta+V)u =  0$ in $Q$. Then
	$$ \sup_{(x,t)\in\frac{1}{2}Q} \AbbsA{u(x,t)}\le \ContainC{\frac{C}{r_Q^2 \mu(B_Q)}\int_{\frac{2}{3}Q}u^2\,d\mu\, dt}^{1/2} $$
	
\end{lem}

\begin{proof}
	Suppose that $u_+$ is a non-negative weak solution to $(\partial_t-\Delta+V)u_+ =  0$ in $Q$. Then $(\partial_t-\Delta)u_+ =  -V u_+ \le 0$, since $V$ is non-negative. Hence Lemma \ref{Lem: MVI for Laplace} applies to $u_+$.
\end{proof}

\begin{lem}\label{Lem: improved subs manifold}
	Let $V\in A_{\infty}$ and $L=-\Delta+V$. Assume $u$ is a weak solution of $(\partial_t + L) u =0$ in $2Q$ for some parabolic cylinder $Q$. Then for each $k>0$ there exists $C_k>0$ such that
	\begin{align}\label{eqn: improved subs manifold}
	\sup_{(x,t)\in \frac{1}{2}Q}\AbbsA{u(x,t)} \le \frac{C_k}{\ContainB{1+r_Q^2\fint_{B(x_Q, r_Q)}V}^k} \BracketC{\frac{1}{r_Q^{2}\mu(B_Q)}\int_Q \AbbsA{u(x,t)}^2d\mu\,dt}^{1/2}.
	\end{align}
\end{lem}
\begin{rem}
	\begin{enumerate}[\upshape(a)]
		\item We can rewrite this in an equivalent form:
		$$ \sup_{\frac{1}{2}Q}\AbbsA{u}\le \frac{C_k}{\ContainB{1+t_Q \fint_{B(x_Q,\sqrt{t_Q})}V}^k}\BracketC{\frac{1}{r_Q^{2}\mu(B_Q)}\int_Q\AbbsA{u}^2}^{1/2}$$
		\item It is possible to improve this to exponential decay:
		$$ \sup_{\frac{1}{2}Q}\AbbsA{u}\le C_k\exp\BracketB{-\ContainB{1+t_Q \fint_{B(x_Q,\sqrt{t_Q})}V}^{\delta}}\BracketC{\frac{1}{r_Q^{2}\mu(B_Q)}\int_Q\AbbsA{u}^2}^{1/2}$$
	\end{enumerate}
	for some $\delta>0$.
\end{rem}

\begin{proof}[Proof of Lemma \ref{Lem: improved subs manifold}]
	Fix $k\in\mathbb{N}$. For each $j = 1,2, \dots, k+1$ set $\alpha_j = \frac{2}{3}+\frac{j-1}{3k}$. Our aim is to prove that there exists $C>0$ such that for each $1\le j\le k$,
	\begin{align}\label{manifold subs pf 1}
	\int_{\alpha_j Q}\AbbsA{u}^2\,dx\,dt \le C\frac{k^2}{\ContainB{1+r_Q^2\fint_{B_Q}V}^{2\beta}}\int_{\alpha_{j+1}Q}\AbbsA{u}^2\,d\mu\,dt
	\end{align}
	By iterating this $k$ times we thus obtain
	\begin{align*}
	\int_{\frac{2}{3}Q}\AbbsA{u}^2\,dx\,dt \le C \frac{k^{2k}}{\ContainB{1+r_Q^2\fint_{B_Q}V}^{2\beta k}}\int_Q\AbbsA{u}^2\,d\mu\,dt.
	\end{align*}
	Then we may insert this into the basic subsolution estimate in Lemma \ref{Lem: MVI for Schrodinger} to obtain
	\begin{align*}
	\sup_{\frac{1}{2}Q}\AbbsA{u}\le C^{k/2}\frac{k^k}{\ContainB{1+r_Q^2\fint_{B_Q}V}^{\beta k/2}}\BracketC{\frac{1}{r_Q^{2}\mu(B_Q)}\int_Q\AbbsA{u}^2\,d\mu\,dt}^{1/2}.
	\end{align*}
	To arrive at \eqref{eqn: improved subs manifold}, for each $k>0$ we simply choose an integer $N$ large enough so that $k<\beta N/2$ and apply the preceding estimate to the integer $N$.
	
	We proceed with obtaining \eqref{manifold subs pf 1}. For each $1\le j\le k$ we pick two cutoff functions as follows. First set
	$$ \widetilde{\alpha}_j = \frac{1}{2}(\alpha_j + \alpha_{j+1}) = \frac{2}{3}+\frac{j}{3k}-\frac{1}{6k}$$
	Then for the spatial cutoff we pick $\chi_j\in C_0^{\infty}(\RR{n})$ with
	\begin{align*}
	{\rm supp}\,\chi_j \subseteq \widetilde{\alpha}_jB_Q, && 0\le \chi_j\le 1, && \chi_j \equiv 1 \;\;\text{on}\;\;\alpha_j B_Q, && \AbbsA{\nabla\chi_j}\lesssim \frac{k}{r_Q}.
	\end{align*}
	For the temporal cutoff we pick $\eta_j \in C_0^{\infty}(M)$ with $0\le\eta_j\le 1$ and
	\begin{align*}
	{\rm supp}\,\eta_j\subseteq (t_Q-(\widetilde{\alpha}_jr_Q)^2, t_Q], &&
	\eta_j \equiv 1 \;\;\text{on}\;\; (t_Q-(\alpha_j r_Q)^2,t_Q].
	\end{align*}
	Let us set
	$$ \widehat{m}_{\beta}(B):= m_{\beta}\ContainC{r_B^2\fint_B V}$$
	
	Then for each $j=1,\dots, k$, we have
	\begin{align*}
	\int_{\alpha_j Q}\AbbsA{u}^2\,d\mu\,dt
	&\le
	\int_{\widetilde{\alpha}_jQ}\AbbsA{\eta_j \chi_j u}^2\,d\mu\,dt \\
	&\le
	C\frac{r_{\widetilde{\alpha}_jQ}^2}{\widehat{m}_{\beta}(\widetilde{\alpha}_jB_Q)} \int_{\widetilde{\alpha}_jQ}\AbbsA{\nabla(\eta_j\chi_j u)}^2 + V\AbbsA{\eta_j\chi_j u}^2\,d\mu\,dt \\
	&\le C \frac{r_{Q}^2}{\widehat{m}_{\beta}(\widetilde{\alpha}_jB_Q)} \int_{\widetilde{\alpha}_jQ}\eta_j^2\ContainB{\chi_j^2\AbbsA{\nabla u}^2+u^2\AbbsA{\nabla\chi_j}^2} + V\AbbsA{\eta_j\chi_j u}^2\,d\mu\,dt \\
	&\le C_k \frac{k^2}{\widehat{m}_{\beta}(\widetilde{\alpha}_jB_Q)} \int_{\alpha_{j+1}Q}\AbbsA{u}^2\,d\mu\,dt \\
	&\le
	C_k \frac{k^2}{\widehat{m}_{\beta}(B_Q)}\int_{\alpha_{j+1}Q}\AbbsA{u}^2\,d\mu\,dt.
	\end{align*}
	In the second line we applied the Fefferman-Phong inequality (Lemma \ref{Lem: Feff-Phong manifold}) to $\eta_j\chi_j u$ and the ball $\widetilde{\alpha}_jB_Q$ with $p=2$.
	In the third line we used that
	$$ \AbbsA{\nabla(\eta_j\chi_j u)}^2\le 2\eta_j^2\ContainB{\chi_j\AbbsA{\nabla u}^2+ u^2\AbbsA{\nabla\chi_j}^2}.$$
	In the fourth line we applied firstly $\AbbsA{\nabla\chi_j}\lesssim k/r_Q$, and secondly Cacciopoli's inequality (Lemma \ref{Lem: cacciopoli manifold}) to $\AbbsA{\nabla u}^2 + V\AbbsA{u}^2$ on $\alpha_{j+1}Q$ with $\sigma = \frac{\widetilde{\alpha}_j}{\alpha_{j+1}}$. In this case we have $(1-\sigma)^{-2} = \ContainB{\frac{2k+j}{3}}^2\le k^2$.
	In the final line we used that $V$ is doubling, and that $2/3\le \widetilde{\alpha}_j\le 1$.
	
	Next we obtain \eqref{manifold subs pf 1} by considering two cases. If $r_Q^2\fint_{B_Q}V >1$ then using
	$$ 2^{\beta} \widehat{m}_{\beta}(B_Q) > \ContainB{1+ r_Q^2\fint_{B_Q}V}^{\beta}$$
	and we obtain \eqref{manifold subs pf 1}. On the other hand if $r_Q^2\fint_{B_Q}V \le 1$ then since
	$$ \ContainB{1+r_Q^2\fint_{B_Q}V}^{\beta}\le 2^{\beta}$$
	we may apply this with the trivial inequality
	$$ \int_{\alpha_jQ}\AbbsA{u}^2\,d\mu\,dt \le k^2\int_{\alpha_{j+1}Q}\AbbsA{u}^2\,d\mu\,dt$$
	which always holds. In either case we obtain \eqref{manifold subs pf 1}.
\end{proof}

In this section we apply the subsolution estimates to obtain
\begin{thm}[Improved heat kernel bounds]\label{Th: improved heat kernel}
		Let $V\in\Aclass_{\infty}$ and $L=-\Delta+V$. Then the heat kernel $p_t(x,y)$ of $L$ satisfies the following: for each $k>0$ there exists $C_k>0$ and $c>0$ such that for all $x,y,\in M$ and $t>0$
	\begin{align*}
	p_t(x,y)\le \frac{C_k}{\ContainC{1+t \fint_{B(x,\sqrt{t})}V+t \fint_{B(x,\sqrt{t})}V}^k}\frac{e^{-d(x,y)^2/ct}}{\mu\ContainB{B(x,\sqrt{t})}}
	\end{align*}
	
\end{thm}
\begin{proof}

	Fix $x,y\in X$ and $t>0$ with $x\ne y$. Set $u(z,s):=p_s(z,y)$ for each $s>0$ and $z\ne y$. We also define the cylinder $Q$ by setting $x_Q = x$, $t_Q=t$ and $r_Q^2 = t/2$. Then clearly $(x,t)\in \frac{1}{2}Q$ and $u$ is a weak solution of $\ContainB{\frac{\partial}{\partial t} + L}u=0$ in $2Q$. Therefore by the improved subsolution estimate in Lemma \ref{Lem: improved subs manifold}, we have for each $k>0$, (note that we write $B_Q:=B(x, r_Q)$)
	\begin{align*}
	\AbbsA{p_t(x,y)}
	&\le \sup_{(z,s)\in \frac{1}{2}Q}\AbbsA{u(z,s)} \\
	&\le \frac{C_k}{\ContainB{1+r_Q^2\fint_{B(x,r_Q)}V}^k} \BracketC{\frac{1}{r_Q^2 \mu(B_Q)}\int_Q \AbbsA{u(z,s)}^2\,dz\,ds}^{1/2} \\
	&\le \frac{C_k}{\ContainB{1+r_Q^2\fint_{B(x,r_Q)}V}^k} \BracketC{\frac{1}{r_Q^2 \mu(B_Q)}\int_Q \frac{e^{-cd(z,y)^2/s}}{\mu\ContainB{B(z,\sqrt{s})}}\,dz\,ds}^{1/2} \\
	&\le \frac{C_k}{\ContainB{1+r_Q^2\fint_{B(x,r_Q)}V}^k} \frac{1}{\mu\ContainB{B(x,\sqrt{t})}} \\
	&\le \frac{C_k}{\ContainB{1+\frac{t}{2}\fint_{B(x,\frac{\sqrt{t}}{\sqrt{2}})}V}^k} \frac{1}{\mu\ContainB{B(x,\sqrt{t})}} \\
	&\le
	\frac{C_k}{\ContainB{1+t\fint_{B(x,\sqrt{t})}V}^k}\frac{1}{\mu\ContainB{B(x,\sqrt{t})}}
	\end{align*}
	In the third inequality we used the well known Gaussian bounds on $p_s(z,y)$ since $V\ge0$. In the final inequality we used that $V$ is a doubling measure.
	
	In the fourth inequality we
	used that $s$ is comparable to $t$ since $s\in [t-r_Q^2,t]$ implies $t\ge s\ge t-r_Q^2 = t/2$, and
	applied the following computation: if $z\in B(x,r_Q) =B(x,\sqrt{t/2})$ and $\sqrt{t/2} \le \sqrt{s}\le \sqrt{t}$, then $x\in B(z,\sqrt{s})$ since $\sqrt{t/2}\le \sqrt{s}$. Then by the doubling property \eqref{doubling-manifold},
	\begin{align*}
	\mu\ContainB{(B(x,\sqrt{t})}
	&\le \mu \ContainB{B(z,2\sqrt{t})} \le C \mu\ContainB{B(z,\sqrt{t/2})} \le C\mu\ContainB{B(z,\sqrt{s})}.
	\end{align*}
	Therefore
	\begin{align*}
	\frac{1}{r_Q^2\mu\ContainB{B(x,r_Q)}}\int_{t-r_Q^2}^t\int_{B(x,r_Q)} \frac{dz\,ds}{\mu\ContainB{B(z,\sqrt{s})} }
	&\le
	\frac{1}{r_Q^2\mu\ContainB{B(x,r_Q)}}\int_{t-r_Q^2}^t\int_{B(x,r_Q)} \frac{C\,dz\,ds}{\mu\ContainB{B(x,\sqrt{t})} } \\
	&=
	\frac{C}{\mu\ContainB{B(x,\sqrt{t})}^2}.
	\end{align*}	
	
	Finally, from the Gaussian bounds on $p_t(x,y)$, we have
	\begin{align*}
	\AbbsA{p_t(x,y)}^2
	&\le \frac{C_k}{\ContainB{1+t\fint_{B(x,\sqrt{t})}V}^k}\frac{1}{\mu\ContainB{B(x,\sqrt{t})}} \AbbsB{p_t(x,y)} 
	&\le \frac{C_k}{\ContainB{1+t\fint_{B(x,\sqrt{t})}V}^k}\frac{e^{-cd(x,y)^2/t}}{\mu\ContainB{B(x,\sqrt{t})}^2}.
	\end{align*}
	Taking square roots gives the estimate
	\begin{align*}
	p_t(x,y)
	\le \frac{C_k}{\ContainB{1+t\fint_{B(x,\sqrt{t})}V}^k}\frac{e^{-cd(x,y)^2/t}}{\mu\ContainB{B(x,\sqrt{t})}}.
	\end{align*}
	Now symmetry of the heat kernel $ p_t(x,y)=p_t(y,x)$
	implies that
	\begin{align*}
	p_t(x,y)^2
	&= p_t(x,y)\,p_t(y,x) \\
	&\le \frac{C^2_k}{\ContainB{1+t\fint_{B(x,\sqrt{t})}V}^k}\frac{1}{\ContainB{1+t\fint_{B(y,\sqrt{t})}V}^k} \frac{e^{-cd(x,y)^2/t}}{\mu\ContainB{B(x,\sqrt{t})}} \\
	&\le
	\frac{2^kC^2_k }{\ContainB{1+t\fint_{B(x,\sqrt{t})}V+t\fint_{B(y,\sqrt{t})}V}^k}\frac{e^{-cd(x,y)^2/t}}{\mu\ContainB{B(x,\sqrt{t})}}.
	\end{align*}
	Taking square roots again gives the required estimate.
	
	Note that we have used the inequality
	$$ (1+A+B)^k\le 2^k (1+A)^k(1+B)^k$$
	valid for all $A,B,k\ge 0$. Indeed if $x,y\ge 1$ then $x^{-1}+y^{-1}\le 2$ and hence $ (x+y)^k\le 2^kx^ky^k$. Then it follows that
	$$ (1+A+B)^k \le (2+A+B)^k = (1+A+1+B)^k\le 2^k(1+A)^k (1+B)^k.
	$$
\end{proof}

We now record without proof some auxiliary results related to the critical function $\rho$. See for example \cite{Sh, YZ}.
\begin{lem}\label{Lem1: rho-manifold}
	Let $V \in RH_q\cap A_\vc$ with $q>\max\{1,n/2\}$ and let $\rho$ be a function defined as in (\ref{rhofunction}). Then we have the following.
	\begin{enumerate}[{\rm (a)}]
		\item $\rho$ is a critical function satisfying \eqref{criticalfunction}.
		
		\item There exists $C>0$ so that
		$$
		\f{r^2}{\mu(B(x,r))}\int_{B(x,r)}V(y)d\mu(y)\leq C\Big(\f{r}{R}\Big)^{2-n/q}\f{R^2}{\mu(B(x,R))}\int_{B(x,R)}V(y)d\mu(y)
		$$
		for all $x\in X$ and $R>r>0$.
		\item For any $x\in M$, we have
		$$
		\f{\rho(x)^2}{\mu(B(x,\rho(x)))}\int_{B(x,\rho(x))}V(y)d\mu(y)=1.
		$$
	\end{enumerate}
\end{lem}

We first prove that $L$ satisfies (B1).
\begin{prop}\label{Proposition1: heat kernel bound}
	Let $L=-\Delta+V$ be a Schr\"odinger operator with $V \in RH_q\cap A_\vc, q>\max\{1,n/2\}$. Then for each $N>0$ there exist $C$ and $c>0$ so that
	\begin{equation}\label{heatkernelboundSchrodinger}
	p_{t}(x,y)\leq \f{C}{\mu(B(x,\sqrt{t}))}\exp\Big(-\f{d(x,y)^2}{ct}\Big)\Big(1+\f{\sqrt{t}}{\rho(x)}+\f{\sqrt{t}}{\rho(y)}\Big)^{-N},
	\end{equation}
	for all $x,y \in X$ and $t>0$. Hence, $L$ satisfies (B1).
\end{prop}
\begin{proof}
	From the symmetry of the heat kernel it suffices to prove that
	\begin{equation}\label{eq1-heatkernelboundSchrodinger}
	p_{t}(x,y)\leq \f{C}{\mu(B(x,\sqrt{t}))}\exp\Big(-\f{d(x,y)^2}{ct}\Big)\Big(1+\f{\sqrt{t}}{\rho(x)}\Big)^{-N}.
	\end{equation}
	By the fact that $p_{t}(x,y)$ satisfies Gaussian upper bounds \eqref{eq-G manifold}, it suffices prove (\ref{eq1-heatkernelboundSchrodinger}) for $\rho(x)\leq \sqrt{t}$.
	
	To do this, applying Lemma \ref{Lem1: rho-manifold}  we have
	$$
	1=\rho(x)^2\fint_{B(x,\rho(x))}V(y)d\mu(y)\leq  C\Big(\f{\rho(x)}{\sqrt{t}}\Big)^{2-n/q} t\fint_{B(x,\sqrt{t})}V(y)d\mu(y)
	$$
	which implies
	$$
	t\fint_{B(x,\sqrt{t})}V(y)d\mu(y)\geq C\Big(\f{\sqrt{t}}{\rho(x)}\Big)^{2-n/q}.
	$$
	This together with Theorem \ref{Th: improved heat kernel} deduces (\ref{eq1-heatkernelboundSchrodinger}).
\end{proof}

For $t>0$ and $x,y\in X$, we set
\[
q_{t}(x,y)=\widetilde{p}_{t}(x,y)-p_{t}(x,y).
\]
We now prove that $L$ satisfies (B2). We have the following result.
\begin{prop}\label{Proposition2: heat kernel bound}
	Let $L=-\Delta+V$ be a Schr\"odinger operator with $V \in RH_q\cap A_\vc, q>\max\{1,n/2\}$. Then there exist $C$ and $c>0$ so that
	\begin{equation}\label{Difference heatkernelSchrodinger LBeltrami}
	|q_{t}(x,y)|\leq C\Big(\f{\sqrt{t}}{\sqrt{t}+\rho(x)}\Big)^{2-n/q}\f{1}{\mu(B(x,\sqrt{t}))}\exp\Big(-\f{d(x,y)^2}{ct}\Big),
	\end{equation}
	for all $x,y \in X$ and $t>0$.
\end{prop}
In order to prove Proposition \ref{Proposition2: heat kernel bound}, we need the following technical lemma.
\begin{lem}\label{Lem: V rho}
	Let $V \in RH_q\cap A_\vc$ with $q>\max\{1,n/2\}$ and let $\alpha>0$. For any $c_0>0$, there exist $C>0$ and $N_0>2-n/\sigma$ so that:
	\begin{enumerate}[{\rm (a)}]
		\item For all $x\in M$ and $\sqrt{t}\leq c_0\rho(x)$, we have
		\begin{equation}\label{eq1-LemV}
		\f{1}{\mu(B(x,\sqrt{t}))\vee\mu(B(y,\sqrt{t}))}\int_X \exp\Big(-\f{d(x,y)^2}{\alpha t}\Big)V(y)d\mu(y)\leq C t^{-1}\Big(\f{\sqrt{t}}{\rho(x)}\Big)^{2-n/q}.
		\end{equation}
		\item For all $x\in X$ and $\sqrt{t}\geq c_0\rho(x)$, we have
		\begin{equation}\label{eq2-LemV}
		\f{1}{\mu(B(x,\sqrt{t}))\vee\mu(B(y,\sqrt{t}))}\int_X \exp\Big(-\f{d(x,y)^2}{\alpha t}\Big)V(y)d\mu(y)\leq C t^{-1}\Big(\f{\sqrt{t}}{\rho(x)}\Big)^{N_0}.
		\end{equation}
	\end{enumerate}
\end{lem}
\begin{proof}
	The proof of (i) and (ii) can be done in a similar way to that in \cite[Lemma 5.1]{DZ}, and we omit details.
\end{proof}

We are ready to give the proof of Proposition \ref{Proposition2: heat kernel bound}.

\begin{proof}[Proof of Proposition \ref{Proposition2: heat kernel bound}]
	Note that from the Gaussian upper bounds of $\widetilde{p}_{t}(x,y)$ and $p_{t}(x,y)$ we have
	$$
	|q_{t}(x,y)|\leq C\f{1}{\mu(B(x,\sqrt{t}))}\exp\Big(-\f{d(x,y)^2}{ct}\Big).
	$$
	Hence it suffices to prove (\ref{Difference heatkernelSchrodinger LBeltrami}) for $\rho(x)\geq \sqrt{t}$.
	
	It is well-known that by the perturbation formula we have
	\begin{equation}\label{perturbation formula}
	\begin{aligned}
	q_{t}(x,y)&=\int_0^t\int_X \widetilde{p}_s(x,z)V(z)p_{t-s}(z,y)d\mu(z)ds\\
	&=\int_0^{t/2}\int_X\ldots+\int_{t/2}^t\int_X\ldots:= I_1+I_2.
	\end{aligned}
	\end{equation}
	We take care of $I_1$ first. Note that since  $\widetilde{p}_s(x,z)$ and $p_{t-s}(z,y)$ satisfy Gaussian upper bounds, there exists $C, c>0$ so that for $0<s\leq t/2$,
	\begin{equation}\label{eq2-Sh-LB}
	\begin{aligned}
	\widetilde{p}_s(x,z)p_{t-s}(z,y)&\leq C\f{1}{\mu(B(x,\sqrt{s}))}\exp\Big(-\f{2d(x,z)^2}{cs}\Big) \f{1}{\mu(B(y,\sqrt{t-s}))}\exp\Big(-\f{d(z,y)^2}{c(t-s)}\Big)\\
	&\lesi \f{1}{\mu(B(x,\sqrt{s}))}\exp\Big(-\f{d(x,z)^2}{cs}\Big)\exp\Big(-\f{d(x,z)^2}{ct}\Big)\f{1}{\mu(B(y,\sqrt{t}))}\exp\Big(-\f{d(z,y)^2}{ct}\Big)\\
	&\lesi \f{1}{\mu(B(x,\sqrt{s}))}\exp\Big(-\f{d(x,z)^2}{cs}\Big)\f{1}{\mu(B(y,\sqrt{t}))}\exp\Big(-\f{d(x,y)^2}{ct}\Big).
	\end{aligned}
	\end{equation}
	Inserting this into the expression of $I_2$ and using \eqref{criticalfunction}, we obtain that
	$$
	\begin{aligned}
	I_1&\lesi \f{1}{\mu(B(y,\sqrt{t}))}\exp\Big(-\f{d(x,y)^2}{ct}\Big) \int_0^{t/2}\int_M\f{1}{\mu(B(x,\sqrt{s}))}\exp\Big(-\f{d(x,z)^2}{cs}\Big)V(z)d\mu(z)\\
	&\lesi \f{1}{\mu(B(x,\sqrt{t}))}\exp\Big(-\f{d(x,y)^2}{ct}\Big)\int_0^{t/2}\Big(\f{\sqrt{s}}{\rho(x)}\Big)^{2-n/q}\f{ds}{s}\\
	&\lesi \f{1}{\mu(B(x,\sqrt{t}))}\exp\Big(-\f{d(x,y)^2}{ct}\Big)\Big(\f{\sqrt{t}}{\rho(x)}\Big)^{2-n/q}.
	\end{aligned}
	$$
	Similarly we obtain that
	$$
	I_2\lesi \f{1}{\mu(B(x,\sqrt{t}))}\exp\Big(-\f{d(x,y)^2}{ct}\Big)\Big(\f{\sqrt{t}}{\rho(y)}\Big)^{2-n/q}.
	$$
	This together with Lemma \ref{Lem1: rho-manifold} gives, for $\rho(x)\geq \sqrt{t}$,
	$$
	\begin{aligned}
	I_2&\lesi \f{1}{\mu(B(x,\sqrt{t}))}\exp\Big(-\f{d(x,y)^2}{ct}\Big)\Big(\f{\sqrt{t}}{\rho(x)}\Big)^{2-n/q}\Big(\f{\rho(x)+d(x,y)}{\rho(x)}\Big)^{k_0(2-n/q)}\\
	&\lesi \f{1}{\mu(B(x,\sqrt{t}))}\exp\Big(-\f{d(x,y)^2}{ct}\Big)\Big(\f{\sqrt{t}}{\rho(x)}\Big)^{2-n/q}\Big(\f{\sqrt{t}+d(x,y)}{\sqrt{t}}\Big)^{k_0(2-n/q)}\\
	&\lesi \f{1}{\mu(B(x,\sqrt{t}))}\exp\Big(-\f{d(x,y)^2}{c't}\Big)\Big(\f{\sqrt{t}}{\rho(x)}\Big)^{2-n/q}.
	\end{aligned}
	$$
	This completes the proof of  (\ref{Difference heatkernelSchrodinger LBeltrami}).
\end{proof}

We have the following result which shows that $L$ satisfies (B3).
\begin{prop}\label{Proposition3: heat kernel bound}
	Let $L=-\Delta+V$ be a Schr\"odinger operator with $V \in RH_q\cap A_\vc, q>\max\{1,n/2\}$. Then for any $0<\delta<\min\{\delta_1, 2-n/q\}$ there exist $C$ and $c>0$ so that
	\begin{equation}\label{Difference heatkernelSchrodinger LBeltrami}
	|q_{t}(x,y)-q_{t}(\overline{x},y)|\leq C\min\left\{\Big(\f{d(x,\overline{x})}{\rho(y)}\Big)^{\delta}, \Big(\f{d(x,\overline{x})}{\sqrt{t}}\Big)^{\delta}\right\}\f{1}{\mu(B(x,\sqrt{t}))}\exp\Big(-\f{d(x,y)^2}{ct}\Big),
	\end{equation}
	for all $t>0$,  $d(x,\overline{x})<d(x,y)/4$ and $d(x,\overline{x})<\rho(x)$.
\end{prop}
\begin{proof}
	By the perturbation formula, we have
	$$
	\begin{aligned}
	q_{t}(x,y)-q_{t}(\overline{x},y)&=\int_0^t\int_X (\widetilde{p}_s(x,z)-\widetilde{p}_s(\overline{x},z))V(z)p_{t-s}(z,y)d\mu(z)ds\\
	&=\int_0^{t/2}\int_X\ldots+\int_{t/2}^t\int_X\ldots:= I_1+I_2.
	\end{aligned}
	$$
	We now take care of $I_1$ first. To do this we write
	$$
	I_1=\int_0^{t/2}\int_{B(x,B(x,d(x,y)/2)}\ldots+\int_0^{t/2}\int_{X\backslash B(x,B(x,d(x,y)/2)}\ldots := I_{11}+I_{12}.
	$$
	Note that for $z\in B(x,d(x,y)/2)$, $d(z,y)\sim d(x,y)$. This, together with \eqref{eq2-H manifold} and the fact that $t-s\sim t$ for  $s\in (0,t/2)$ gives
	$$
	\begin{aligned}
	I_{11}&\lesi\int_0^{t/2}\int_{B(x,B(x,d(x,y)/2)}\Big(\f{d(x,\overline{x})}{\sqrt{s}}\Big)^{\delta}\Big)\f{1}{\mu(B(z,\sqrt{s}))}\Big[\exp\Big(-\f{d(x,z)^2}{cs}\Big)+\exp\Big(-\f{d(\overline{x},z)^2}{cs}\Big)\Big]V(z)\\
	& \ \ \ \ \ \times \f{1}{\mu(B(y,\sqrt{t}))}\exp\Big(-\f{d(x,y)^2}{ct}\Big)\Big(1+\f{\sqrt{t}}{\rho(y)}\Big)^{-N}d\mu(z)ds\\
	&\lesi\int_0^{\rho(x)^2}\int_{B(x,2d(x,\overline{x}))}\ldots+\int_{\rho(x)^2}^{t/2}\int_{B(x,2d(x,\overline{x}))}\ldots:= J_1+J_2,
	\end{aligned}
	$$
	where $N$ is a sufficiently large number which will be fixed later.
	
	Note that $\rho(x)\sim \rho(\overline{x})$ for $d(x,\overline{x})\leq \rho(x)$. This, together with Lemma \ref{Lem: V rho} (a) and $\delta<2-n/q$, gives
	$$
	\begin{aligned}
	J_1&\lesi \f{1}{\mu(B(y,\sqrt{t}))}\exp\Big(-\f{d(x,y)^2}{ct}\Big)\Big(1+\f{\sqrt{t}}{\rho(y)}\Big)^{-N-\delta}\int_0^{\rho(x)^2}\Big(\f{d(x,\overline{x})}{\sqrt{s}}\Big)^\delta\Big(\f{\sqrt{s}}{\rho(x)}\Big)^{2-n/q}\f{ds}{s}\\
	&\lesi \Big(\f{d(x,\overline{x})}{\rho(x)}\Big)^\delta\f{1}{\mu(B(y,\sqrt{t}))}\exp\Big(-\f{d(x,y)^2}{ct}\Big)\Big(1+\f{\sqrt{t}}{\rho(y)}\Big)^{-N-\delta}.
	\end{aligned}
	$$
	This along with Lemma \ref{Lem1: rho-manifold} implies that
	$$
	\begin{aligned}
	J_1&\lesi \Big(\f{d(x,\overline{x})}{\rho(y)}\Big)^\delta\f{1}{\mu(B(y,\sqrt{t}))}\exp\Big(-\f{d(x,y)^2}{ct}\Big)\Big(1+\f{d(x,y)}{\rho(y)}\Big)^{\delta k_0}\Big(1+\f{\sqrt{t}}{\rho(y)}\Big)^{-N-\delta}.
	\end{aligned}
	$$
	Taking $N=\delta k_0$, we have
	$$
	\Big(1+\f{d(x,y)}{\rho(y)}\Big)^{\delta k_0}\Big(1+\f{\sqrt{t}}{\rho(y)}\Big)^{-N}\lesi \Big(1+\f{d(x,y)}{\sqrt{t}}\Big)^{\delta k_0}.
	$$
	Hence,
	$$
	\begin{aligned}
	J_1&\lesi \Big(\f{d(x,\overline{x})}{\rho(y)}\Big)^\delta\f{1}{\mu(B(y,\sqrt{t}))}\exp\Big(-\f{d(x,y)^2}{ct}\Big)\Big(1+\f{d(x,y)}{\sqrt{t}}\Big)^{\delta k_0}\Big(1+\f{\sqrt{t}}{\rho(y)}\Big)^{-\delta}\\
	&\lesi \Big(\f{d(x,\overline{x})}{\rho(y)}\Big)^\delta\f{1}{\mu(B(y,\sqrt{t}))}\exp\Big(-\f{d(x,y)^2}{ct}\Big)\Big(1+\f{\sqrt{t}}{\rho(y)}\Big)^{-\delta}.
	\end{aligned}
	$$
	Similarly, by Lemma \ref{Lem: V rho} (b) and $N_0>2-n/q>\delta$, we have
	$$
	\begin{aligned}
	J_2&\lesi \f{1}{\mu(B(y,\sqrt{t}))}\exp\Big(-\f{d(x,y)^2}{ct}\Big)\Big(1+\f{\sqrt{t}}{\rho(y)}\Big)^{-N-\delta}\int_{\rho(x)^2}^{t/2}\Big(\f{d(x,\overline{x})}{\sqrt{s}}\Big)^\delta\Big(\f{\sqrt{s}}{\rho(x)}\Big)^{N_0}
	\f{ds}{s}\\
	&\lesi \f{1}{\mu(B(y,\sqrt{t}))}\exp\Big(-\f{d(x,y)^2}{ct}\Big)\Big(\f{d(x,\overline{x})}{\rho(x)}\Big)^\delta\Big(\f{\sqrt{t}}{\rho(x)}\Big)^{N_0-\delta}
	\Big(1+\f{\sqrt{t}}{\rho(y)}\Big)^{-N-\delta_0}\\
	&\lesi \f{1}{\mu(B(y,\sqrt{t}))}\exp\Big(-\f{d(x,y)^2}{ct}\Big)\Big(\f{d(x,\overline{x})}{\rho(x)}\Big)^\delta\Big(\f{\rho(y)}{\rho(x)}\Big)^{N_0-\delta}
	\Big(1+\f{\sqrt{t}}{\rho(y)}\Big)^{-N-\delta+N_0}.
	\end{aligned}
	$$
	We now take $N=N_0(k_0+1)$ and use the argument above to obtain that
	$$
	J_2\lesi \Big(\f{d(x,\overline{x})}{\rho(y)}\Big)^\delta\f{1}{\mu(B(y,\sqrt{t}))}\exp\Big(-\f{d(x,y)^2}{ct}\Big)\Big(1+\f{\sqrt{t}}{\rho(y)}\Big)^{-\delta}.
	$$
	Arguing similarly we obtain
	$$
	I_{12}\lesi \Big(\f{d(x,\overline{x})}{\rho(y)}\Big)^\delta\f{1}{\mu(B(y,\sqrt{t}))}\exp\Big(-\f{d(x,y)^2}{ct}\Big)\Big(1+\f{\sqrt{t}}{\rho(y)}\Big)^{-\delta}.
	$$
	Taking estimates $J_1, J_2$ and $I_{12}$ into account we conclude that
	\[
	\begin{aligned}
	I_1&\lesi \Big(\f{d(x,\overline{x})}{\rho(y)}\Big)^\delta\f{1}{\mu(B(y,\sqrt{t}))}\exp\Big(-\f{d(x,y)^2}{ct}\Big)\Big(1+\f{\sqrt{t}}{\rho(y)}\Big)^{-\delta}\\
	&\lesi \min\left\{\Big(\f{d(x,\overline{x})}{\sqrt{t}}\Big)^\delta,\Big(\f{d(x,\overline{x})}{\rho(y)}\Big)^\delta\right\}\f{1}{\mu(B(y,\sqrt{t}))}\exp\Big(-\f{d(x,y)^2}{ct}\Big).
	\end{aligned}
	\]
	
	We turn to the term $I_2$. By a change of variable we can rewrite
	$$
	I_2=\int_{0}^{t/2}\int_X (\widetilde{p}_{t-s}(x,z)-\widetilde{p}_{t-s}(\overline{x},z))V(z)p_{s}(z,y)d\mu(z)ds.
	$$
	Using (\ref{eq2-H manifold}), Proposition \ref{Proposition1: heat kernel bound} and the fact that $t-s\sim t$ for $s\in (0,t/2]$, we obtain 
	$$
	\begin{aligned}
	I_{2}&\lesi\int_0^{t/2}\int_{X}\Big(\f{d(x,\overline{x})}{\sqrt{t}}\Big)^{\delta}\f{1}{\mu(B(z,\sqrt{t}))}\exp\Big(-\f{d(x,z)^2}{ct}\Big)V(z)\\
	& \ \ \times \f{1}{\mu(B(y,\sqrt{s}))}\exp\Big(-\f{d(z,y)^2}{cs}\Big)\Big(1+\f{\sqrt{s}}{\rho(y)}\Big)^{-N}d\mu(z)ds\\
	&\ +\int_0^{t/2}\int_{X}\Big(\f{d(x,\overline{x})}{\sqrt{t}}\Big)^\delta\f{1}{\mu(B(z,\sqrt{t}))}\exp\Big(-\f{d(\overline{x},z)^2}{ct}\Big)V(z)\\
	& \ \ \times \f{1}{\mu(B(y,\sqrt{s}))}\exp\Big(-\f{d(z,y)^2}{cs}\Big)\Big(1+\f{\sqrt{s}}{\rho(y)}\Big)^{-N}d\mu(z)ds\\
	&=I_{21}+I_{22}.
	\end{aligned}
	$$
	Note that for $s\in (0,t/2]$ we have
	$$
	\exp\Big(-\f{d(x,z)^2}{ct}\Big)\exp\Big(-\f{d(z,y)^2}{cs}\Big)\lesi \exp\Big(-\f{d(x,y)^2}{c't}\Big)\exp\Big(-\f{d(z,y)^2}{c''s}\Big).
	$$
	Inserting this into the expression of $I_{21}$ we obtain, for $N>N_0$,
	\[
	\begin{aligned}
	I_{21}& \lesi\Big(\f{d(x,\overline{x})}{\sqrt{t}}\Big)^\delta\f{1}{\mu(B(x,\sqrt{t}))}\exp\Big(-\f{d(x,y)^2}{c't}\Big)\\
	& \ \ \times \int_0^{t/2}\int_{X}V(z) \f{1}{\mu(B(y,\sqrt{s}))}\exp\Big(-\f{d(z,y)^2}{c''s}\Big)\Big(1+\f{\sqrt{s}}{\rho(y)}\Big)^{-N}d\mu(z)ds.
	\end{aligned}
	\]
	If $t/2>\rho(y)$, then by Lemma \ref{Lem: V rho} we have
	\[
	\begin{aligned}
      \int_0^{t/2}\int_{X}V(z) &\f{1}{\mu(B(y,\sqrt{s}))}\exp\Big(-\f{d(z,y)^2}{c''s}\Big)\Big(1+\f{\sqrt{s}}{\rho(y)}\Big)^{-N}d\mu(z)ds\\
      &\lesi \int_0^{\rho(y)^2}\Big(\f{\sqrt{s}}{\rho(y)}\Big)^{2-n/q}\f{ds}{s}+ \int_{\rho(y)^2}^\vc\Big(\f{\sqrt{s}}{\rho(y)}\Big)^{N_0}\Big(\f{\sqrt{s}}{\rho(y)}\Big)^{-N}\f{ds}{s}\\
      &\lesi 1.
    \end{aligned}
    \]
    Hence,
    \[
   \begin{aligned}
    I_{21}&
    \lesi\Big(\f{d(x,\overline{x})}{\sqrt{t}}\Big)^\delta\f{1}{\mu(B(x,\sqrt{t}))}\exp\Big(-\f{d(x,y)^2}{c't}\Big)\\
    &
    \lesi\min\left\{\Big(\f{d(x,\overline{x})}{\sqrt{t}}\Big)^\delta,\Big(\f{d(x,\overline{x})}{\rho(y)}\Big)^\delta\right\}\f{1}{\mu(B(x,\sqrt{t}))}\exp\Big(-\f{d(x,y)^2}{c't}\Big).
    \end{aligned}
    \]
    If $t/2<\rho(y)$, then by Lemma \ref{Lem: V rho} (a) we obtain 	
	\[
	\begin{aligned}
	I_{21}&\lesi \Big(\f{d(x,\overline{x})}{\sqrt{t}}\Big)^\delta\f{1}{\mu(B(x,\sqrt{t}))}\exp\Big(-\f{d(x,y)^2}{c't}\Big)\int_0^{t/2}\Big(\f{\sqrt{s}}{\rho(y)}\Big)^{2-n/q}\Big(\f{\sqrt{s}}{\rho(y)}\Big)^{2-n/q-\delta}\f{ds}{s}\\
	&\lesi \Big(\f{d(x,\overline{x})}{\sqrt{t}}\Big)^\delta\f{1}{\mu(B(x,\sqrt{t}))}\exp\Big(-\f{d(x,y)^2}{ct}\Big)\Big(\f{\sqrt{t}}{\rho(y)}\Big)^\delta\\
	&\lesi \f{1}{\mu(B(x,\sqrt{t}))}\exp\Big(-\f{d(x,y)^2}{ct}\Big)\Big(\f{d(x,\overline{x})}{\rho(y)}\Big)^\delta\\
	 &
	 \lesi\min\left\{\Big(\f{d(x,\overline{x})}{\sqrt{t}}\Big)^\delta,\Big(\f{d(x,\overline{x})}{\rho(y)}\Big)^\delta\right\}\f{1}{\mu(B(x,\sqrt{t}))}\exp\Big(-\f{d(x,y)^2}{c't}\Big).
	\end{aligned}
	\]	
	By a  similar argument, we also have
	$$
	\begin{aligned}
	I_{22}&\lesi \f{1}{\mu(B(\overline{x},\sqrt{t}))}\exp\Big(-\f{d(\overline{x},y)^2}{ct}\Big)\min\left\{\Big(\f{d(x,\overline{x})}{\sqrt{t}}\Big)^\delta,\Big(\f{d(x,\overline{x})}{\rho(y)}\Big)^\delta\right\}\\
	&\lesi \f{1}{\mu(B(y,\sqrt{t}))}\exp\Big(-\f{d(\overline{x},y)^2}{ct}\Big)\min\left\{\Big(\f{d(x,\overline{x})}{\sqrt{t}}\Big)^\delta,\Big(\f{d(x,\overline{x})}{\rho(y)}\Big)^\delta\right\}\\
	&\lesi \f{1}{\mu(B(y,\sqrt{t}))}\exp\Big(-\f{d(x,y)^2}{ct}\Big)\Big(\f{d(x,\overline{x})}{\rho(y)}\Big)^\delta.
	\end{aligned}
	$$
	This completes our proof.
\end{proof}

\bigskip

{\bf Acknowledgement.} T. A. Bui and X. T. Duong were supported by the research grant ARC DP140100649 from the Australian Research Council. The authors would like to thank the referee for useful comments and suggestions.

\end{document}